\documentclass[11pt, reqno]{amsart}
\usepackage{amssymb}
\usepackage{eucal}
\usepackage{amsmath}
\usepackage{amscd}
\usepackage[dvips]{color}
\usepackage{multicol}
\usepackage[all]{xy}           
\usepackage{graphicx}
\usepackage{color}
\usepackage{colordvi}
\usepackage{xspace}
\usepackage{tikz}
\usepackage{ulem}

\usepackage{ifpdf}
\ifpdf
 \usepackage[colorlinks,final,hyperindex]{hyperref}
\else
 \usepackage[colorlinks,final,backref=page,hyperindex,hypertex]{hyperref}
\fi

\begin{document}
\newcommand {\emptycomment}[1]{} 

\newcommand{\nc}{\newcommand}
\newcommand{\delete}[1]{}
\nc{\mfootnote}[1]{\footnote{#1}} 
\nc{\todo}[1]{\tred{To do:} #1}

\delete{
\nc{\mlabel}[1]{\label{#1}}  
\nc{\mcite}[1]{\cite{#1}}  
\nc{\mref}[1]{\ref{#1}}  
\nc{\meqref}[1]{\eqref{#1}} 
\nc{\bibitem}[1]{\bibitem{#1}} 
}

\nc{\mlabel}[1]{\label{#1}  
{\hfill \hspace{1cm}{\bf{{\ }\hfill(#1)}}}}
\nc{\mcite}[1]{\cite{#1}{{\bf{{\ }(#1)}}}}  
\nc{\mref}[1]{\ref{#1}{{\bf{{\ }(#1)}}}}  
\nc{\meqref}[1]{\eqref{#1}{{\bf{{\ }(#1)}}}} 
\nc{\mbibitem}[1]{\bibitem[\bf #1]{#1}} 


\nc{\tred}[1]{\textcolor{red}{#1}}
\nc{\tblue}[1]{\textcolor{blue}{#1}}
\nc{\tgreen}[1]{\textcolor{green}{#1}}
\nc{\tpurple}[1]{\textcolor{purple}{#1}}
\nc{\btred}[1]{\textcolor{red}{\bf #1}}
\nc{\btblue}[1]{\textcolor{blue}{\bf #1}}
\nc{\btgreen}[1]{\textcolor{green}{\bf #1}}
\nc{\btpurple}[1]{\textcolor{purple}{\bf #1}}

\nc{\cm}[1]{\textcolor{red}{Chengming:#1}}
\nc{\gl}[1]{\textcolor{blue}{Guilai:#1}}

\nc{\Witt}{invariant-c2c\xspace} 

\newtheorem{thm}{Theorem}[section]
\newtheorem{lem}[thm]{Lemma}
\newtheorem{cor}[thm]{Corollary}
\newtheorem{pro}[thm]{Proposition}
\newtheorem{conj}[thm]{Conjecture}
\theoremstyle{definition}
\newtheorem{defi}[thm]{Definition}
\newtheorem{ex}[thm]{Example}
\newtheorem{rmk}[thm]{Remark}
\newtheorem{pdef}[thm]{Proposition-Definition}
\newtheorem{condition}[thm]{Condition}



\title[A new approach to the bialgebra theory for relative Poisson algebras]{A new approach to the bialgebra theory for relative Poisson algebras}

\author{Guilai Liu}
\address{Chern Institute of Mathematics \& LPMC, Nankai University, Tianjin 300071, China}
\email{liugl@mail.nankai.edu.cn}

\author{Chengming Bai}
\address{Chern Institute of Mathematics \& LPMC, Nankai University, Tianjin 300071, China }
\email{baicm@nankai.edu.cn}


\begin{abstract}

It is natural to consider  extending the typical construction of
relative Poisson algebras from commutative differential algebras
to the context of bialgebras.
 The known bialgebra structures for relative Poisson
    algebras, namely relative Poisson
    bialgebras, are equivalent to Manin triples of relative Poisson
    algebras with respect to the symmetric bilinear forms which are invariant
    on both the commutative associative and Lie algebras. However, they are not
    consistent with commutative and cocommutative differential
    antisymmetric infinitesimal (ASI) bialgebras as the bialgebra
    structures for commutative differential algebras.
 Alternatively, with the invariance
    replaced by the commutative $2$-cocycles on the Lie algebras, the
    corresponding Manin triples of relative Poisson algebras are proposed, which are shown to be
    equivalent to certain bialgebra structures, namely relative PCA
    bialgebras. They serve as another approach to the bialgebra theory for
    relative Poisson algebras, which can be naturally constructed from
    commutative and cocommutative differential ASI bialgebras.
 The notion of the relative PCA Yang-Baxter equation
(RPCA-YBE) in a relative PCA algebra is introduced, whose
antisymmetric solutions give coboundary relative PCA bialgebras.
The notions of $\mathcal{O}$-operators of relative PCA algebras
and relative pre-PCA algebras are also introduced to give
antisymmetric solutions of the RPCA-YBE.

\end{abstract}

\subjclass[2020]{
17A36,  
17A40,  
17B38, 
17B63,  
17D25,  
37J39,   
53D17
}

\keywords{Relative Poisson algebra; commutative $2$-cocycle; Manin triple; bialgebra; classical Yang-Baxter equation;
$\mathcal{O}$-operator}

\maketitle


\tableofcontents

\allowdisplaybreaks

\section{Introduction}

Poisson algebras originate from the study of Poisson geometry
 \cite{BV1,Li77,Wei77}, which are closely related to a lot of
topics in mathematics and physics.
Recently, there is a generalization of Poisson algebras, namely a relative Poisson algebra that we shall study in this paper, where the Leibniz rule is presented in a modified type.
\begin{defi}
A \textbf{relative Poisson algebra} is a quadruple $(A,\cdot,
[-,-],P)$, where $(A,\cdot)$ is a commutative associative algebra,
$(A,[-,-])$ is a Lie algebra, $P$ is a derivation of both
$(A,\cdot)$ and $(A,[-,-])$, and the following relative Leibniz
rule is satisfied:
\begin{equation}\label{eq:GLR}
[z,x\cdot y]=[z,x]\cdot y+x\cdot[z,y]+x\cdot y\cdot P(z),\;\;\forall
x,y,z\in A.
\end{equation}
\end{defi}
Relative Poisson algebras first appeared as a graded form
\cite{Can} in the classification of simple Jordan superalgebras,
named Kantor series \cite{Kan,Kin}. There is also a one-to-one
correspondence between unital relative Poisson algebras in the
sense that the commutative associative algebras have units  and
Jacobi algebras \cite{Ago}, where the latter correspond to Jacobi
manifolds \cite{Kir, Lic}.


A (commutative) associative algebra $(A,\cdot)$ with  a derivation
$P$ is called  a {\bf (commutative) differential algebra}, which
is denoted by $(A,\cdot, P)$ or simply $(A,P)$ without confusion.
In this case, there is a Lie algebra structure on $A$ given by
\begin{equation}\label{eq:Witt Lie}
    [x,y]=x\cdot P(y)-P(x)\cdot y,\;\forall x,y\in A,
\end{equation}
which is called the \textbf{Witt type Lie algebra of $(A,\cdot, P)$} \cite{SXZ,XXu}. 
There is the following construction of relative Poisson algebras.

\begin{pro}\cite{RPA}\label{ex1.2}
Let $(A,\cdot, P)$ be a commutative differential algebra and
$(A,[-,-])$ be the corresponding Witt type Lie algebra defined by
\eqref{eq:Witt Lie}. Then $(A,\cdot,[-,-],P)$ is a relative
Poisson algebra.
\end{pro}

A bialgebra structure is a vector space equipped with both an
algebra structure and a coalgebra structure satisfying some
compatible conditions. Among the bialgebra
structures,  Lie bialgebras \cite{CP1,Dri} are closely related to
Poisson-Lie groups and play an important role in the
infinitesimalization of quantum groups. Another well-known bialgebra structures are antisymmetric
infinitesimal (ASI) bialgebras \cite{Agu2000, Agu2001, Agu2004,
Bai2010}  which are special infinitesimal bialgebras introduced by
Joni and Rota in order to provide an algebraic framework for the
calculus of divided differences \cite{JR}.
 These bialgebras have a common property that they are equivalently characterized in terms of the Manin triples
with respect to the (symmetric or antisymmetric) nondegenerate
bilinear forms satisfying certain conditions. For example, a Lie
bialgebra is equivalent to a Manin triple of Lie algebras with
respect to the symmetric invariant bilinear form. A
(commutative and cocommutative) ASI bialgebra is equivalent to a
double construction of a (commutative) Frobenius algebra which is
the associative algebra version of a Manin triple with respect to
the symmetric invariant bilinear form.

\delete{ More precisely, the property of Manin triples decides
different types of bialgebras. For example, a Manin triple of Lie
algebras with respect to the symmetric invariant bilinear form is
equivalent to a Lie bialgebra, whereas a Manin triple of Lie
algebras with respect to the symplectic form is equivalent to a
left-symmetric bialgebra [ref].}

 \delete{ , where there is a relative
Poisson algebra structure $(A\oplus A^{*},\cdot,[-,-],P+Q^{*})$ on
the direct sum of vector spaces $A\oplus A^{*}$ containing
$(A,\cdot_{A},[-,-]_{A},P)$ and
$(A^{*},\cdot_{A^{*}},[-,-]_{A^{*}},Q^{*})$ as relative Poisson
subalgebras,  and the natural nondegenerate symmetric bilinear
form $\mathcal{B}_{d}$ given by
\begin{equation*}
    \mathcal{B}_{d}(x+a^{*},y+b^{*})=\langle x, b^{*}\rangle+\langle a^{*},y\rangle,\;\forall x,y\in A, a^{*},b^{*}\in A^{*}
\end{equation*}
is invariant on both the commutative associative algebra $(A\oplus A^{*},\cdot)$ and the Lie algebra $(A\oplus A^{*},[-,-])$.}

It is natural to consider  extending the typical construction of
relative Poisson algebras from commutative differential algebras
given in Proposition~\ref{ex1.2} to the context of bialgebras. On
the one hand, the notion of a differential ASI bialgebra
\cite{LLB2023} was introduced as an equivalent structure of a
double construction of a differential Frobenius algebra $\big((A\oplus A^{*},P+Q^*), (A,P), (A^{*},Q^*)\big)$. On the other hand, in order to
construct Frobenius Jacobi algebras which reflect certain
Frobenius property, the notion of relative Poisson bialgebras was
introduced with a systematic study \cite{RPA}, generalizing the
notion of Poisson bialgebras given in \cite{NB}. In fact, a
(relative) Poisson bialgebra is equivalent to a Manin triple of
(relative) Poisson algebras with respect to the symmetric bilinear
form which is invariant on both the commutative associative and
Lie
algebras. 

Unfortunately, as the ``bialgebra version" of commutative
differential algebras, commutative and cocommutative differential
ASI bialgebras are not consistent with relative Poisson
bialgebras. Equivalently, a double construction of a commutative
differential Frobenius algebra is not consistent with a Manin
triple of relative Poisson algebras with respect to the symmetric
bilinear form which is invariant on both the commutative
associative and Lie algebras. Such a fact can be seen as follows.
Let $(A,\cdot,P)$ be a commutative differential algebra and
$\mathcal{B}$ be a nondegenerate symmetric bilinear form on $A$.
Suppose that $\mathcal{B}$
 is {\bf invariant} on $(A,\cdot)$ in the sense that $$\mathcal{B}(x\cdot
y,z)=\mathcal{B}(x,y\cdot z),\;\;\forall x,y,z\in A.$$
 If
$\mathcal{B}$ is also {\bf invariant} on the corresponding Witt type
algebra $(A,[-,-])$ defined by \eqref{eq:Witt Lie} in the sense
that $$\mathcal{B}([x,y],z)=\mathcal{B}(x,[y,z]),\;\; \forall
x,y,z\in
    A,$$
     then $(A,[-,-])$ is abelian, that is, the
    additional (strong) identity $P(x)\cdot y=x\cdot P(y)$ for all $x,y\in
    A$ holds. Consequently a commutative and cocommutative differential  ASI bialgebra  usually fails to give
a relative Poisson bialgebra.

Therefore one should ``adjust" the bialgebra theory for relative
Poisson algebras in order to be consistent with commutative and
cocommutative differential ASI bialgebras. With the same
assumption, note that by \cite{LB2022}, $\mathcal{B}$ is  a
{\bf commutative $2$-cocycle} \cite{Dzh} on the Lie algebra $(A,[-,-])$
in the sense that
\begin{equation}\label{eq:2-coc}
\mathcal{B}([x,y],z)+\mathcal{B}([y,z],x)+\mathcal{B}([z,x],y)=0,\;\;\forall
x,y,z\in A.
\end{equation}
Hence it is reasonable to consider  replacing the
condition of ``invariance" by the condition of being a commutative
$2$-cocycle on the Lie algebra. That is,  alternatively, we consider the Manin triples of relative
Poisson algebras with respect to the symmetric bilinear forms
which are invariant on the commutative associative algebras and
commutative $2$-cocycles on the Lie algebras. Accordingly, there
are the corresponding bialgebra structures  as a new approach to
the bialgebra theory for relative Poisson algebras. Moreover, the
new bialgebra structures can be obtained from commutative and
cocommutative differential ASI bialgebras as expected, giving the
bialgebra context of Proposition~\ref{ex1.2}.

Explicitly, as the first step, the notion of  relative PCA
algebras is introduced to interpret the relative Poisson algebras
equipped with nondegenerate symmetric  bilinear forms which are
invariant on the commutative associative algebras and commutative
$2$-cocycles on the Lie algebras. That is, there naturally exist
relative PCA algebra structures on the latter and conversely,
there is also a natural construction of the latter from the double
spaces of arbitrary relative PCA algebras. Moreover, a relative
PCA algebra is characterized as a relative Poisson algebra  with a
specific representation  in which an anti-pre-Lie algebra
introduced in \cite{LB2022} and a linear operator are involved. It
is noteworthy that there are two linear operators $P$ and $Q$ in
the definition of a relative PCA algebra, where $P$ is the
original derivation in the definition of the relative Poisson
algebra, and $Q$ is involved in the representation of the relative
Poisson algebra.

Then we introduce the notion  of a Manin triple of relative
Poisson algebras with respect to the symmetric bilinear form which
is invariant on the commutative associative algebra and a
commutative $2$-cocycle on the Lie algebra, and the notion of a
relative PCA bialgebra. Such two structures are shown to be
equivalent, that is, relative PCA bialgebras give a new approach
to the bialgebra theory for relative Poisson algebras. Moreover,
there is a natural construction of relative PCA bialgebras from
commutative and cocommutative differential ASI bialgebras,
extending Proposition~\ref{ex1.2} to the context of bialgebras.
The study of coboundary cases leads to the notion of the relative
PCA Yang-Baxter equation (RPCA-YBE) in a relative PCA algebra,
whose antisymmetric solutions in turn give rise to coboundary
relative PCA bialgebras. The notions of $\mathcal{O}$-operators of
relative PCA algebras and relative pre-PCA algebras are also
introduced to give antisymmetric solutions of the RPCA-YBE.

 The paper is organized as follows.
In Section \ref{sec2}, we introduce the notion of relative PCA
algebras and study their relationships with relative Poisson
algebras. \delete{In particular, relative PCA algebras serve as
the underlying algebra structures of relative Poisson algebras
with nondegenerate symmetric bilinear forms that are invariant on
the commutative associative algebras and commutative $2$-cocycles
on the Lie algebras.}Then we study representations and matched
pairs of relative PCA algebras. In Section \ref{sec3}, we
introduce the notion of a Manin triple of relative Poisson
algebras with respect to the symmetric bilinear form which is
invariant on the commutative associative algebra and a commutative
$2$-cocycle on the Lie algebra, which is shown to be equivalent to
a relative PCA bialgebra through specific matched pairs of
relative Poisson algebras. Moreover, relative PCA bialgebras are
naturally obtained from commutative and cocommutative differential ASI
bialgebras. In Section \ref{sec4}, we show that antisymmetric
solutions of the RPCA-YBE give rise to coboundary relative PCA
bialgebras. Such solutions are interpreted in terms of
$\mathcal{O}$-operators of relative PCA algebras and can be
obtained from relative pre-PCA algebras.

Throughout this paper,  unless otherwise specified, all the vector
spaces and algebras are finite-dimensional over an algebraically
closed field $\mathbb {K}$ of characteristic zero, although many
results and notions remain valid in the infinite-dimensional case.
For a vector space $A$ with a binary operation $\circ$, the linear maps
${\mathcal L}_\circ, {\mathcal R}_\circ: A\rightarrow {\rm
End}_{\mathbb K}(A)$ are respectively defined by
$${\mathcal L}_\circ(x)y:=x\circ y,\;\;{\mathcal R}_\circ(x)y:=y\circ x,\;\; \forall x,y\in A.
$$
In particular, when $(A,[-,-])$ is a Lie algebra, we let ${\rm ad}:
A\rightarrow {\rm End}_{\mathbb K}(A)$ with ${\rm ad}
(x)(y)=[x,y]$ for all $x,y\in A$ be the adjoint representation.

\section{Relative PCA algebras}\label{sec2}
In this section, we introduce the notion of a relative PCA
algebra, which generalizes the notion of a PCA algebra introduced
in \cite{NSP}. It can be  characterized as a relative Poisson
algebra with a specific representation, in which an anti-pre-Lie
algebra and a linear operator are involved. There is a
construction of relative PCA algebras from admissible commutative
differential algebras. We show that a relative Poisson algebra
with a nondegenerate symmetric  bilinear form which is invariant
on the commutative associative algebra and a commutative
    $2$-cocycle on the Lie algebra induces a relative PCA algebra. Conversely, there is a construction of the former from the double
space of an arbitrary relative PCA algebra. We also study
representations and matched pairs of relative PCA algebras.

\subsection{Relative PCA algebras and their relationships with relative Poisson algebras}\

We first recall the notion of representations of relative Poisson algebras.

\begin{defi}\label{de:rep}\cite{RPA}
Let $(A,\cdot,[-,-],P)$ be a relative Poisson algebra, $V$ be a
vector space and $\mu,\rho:A\rightarrow\mathrm{End}_{\mathbb K}
(V),\alpha:V\rightarrow V$ be linear maps. We say the quadruple
$(\mu,\rho,\alpha,V)$ is a \textbf{representation of
$(A,\cdot,[-,-],P)$} if the following conditions hold:
\begin{enumerate}
    \item $(\mu,V)$ is a representation of the commutative associative
algebra $(A,\cdot)$, that is,
\begin{equation*}
    \mu(x\cdot y)v=\mu(x)\mu(y)v,\;\;\forall x,y\in A;
\end{equation*}
\item $(\rho,V)$ is a representation of the Lie algebra $(A,[-,-])$, that is,
\begin{equation*}
    \rho([x,y])v=\rho(x)\rho(y)v-\rho(y)\rho(x)v,\;\;\forall x,y\in A;
\end{equation*}
\item the following equations hold:
\begin{eqnarray}
    \alpha\big(\mu(x)v\big)-\mu\big(P(x)\big)v-\mu(x)\alpha(v)&=&0,\label{eq:repRPA1}\\
    \alpha\big(\rho(x)v\big)-\rho\big(P(x)\big)v-\rho(x)\alpha(v)&=&0,\label{eq:repRPA2}\\
    \rho(y)\mu(x)v-\mu(x)\rho(y)v+\mu([x,y])v-\mu\big(x\cdot
P(y)\big)v&=&0,\label{eq:repRPA3}\\
    \rho(x\cdot y)v-\mu(x)\rho(y)v-\mu(y)\rho(x)v+\mu(x\cdot y)\alpha(v)&=&0,\;\;\forall x,y\in A,v\in V.\label{eq:repRPA4}
\end{eqnarray}
\end{enumerate}
\end{defi}

By taking $P=0,\alpha=0$, we recover the notion of a
representation of a Poisson algebra \cite{NB}.

\begin{pro} {\rm \cite{RPA}}\label{pro:semidirect}
Let  $(A,\cdot,[-,-],P)$ be a relative Poisson algebra, $V$ be a
vector space and  $\mu,\rho:A\rightarrow\mathrm{End}_{\mathbb
K} (V),\alpha:V\rightarrow V$ be linear maps. Then the quadruple
$(\mu,\rho,\alpha,V)$ is a  representation of  $(A,\cdot,[-,-],P)$
if and only if there is a relative Poisson algebra structure
$(A\oplus V,\cdot,[-,-],P+\alpha)$ on the direct sum $A\oplus V$
of vector spaces, where the binary operations are respectively given by
\begin{eqnarray}
\ (x+u)\cdot(y+v)&=&x\cdot y+\mu(x)v+\mu(y)u,\label{eq:sd,asso}\\
\ [x+u,y+v]&=&[x,y]+\rho(x)v-\rho(y)u,\;\;\forall x,y\in A,u,v\in V.
\end{eqnarray}
In this case, the relative Poisson algebra structure on $A\oplus V$ is denoted  by $(A\ltimes_{\mu,\rho}V,P+\alpha)$ and is called the {\bf semi-direct product relative Poisson algebra}.
\end{pro}

\begin{ex}
    Let $(A,\cdot,[-,-],P)$ be a relative Poisson algebra. Then $({\mathcal L}_{\cdot},\mathrm{ad},P,A)$ is a representation, 
    which is called the {\bf adjoint representation of $(A,\cdot,[-,-],P)$}.
\end{ex}

Let $A$ and $V$ be vector spaces and
$f:A\rightarrow\mathrm{End}_{\mathbb K} (V)$ be a linear map. We
set a linear map $f^{*}:A\rightarrow\mathrm{End}_{\mathbb K}
(V^{*})$ by
\begin{equation}
    \langle f^{*}(x)u^{*},v\rangle=-\langle u^{*},f(x)v\rangle,\;\;\forall x\in A, u^{*}\in V^{*},v\in
    V,
\end{equation}
where $\langle, \rangle$ is the ordinary pairing between $V^*$ and
$V$.

\begin{pro}{\rm \cite{RPA}}\label{pro:dual} Let $(\mu,\rho,\alpha,V)$ be a representation of a relative Poisson algebra $(A,\cdot$, $[-,-]$, $P)$. Then $(-\mu^{*},\rho^{*},-\alpha^{*},V^{*})$ is a representation of $(A,\cdot$, $[-,-]$, $P)$.
\end{pro}

We recall the notion of an anti-pre-Lie algebra.

\begin{defi}\cite{LB2022}
    An {\bf anti-pre-Lie algebra} is a pair $(A,\circ)$, such that $A$ is a vector space with a binary operation  $\circ:A\otimes A\rightarrow A$,
    and the following equations hold:
    \begin{eqnarray}
       && x\circ(y\circ z)-y\circ(x\circ z)=[y,x]\circ z,\label{eq:apl1}\\
       && [x,y]\circ z+[y,z]\circ x+[z,x]\circ y=0,\;\forall x,y,z\in A,\label{eq:apl2}
    \end{eqnarray}
    where
    \begin{equation}\label{eq:sub-adj}
        [x,y]=x\circ y-y\circ x,\;\;\forall x,y\in A.
    \end{equation}
\end{defi}

Let $A$ be a vector space together with a binary operation
$\circ:A\otimes A\rightarrow A$. The pair $(A,\circ)$ is called a
{\bf Lie-admissible algebra} if $[-,-]$ given by
(\ref{eq:sub-adj}) defines a Lie algebra and in this case,
$(A,[-,-])$ is called the {\bf sub-adjacent Lie algebra} of
$(A,\circ)$. In particular,
the pair $(A,\circ)$ is an anti-pre-Lie algebra if and only if $(A,\circ)$ is a Lie-admissible algebra and  $(-\mathcal{L}_{\circ},A)$ is a representation of the sub-adjacent Lie algebra $(A,[-,-])$.

\begin{defi}\cite{NSP}
Let $(A,\cdot)$ be a commutative associative algebra and $(A,\circ)$ be an anti-pre-Lie algebra. Suppose that $(A,[-,-])$ is the sub-adjacent Lie algebra of
$(A,\circ)$. If $(A,\cdot,[-,-])$ is a Poisson algebra and the following equations hold:
\begin{eqnarray}
       (x\cdot y)\circ z&=&x\cdot(y\circ z)+y\cdot(x\circ z),\label{eq:PCA1}\\
            (x\circ y-y\circ x)\cdot z&=&y\cdot(x\circ z)-x\circ(y\cdot z),\;\forall x,y,z\in A,\label{eq:PCA2}
            \end{eqnarray}
then we say $(A,\cdot,\circ)$ is a  {\bf PCA algebra}.
\end{defi}

We have the following characterization of  PCA algebras.

\begin{pro}\cite{NSP}
Let $(A,\cdot,[-,-])$ be a Poisson algebra and $(A,\circ)$ be an anti-pre-Lie algebra whose sub-adjacent Lie algebra is $(A,[-,-])$. Then $(A,\cdot,\circ)$ is a PCA algebra if and only if
 $(\mathcal{L}_{\cdot},-\mathcal{L}_{\circ},A)$  {\rm(}or equivalently, $(-\mathcal{L}^{*}_{\cdot},-\mathcal{L}^{*}_{\circ},A^{*})${\rm)} is a representation of $(A,\cdot,[-,-])$.
\end{pro}

\begin{rmk}
The notion of a PCA algebra is introduced  by the initials of the
Poisson algebra, commutative associative algebra and anti-pre-Lie
algebra, in order to exhibit that the Poisson algebra admits
representations $(\mathcal{L}_{\cdot},-\mathcal{L}_{\circ},A)$ as
well as
$(-\mathcal{L}^{*}_{\cdot},-\mathcal{L}^{*}_{\circ},A^{*})$ from a
commutative associative algebra and an anti-pre-Lie algebra. More
precisely, in \cite{NSP} the original definition of a PCA algebra
is a commutative associative algebra $(A,\cdot)$ together with an
anti-pre-Lie algebra $(A,\circ)$ such that \eqref{eq:PCA1},
\eqref{eq:PCA2} and the following equation hold:
\begin{equation}\label{eq:PCA3}
    z\circ(x\cdot y)+z\cdot(x\circ y)+z\cdot (y\circ x)-2(x\cdot y)\circ z=0,\;\forall x,y,z\in A.
\end{equation}
In fact, if \eqref{eq:PCA1} and
\eqref{eq:PCA2} hold, then \eqref{eq:PCA3} holds if and only if
$(A,\cdot,[-,-])$ is a Poisson algebra. The PCA algebra is also
understood as giving a new splitting of a Poisson algebra besides
the ordinary splitting in the sense \cite{BBGN} given by a
pre-Poisson algebra \cite{Agu2000.0}.
\end{rmk}

\begin{defi}
Let $(A,\cdot)$ be a commutative associative algebra and $(A,\circ)$ be an anti-pre-Lie algebra. Suppose that $(A,[-,-])$ is the sub-adjacent Lie algebra of $(A,\circ)$ and $P,Q:A\rightarrow A$ are linear maps. If $(A,\cdot,[-,-],P)$ is a relative Poisson algebra and the following equations hold:
        \begin{eqnarray}
            x\cdot Q(y)-P(x)\cdot y-Q(x\cdot y)&=&0,\label{eq:rps1}\\
            x\circ Q(y)-P(x)\circ y-Q(x\circ y)&=&0,\label{eq:rps2}\\
            x\cdot(y\circ z)-y\circ(x\cdot z)+[x,y]\cdot z-x\cdot P(y)\cdot z&=&0,\label{eq:rps3}\\
            (x\cdot y)\circ z-y\circ(x\cdot z)-x\circ(y\cdot z)+Q(x\cdot y\cdot z)&=&0,\;\;\forall x,y,z\in A,\label{eq:rps4}
        \end{eqnarray}
 then we say $(A,\cdot,\circ,P,Q)$ is a {\bf relative PCA
 algebra}. We also say $(A,\cdot,[-,-],P)$ is the {\bf associated
 relative Poisson algebra of $(A,\cdot,\circ,P,Q)$}.
\end{defi}

Let $(A,\cdot,\circ)$ be a PCA algebra. Then by \eqref{eq:PCA1} and \eqref{eq:PCA2}, we have
\begin{eqnarray*}
            x\cdot(y\circ z)-y\circ(x\cdot z)+[x,y]\cdot z&=&0,\\
            (x\cdot y)\circ z-y\circ(x\cdot z)-x\circ(y\cdot z)&=&0,\;\;\forall x,y,z\in A.
\end{eqnarray*}
Hence a PCA algebra is a relative PCA algebra by taking $P=Q=0$.
Moreover,
we have the following characterization of relative PCA algebras.

\begin{pro}\label{pro2.5}
    Let $(A,\cdot,[-,-],P)$ be  a relative Poisson algebra, $(A,\circ)$ be an anti-pre-Lie algebra whose sub-adjacent Lie algebra is $(A,[-,-])$ and $Q:A\rightarrow A$ be a linear map. Then the following conditions are equivalent.
\begin{enumerate}  \item \label{it:11} $(A,\cdot,\circ,P,Q)$ is a relative PCA algebra.
\item \label{it:12}  $(-\mathcal{L}^{*}_{\cdot},-\mathcal{L}^{*}_{\circ},Q^{*},A^{*})$ is a representation of $(A,\cdot,[-,-],P)$.
\item \label{it:13} $(\mathcal{L}_{\cdot},-\mathcal{L}_{\circ},-Q,A)$ is a representation of $(A,\cdot,[-,-],P)$.
\item \label{it:14} $(A\ltimes_{-\mathcal{L}_{\cdot}^{*},-\mathcal{L}_{\circ}^{*}}A^{*},P+Q^{*})$ is a relative Poisson algebra.
\end{enumerate}
\end{pro}

\begin{proof}
(\ref{it:11}) $\Longleftrightarrow$ (\ref{it:12}). By
\cite{Bai2010} and \cite{LB2022},
$(-\mathcal{L}^{*}_{\cdot},A^{*})$ is a representation of
$(A,\cdot)$ and $(-\mathcal{L}^{*}_{\circ},A^{*})$ is a
representation of $(A,[-,-])$. For all $x,y\in A, a^{*}\in A^{*}$,
we have
\begin{equation*}
    \langle -Q^{*}\big(\mathcal{L}^{*}_{\circ}(x)a^{*}\big)+\mathcal{L}^{*}_{\circ}\big(P(x)\big)a^{*}+\mathcal{L}^{*}_{\circ}(x)Q^{*}(a^{*}),y\rangle=\langle a^{*}, x\circ Q(y)-P(x)\circ y-Q(x\circ y)\rangle.
\end{equation*}
Hence \eqref{eq:rps2} holds if and only if \eqref{eq:repRPA2} holds for  $\rho=-\mathcal{L}^{*}_{\circ},\alpha=Q^{*}, V=A^{*}$.
 Similarly, \eqref{eq:rps1}, \eqref{eq:rps3} and  \eqref{eq:rps4} hold if and only if \eqref{eq:repRPA1}, \eqref{eq:repRPA3} and \eqref{eq:repRPA4} hold for $\mu=-\mathcal{L}^{*}_{\cdot},\rho=-\mathcal{L}^{*}_{\circ},\alpha=Q^{*}, V=A^{*}$ respectively. Hence $(A,\cdot,\circ,P,Q)$ is a relative PCA algebra if and only if $(-\mathcal{L}^{*}_{\cdot},-\mathcal{L}^{*}_{\circ},Q^{*},A^{*})$ is a representation of $(A,\cdot,[-,-],P)$.

 (\ref{it:12}) $\Longleftrightarrow$ (\ref{it:13}). It follows from Proposition~\ref{pro:dual}.

 (\ref{it:12}) $\Longleftrightarrow$ (\ref{it:14}). It follows from Proposition~\ref{pro:semidirect}.
\end{proof}

\delete{
\begin{rmk}
Let $(A,\cdot,[-,-],P)$ be a relative Poisson algebra.
In \cite{RPA}, we introduce the condition when there is a linear map $Q:A\rightarrow A$ such that $(-\mathcal{L}^{*}_{\cdot},\mathrm{ad}^{*},Q^{*},A^{*})$ is a representation of $(A,\cdot,[-,-],P)$.
We treat the linear map $Q$ as an additional structure of $(A,\cdot,[-,-],P)$, since $Q$ is the only new participant in composing  $(-\mathcal{L}^{*}_{\cdot},\mathrm{ad}^{*},Q^{*},A^{*})$ as a representation of $(A,\cdot,[-,-],P)$.
In our present study, the representation on the dual space is turned to $(-\mathcal{L}^{*}_{\cdot},-\mathcal{L}^{*}_{\circ},Q^{*},A^{*})$, where there is another compatible anti-pre-Lie algebra $(A,\circ)$ involved, besides a linear map $Q$.
Hence we would like to consider $(A,\cdot,[-,-],P)$ together with the representation $(-\mathcal{L}^{*}_{\cdot},-\mathcal{L}^{*}_{\circ},Q^{*},A^{*})$ as a new algebra structure.
Also note that for a relative Poisson algebra $(A,\cdot,[-,-],P)$, the quadruple $(\mu,\rho,\alpha,V)$ is a representation if and only if $(-\mu^{*},\rho^{*},-\alpha^{*},V^{*})$ is a representation.
Hence the notion of a relative PCA algebra $(A,\cdot,\circ,P,Q)$ can also be characterized as a relative Poisson algebra $(A,\cdot,[-,-],P)$ with a representation $(\mathcal{L}_{\cdot},-\mathcal{L}_{\circ},-Q,A)$.
\end{rmk}}

\begin{defi}\cite{HBG}
Let $(A,\cdot, P)$ be a commutative differential algebra. Suppose
that $Q:A\rightarrow A$ is a linear map satisfying
\eqref{eq:rps1}. Then the quadruple $(A,\cdot,P,Q)$ is called an
{\bf admissible commutative differential algebra}.
\end{defi}

\begin{pro}\label{ex:ex}
    Let $(A,\cdot, P, Q)$ be an admissible commutative differential algebra. Define a binary operation $\circ$ on $A$ by
    \begin{equation}\label{eq:ex}
        x\circ y=Q(x\cdot y)-P(x)\cdot y,\;\forall x,y\in A.
    \end{equation}
 Then $(A,\circ)$ is  an anti-pre-Lie algebra
 and moreover $(A,\cdot,\circ,P,Q)$ is a relative PCA algebra.
\end{pro}
\begin{proof}
    It can be obtained from a direct verification. We give another proof in terms of representations of relative Poisson algebras.
    By Proposition~\ref{ex1.2}, there is a relative Poisson algebra $(A,\cdot,[-,-],P)$, where $(A,[-,-])$ is the corresponding Witt type Lie algebra
   defined by (\ref{eq:Witt Lie}). By \cite{LLB2023} or by \eqref{eq:rps1} with a straightforward checking,  $P+Q^{*}$ is a derivation of the semi-direct product commutative associative algebra $A\ltimes_{-\mathcal{L}^{*}_{\cdot}}A^{*}$ whose product is given by
    \begin{equation}
        (x+a^{*})\cdot(y+b^{*})=x\cdot y-\mathcal{L}^{*}_{\cdot}(x)b^{*}-\mathcal{L}^{*}_{\cdot}(y)a^{*},\;\forall x,y\in A, a^{*},b^{*}\in A^{*}.
    \end{equation}
    Hence there is a Witt type Lie algebra structure on $A\oplus A^{*}$ given by
    \begin{small}
    \begin{eqnarray*}
        \ [x+a^{*},y+b^{*}]&=&(x+a^{*})\cdot(P+Q^{*})(y+b^{*})-(P+Q^{*})(x+a^{*})\cdot(y+b^{*})\\
        &=&x\cdot P(y)-\mathcal{L}^{*}_{\cdot}(x)Q^{*}(b^{*})-\mathcal{L}^{*}_{\cdot}\big(P(y)\big)a^{*}-P(x)\cdot y+\mathcal{L}^{*}_{\cdot}\big(P(x)\big)b^{*}+\mathcal{L}^{*}_{\cdot}(y)Q^{*}(a^{*}),
    \end{eqnarray*}
    \end{small}for all $x,y\in A, a^{*},b^{*}\in A^{*}$.
 Note that \eqref{eq:ex} holds if and only if the following equation holds:
    \begin{equation*}
        -\mathcal{L}^{*}_{\cdot}(x)Q^{*}(b^{*})+\mathcal{L}^{*}_{\cdot}\big(P(x)\big)b^{*}=-\mathcal{L}_{\circ}^{*}(x)b^{*}, \;\;\forall x\in A, b^*\in A^*.
    \end{equation*}
Therefore we have
    \begin{equation*}
        [x+a^{*},y+b^{*}]=[x,y]-\mathcal{L}^{*}_{\circ}(x)b^{*}+\mathcal{L}^{*}_{\circ}(y)a^{*},\;\;\forall x,y\in A, a^*,b^*\in A^*.
    \end{equation*}
Also note that $(A,\circ)$ is a Lie-admissible Lie algebra since the sub-adjacent Lie algebra of $(A,\circ)$ is also
$(A,[-,-])$ itself.  Thus $(A,\circ)$ is an anti-pre-Lie algebra since  $(-\mathcal{L}^{*}_{\circ},A^{*})$ is
     a representation  of $(A,[-,-])$.
   Moreover, by Proposition~\ref{ex1.2} again, $(A\ltimes_{-\mathcal{L}^{*}_{\cdot},-\mathcal{L}^{*}_{\circ}}A^{*},P+Q^{*})$ is a
   relative Poisson algebra, and thus $(-\mathcal{L}^{*}_{\cdot},-\mathcal{L}^{*}_{\circ},Q^{*},A^{*})$ is a representation of $(A,\cdot,[-,-],P)$ by Proposition~\ref{pro:semidirect}.
    Then by Proposition \ref{pro2.5}, $(A,\cdot,\circ,P,Q)$ is a relative PCA algebra.
\end{proof}

\begin{pro}\label{pro2.9}\cite{LB2022}
 Let $(A,\cdot, P)$ be a commutative differential algebra and $\mathcal{B}$ be a symmetric invariant bilinear form on $(A,\cdot)$.
 Then $\mathcal{B}$ is a  commutative $2$-cocycle on the corresponding Witt type Lie algebra $(A,[-,-])$ defined by \eqref{eq:Witt Lie}.
\end{pro}

\begin{defi}
Let $(A,\cdot,[-,-],P)$ be a relative Poisson algebra. If there is
a nondegenerate  symmetric bilinear form $\mathcal{B}$ on $A$ such
that it is invariant on $(A,\cdot)$ and a commutative $2$-cocycle
on $(A,[-,-])$, then we say $\mathcal{B}$ is an {\bf \Witt form
on $(A,\cdot,[-,-],P)$}.
\end{defi}

\begin{ex}\label{ex:poly}
Let $A=\mathbb K[t,t^{-1}]$ be the space of Laurent polynomials.
Then $A$ is equipped with a commutative associative multiplication by
\begin{equation}
    t^{m}\cdot t^{n}=t^{m+n},\;\forall m,n\in\mathbb{Z}.
\end{equation}
Moreover, there is a derivation $P$ on $(A,\cdot)$ given by
\begin{equation}
    P(t^{n})=nt^{n-1},
\end{equation}
and for any $a\in\mathbb Z$, there is a nondegenerate symmetric invariant bilinear form on $(A,\cdot)$ defined by
\begin{equation}
    \mathcal{B}_a(t^{m},t^{n})=\delta_{m+n,a}.
\end{equation}
Note that the corresponding Witt type Lie algebra $(A,[-,-])$ is given by
\begin{equation}
[t^{m},t^{n}]=t^{m}\cdot P(t^{n})-P(t^{m})\cdot t^{n}=(n-m)t^{m+n-1}.
\end{equation}
Therefore  $(A,\cdot,[-,-],P)$ is a relative Poisson algebra and for any $a\in \mathbb Z$,  $\mathcal B_a$ is a commutative $2$-cocycle on $(A$, $[-,-])$, that is, $\mathcal B_a$ is an \Witt form on $(A,\cdot,[-,-]$,$P)$.
\end{ex}

\begin{rmk}
    Recall \cite{RPA} that a {\bf (symmetric) Frobenius relative Poisson algebra} is a  relative Poisson algebra $(A,\cdot,[-,-],P)$
   together with a nondegenerate (symmetric) bilinear form $\mathcal{B}$ which is invariant on both $(A,\cdot)$ and $(A,[-,-])$.
   Note that a Lie algebra with a nondegenerate symmetric bilinear form $\mathcal{B}$ which is both invariant and a commutative $2$-cocycle is abelian. Hence with the condition in Proposition \ref{pro2.9}, if in addition the bilinear from $\mathcal B$ is nondegenerate, then $(A,\cdot,[-,-],P)$ with $\mathcal B$ in general cannot be  a symmetric Frobenius relative Poisson algebra unless $(A,[-,-])$ is abelian. Also note that in this case, the derivation $P$ satisfies $P(x)\cdot y=x\cdot P(y)$ for all $x,y\in A$.
 Therefore the typical examples of relative Poisson algebras constructed by Proposition~\ref{ex1.2} with the nondegenerate bilinear forms which are invariant on the commutative associative algebras cannot be symmetric Frobenius relative Poisson algebras unless the corresponding Witt type Lie algebras are abelian.
 \end{rmk}

\begin{pro}\label{ex:ex2}
Let $(A,\cdot,\circ,P,Q)$ be a relative PCA algebra and $(A,\cdot,
[-,-],P)$ be the associated relative Poisson algebra. Suppose that
$(A\ltimes_{-\mathcal{L}_{\cdot}^{*},-\mathcal{L}_{\circ}^{*}}A^{*},P+Q^{*})$
is the corresponding relative Poisson algebra given by
Proposition~\ref{pro2.5}. Then the natural  nondegenerate
symmetric  bilinear form $\mathcal{B}_{d}$ on $A\oplus A^*$
defined by
    \begin{equation}\label{eq:Bd}
        \mathcal{B}_{d}(x+a^{*},y+b^{*})=\langle x, b^{*}\rangle+\langle a^{*},y\rangle,\;\forall x,y\in A, a^{*},b^{*}\in A^{*}
    \end{equation}
    is invariant on
    the commutative associative algebra $A\ltimes_{-\mathcal{L}_{\cdot}^{*}} A^*$ and a commutative $2$-cocycle on the Lie algebra
    $A\ltimes_{-\mathcal{L}_{\circ}^{*}}A^{*}$, that is,
    $\mathcal{B}_{d}$ is an \Witt form on
    $(A\ltimes_{-\mathcal{L}_{\cdot}^{*},-\mathcal{L}_{\circ}^{*}}A^{*},P+Q^{*})$.
\end{pro}

\begin{proof}
It is straightforward to show that $\mathcal{B}_{d}$ is invariant
on the commutative associative algebra
$A\ltimes_{-\mathcal{L}_{\cdot}^{*}}A^*$. Moreover, by
\cite[Proposition 2.26]{LB2022}, $\mathcal{B}_{d}$ is a
commutative $2$-cocycle on the Lie algebra
$A\ltimes_{-\mathcal{L}_{\circ}^{*}}A^{*}$.
\end{proof}

\begin{rmk}
Let $(A,\cdot, P, Q)$ be an admissible commutative differential
algebra and $(A,\cdot,\circ,P,Q)$ be the relative PCA algebra
given in Proposition~\ref{ex:ex}, where  $\circ$ is defined by
{\rm (\ref{eq:ex})}. The fact that $\mathcal B_d$ is a commutative
$2$-cocycle on the Lie algebra
    $A\ltimes_{-\mathcal{L}_{\circ}^{*}}A^{*}$ also can be
    obtained from Proposition~\ref{pro2.9} since the Lie algebra
 $A\ltimes_{-\mathcal{L}_{\circ}^{*}}A^{*}$ is the
Witt type algebra of the commutative differential algebra
$(A\ltimes_{-\mathcal{L}_{\cdot}^{*}}A^*,P+Q^*)$.
\end{rmk}

\delete{
\begin{rmk}
    With the condition in
    Example \ref{ex:ex}, there is another semi-direct Lie algebra $A\ltimes_{\mathrm{ad}^{*}}A^{*}$ with respect to the coadjoint representation $(\mathrm{ad}^{*}, A^{*})$, that is,
    \begin{equation*}
        [x+a^{*},y+b^{*}]'=[x,y]+\mathrm{ad}^{*}(x)b^{*}-\mathrm{ad}^{*}(y)a^{*},\;\forall x,y\in A, a^{*},b^{*}\in A^{*}.
    \end{equation*}
    Moreover, the  natural nondegenerate symmetric bilinear form $\mathcal{B}_{d}$ given by
    \eqref{eq:Bd} is invariant on $A\ltimes_{\mathrm{ad}^{*}}A^{*}$.
    The two Lie algebras $A\ltimes_{\mathrm{ad}^{*}}A^{*}$ and $A\ltimes_{-\mathcal{L}_{\circ}^{*}}A^{*}$ coincide if and only if $\mathrm{ad}=-\mathcal{L}_{\circ}$, that is,
    \begin{equation}\label{eq:coin}
        [x,y]=-x\circ y.
    \end{equation}
    By \eqref{eq:Witt Lie}, \eqref{eq:rps1} and \eqref{eq:ex}, the condition \eqref{eq:coin} holds if and only if
    \begin{equation*}
        Q(x\cdot y)=x\cdot P(y)-2P(x)\cdot y,
        \end{equation*}
        which also indicates $P(x)\cdot y=x\cdot P(y)$. In this case, $A\ltimes_{\mathrm{ad}^{*}}A^{*}=A\ltimes_{-\mathcal{L}_{\circ}^{*}}A^{*}$ is an abelian Lie algebra.
\end{rmk}
}

\begin{ex}
Let $A=\mathrm{span}\{e_{1}, e_{2}, e_{3}\}$ be a vector space,
$\cdot:A\otimes A\rightarrow A$ be a binary operation with nonzero
products given by
    \begin{equation}\label{eq:comm product}
        e_{1}\cdot e_{1}=2e_{3},\; e_{1}\cdot e_{2}=e_{3}
    \end{equation}
    and $P:A\rightarrow A$ be a linear map given by
    \begin{equation*}
        P(e_{1})=e_{1}+e_{2},\; P(e_{2})=2e_{2},\; P(e_{3})=3e_{3}.
    \end{equation*}
    Then $(A,\cdot,P,-P)$ is an admissible  commutative differential
    algebra. By Proposition \ref{ex:ex}, there is an anti-pre-Lie algebra $(A,\circ)$ defined by
\begin{equation*}
    x\circ y=-P(x\cdot y)-P(x)\cdot y,\;\forall x,y\in A,
\end{equation*}
whose nonzero products are
\begin{equation}\label{eq:apl product}
e_{1}\circ e_{1}=-9e_{3},\; e_{1}\circ e_{2}=-4e_{3},\; e_{2}\circ e_{1}=-5e_{3}.
\end{equation}
Moreover, $(A,\cdot,\circ,P,-P)$ is a relative PCA algebra. By
Proposition~\ref{pro2.5},
$(A\ltimes_{-\mathcal{L}_{\cdot}^{*},-\mathcal{L}_{\circ}^{*}}A^{*},P-P^{*})$
is a relative Poisson algebra, whose nonzero products are given by
\eqref{eq:comm product} and
\begin{eqnarray*}
&& e_{1}\cdot e_{3}^{*}=2e^{*}_{1}+e^{*}_{2},\; e_{2}\cdot e^{*}_{3}=e^{*}_{1},\\
&& [e_{1},e_{2}]=e_{3},\; [e_{1}, e^{*}_{3}]=-9e^{*}_{1}-4e^{*}_{2},\; [e_{2}, e^{*}_{3}]=-5e^{*}_{1}.
\end{eqnarray*}
 By
Proposition~\ref{ex:ex2}, the bilinear form $\mathcal{B}_{d}$ on
$A\oplus A^*$ defined by \eqref{eq:Bd} is an \Witt form on
$(A\ltimes_{-\mathcal{L}_{\cdot}^{*},-\mathcal{L}_{\circ}^{*}}A^{*},P-P^{*})$.
\end{ex}


Conversely, a relative Poisson algebra with an \Witt form induces
a relative PCA algebra.

\begin{lem}\cite{LB2022}
Let $(A,[-,-])$ be a Lie algebra with a nondegenerate commutative $2$-cocycle $\mathcal{B}$.
Then there is an anti-pre-Lie algebra structure on $A$ defined by
\begin{equation}\label{eq:APLc2c}
    \mathcal{B}(x\circ y,z)=\mathcal{B}(y,[x,z]),\;\forall x,y,z\in A,
\end{equation}
whose sub-adjacent Lie algebra is $(A,[-,-])$.
\delete{
In particular, if there is a commutative associative algebra $(A,\cdot)$ with a derivation $P$ such that $(A,[-,-])$ is the Witt type Lie algebra of $(A,\cdot,P)$ and $\mathcal{B}$ is invariant on $(A,\cdot)$, then we have
\begin{equation}
    x\circ y=\hat{P}(x\cdot y)-P(x)\cdot y,
\end{equation}
where $\hat{P}:A\rightarrow A$ is the adjoint map of $P$ given by
\begin{equation}\label{eq:hat}
    \mathcal{B}\big(\hat{P}(x),y\big)=\mathcal{B}\big(x,P(y)\big),\;\forall x,y\in A.
\end{equation}}
\end{lem}

\begin{pro}\label{pro:from Witt to system}
Let $(A,\cdot,[-,-],P)$ be a relative Poisson algebra with an
\Witt form $\mathcal{B}$. Then there is a relative PCA algebra
$(A,\cdot,\circ,P,\hat{P})$ in which $\circ$ is defined by
\eqref{eq:APLc2c}, and $\hat{P}:A\rightarrow A$ is the adjoint map
of $P$ defined by
\begin{equation}\label{eq:hat}
    \mathcal{B}\big(\hat{P}(x),y\big)=\mathcal{B}\big(x,P(y)\big),\;\forall x,y\in A.
\end{equation}
\end{pro}

\begin{proof}
\delete{
    By \cite{Bai2010} and \cite{LB2022}, $(-\mathcal{L}^{*}_{\cdot},A^{*})$ is a representation of $(A,\cdot)$ and
$(-\mathcal{L}^{*}_{\circ},A^{*})$ is a representation of $(A,[-,-])$.
We set a linear isomorphism $\phi:A\rightarrow A^{*}$ by
\begin{equation}
    \langle \phi(x),y\rangle=\mathcal{B}(x,y),\;\forall x,y\in A.
\end{equation}}
For all $x,y,z\in A$, we have
\begin{eqnarray*}
    \mathcal{B}\big(y,x\circ \hat{P}(z)-P(x)\circ z-\hat{P}(x\circ z)\big)&\overset{\eqref{eq:APLc2c}}{=}&\mathcal{B}\big( [x,y],\hat{P}(z) \big)-\mathcal{B}\big( [P(x),y],z \big)-\mathcal{B}\big(P(y),x\circ z\big)\\
    &=&\mathcal{B}\big( P([x,y]),z\big)-\mathcal{B}\big( [P(x),y],z \big)-\mathcal{B}\big([x,P(y)],  z\big)\\
    &=&0.
\end{eqnarray*}
Hence \eqref{eq:rps2} with $Q=\hat P$ holds. Similarly \eqref{eq:rps1}, \eqref{eq:rps3} and \eqref{eq:rps4} with $Q=\hat P$ hold, and thus $(A,\cdot,\circ,P,\hat{P})$ is a relative PCA algebra.
\end{proof}

\begin{ex}
Let $(A,\cdot,[-,-],P)$ be the relative Poisson algebra together
with the bilinear form $\mathcal{B}_a$ (for any $a\in \mathbb Z$)
given in Example \ref{ex:poly}. Then for any $a \in \mathbb Z$, by
Proposition \ref{pro:from Witt to system}, there is a relative PCA
algebra  $(A,\cdot,\circ,P,\hat{P})$, where
\begin{equation*}
  \hat{P}(t^{n})=(a+1-n)t^{n-1},\;  t^{m}\circ t^{n}=\hat{P}(t^{m}\cdot t^{n})-P(t^{m})\cdot t^{n}=(a+1-2m-n)t^{m+n-1},\;\forall m,n\in\mathbb Z.
\end{equation*}
In particular when $a=-1$, we have
\begin{equation*}
    \hat{P}=-P,\; t^{m}\circ t^{n}=-( 2m+n)t^{m+n-1}.
\end{equation*}
\end{ex}

At the end of this subsection, we would like to study unital
relative Poisson algebras, or equivalently, Jacobi algebras, with
\Witt forms.
 First, we recall the notion of Jacobi algebras which arise in the study of Jacobi
    manifolds.
\begin{defi} \cite{Ago}
A
    {\bf Jacobi algebra} is a triple $(A,\cdot,$ $[-,-])$, such that
    $(A,\cdot)$ is a unital commutative associative algebra,
    $(A,[-,-])$ is a Lie algebra and the following equation is satisfied:
\begin{equation}\label{eq:JA}
    [z,x\cdot y]=[z,x]\cdot y+x\cdot[z,y]+x\cdot y\cdot[1_{A},z],
    \;\;\forall x,y,z\in A.
\end{equation}
\end{defi}
\begin{lem}\cite{RPA}\label{lem:930}
Suppose that $(A,\cdot,[-,-],P)$ is a  relative Poisson algebra with the unit $1_{A}$.
Then $(A,\cdot,[-,-] )$ is a  Jacobi algebra and $P=\mathrm{ad}(1_{A})$.
Conversely, suppose that $(A,\cdot,[-,-] )$ is a  Jacobi algebra. Then $(A,\cdot,[-,-],P=\mathrm{ad}(1_{A}))$ is a  relative Poisson algebra with the unit $1_{A}$.
Hence there is a one-to-one correspondence between unital relative Poisson algebras and Jacobi algebras.
\end{lem}

Frobenius Jacobi algebras, as Jacobi algebras with nondegenerate
symmetric bilinear forms which are invariant on both the
commutative associative and Lie algebras, play important roles in
the study of Jacobi manifolds (\cite{Ago}). So it is reasonable to
consider the possible applications of Jacobi algebras with \Witt
forms in the study of Jacobi manifolds. For this purpose, we give
a structure theory of Jacobi algebras with \Witt forms as follows.

\begin{pro}\label{pro:939}
Suppose that $(A,\cdot,\circ,P,Q)$ is a relative PCA algebra.
If $(A,\cdot)$ has the unit $1_{A}$ such that $(A,\cdot,[-,-])$ is a Jacobi algebra, then $(A,[-,-])$ is the Witt type Lie algebra of the commutative differential algebra $(A,\cdot,P=\mathrm{ad}(1_{A}))$.
\end{pro}
\begin{proof}
Taking $y=z=1_{A}$ into \eqref{eq:rps4}, we have $1_{A}\circ x=Q(x)$ and hence $x\circ 1_{A}=(Q-P)x$.
    Taking $ z=1_{A}$ into \eqref{eq:rps3}, we have
    \begin{eqnarray}
        0=x\cdot(y\circ 1_{A}) -y\circ x+[x,y]-x\cdot P(y)=x\cdot (Q-2P)y -y\circ x+[x,y].\label{eq:954}
    \end{eqnarray}
Thus we have
\begin{eqnarray*}
    3[x,y]=x\circ y-[y,x] -y\circ x+[x,y]\overset{\eqref{eq:954}}{=}x\cdot (2P-Q)y-y\cdot (2P-Q)x\overset{\eqref{eq:rps1}}{=}3x\cdot P(y)-3y\cdot P(x).
\end{eqnarray*}
Hence  $(A,[-,-])$ is the Witt type Lie algebra of $(A,\cdot,P=\mathrm{ad}(1_{A}))$.
\end{proof}
\begin{cor}
Suppose that $(A,\cdot,[-,-])$ is a Jacobi algebra with an \Witt form $\mathcal{B}$.
Then $(A,[-,-])$ is the Witt type Lie algebra of the commutative differential algebra $(A,\cdot,P=\mathrm{ad}(1_{A}))$.
\end{cor}
\begin{proof}
By Lemma \ref{lem:930}, $(A,\cdot,[-,-]$,   $P=\mathrm{ad}(1_{A}))$ is a relative Poisson algebra with the \Witt form $\mathcal{B}$.    By Proposition \ref{pro:from Witt to system}, there is a relative PCA algebra $(A,\cdot,\circ,P,Q=\hat{P})$ in which $\circ$ is defined by  \eqref{eq:APLc2c}.
Hence the conclusion follows from Proposition \ref{pro:939}.
\end{proof}
Conversely, suppose that $(A,\cdot,P)$ is a unital commutative differential algebra and $\mathcal{B}$ is a  nondegenerate symmetric invariant bilinear form on $(A,\cdot)$. Then it is straightforward to check that $(A,\cdot,[-,-])$ is a Jacobi algebra with an \Witt form $\mathcal{B}$, where $(A,[-,-])$ is the Witt type Lie algebra of $(A,\cdot,P)$ and in this case we also have $P=\mathrm{ad}(1_{A})$.
In conclusion, we have the following result.
\begin{thm}
Every Jacobi algebra  with an \Witt form is realized from a unital commutative differential algebra  with a nondegenerate symmetric invariant bilinear form  and the corresponding Witt type Lie algebra.
\end{thm}

\subsection{Representations and matched pairs of relative PCA algebras}\

Recall that a {\bf representation of an anti-pre-Lie
algebra $(A,\circ)$} \cite{TPA} is a triple $(l_{\circ},r_{\circ},V)$, such
that $V$ is a vector space,
$l_{\circ},r_{\circ}:A\rightarrow\mathrm{End}_{\mathbb K} (V)$ are
linear maps and the following equations hold: 
    \begin{eqnarray*}
    l_{\circ}(y\circ x)v-l_{\circ}(x\circ y)v&=&l_{\circ}(x)l_{\circ}(y)v-l_{\circ}(y)l_{\circ}(x)v,\label{eq:defi:rep anti-pre-Lie algebra1}\\
    r_{\circ}(x\circ y)v&=&l_{\circ}(x)r_{\circ}(y)v+r_{\circ}(y)l_{\circ}(x)v-r_{\circ}(y)r_{\circ}(x)v,\label{eq:defi:rep anti-pre-Lie algebra2}\\
     l_{\circ}(y\circ x)v-l_{\circ}(x\circ y)v&=&r_{\circ}(x)l_{\circ}(y)v-r_{\circ}(y)l_{\circ}(x)v-r_{\circ}(x)r_{\circ}(y)v+r_{\circ}(y)r_{\circ}(x)v,\label{eq:defi:rep anti-pre-Lie algebra3}
\end{eqnarray*}
for all $x,y\in A,v\in V$. Let $(A,\circ)$ be an anti-pre-Lie
algebra, $V$ be a vector space and
$l_{\circ},r_{\circ}:A\rightarrow\mathrm{End}_{\mathbb K} (V)$ be
linear maps. Then $(l_{\circ},r_{\circ},V)$ is a representation of
$(A,\circ)$ if and only if there is an anti-pre-Lie algebra
structure on the direct sum $A\oplus V$ of vector spaces defined by
\begin{equation}\label{eq:sd,apl}
    (x+u)\circ(y+v)=x\circ y+l_{\circ}(x)v+r_{\circ}(y)u,\;\forall x,y\in A, u,v\in V.
\end{equation}

Now we introduce the notion of a representation of a relative PCA algebra.

\begin{defi}
Let $(A,\cdot,\circ,P,Q)$ be a relative PCA algebra, $V$ be a
vector space, and
$\mu,l_{\circ},r_{\circ}:A\rightarrow\mathrm{End}_{\mathbb K} (V)$
and $\alpha,\beta:V\rightarrow V$ be linear maps. Suppose that
$(\mu,V)$ is a representation of $(A,\cdot)$ and
$(l_{\circ},r_{\circ},V)$ is a representation of $(A,\circ)$. If
the following equations hold:
\begin{eqnarray}
\mu(x)(l_{\circ}-r_{\circ})(y)v-\mu([x,y])v-(l_{\circ}-r_{\circ})(y)\mu(x)v+\mu\big(x\cdot P(y)\big)v&=&0,\label{eq:defi:RAPLP rep1}\\
(l_{\circ}-r_{\circ})(x\cdot y)v-\mu(x)(l_{\circ}-r_{\circ})(y)v-\mu(y)(l_{\circ}-r_{\circ})(x)v+\mu(x\cdot y)\alpha(v)&=&0,\label{eq:defi:RAPLP rep2}\\
\alpha\big(\mu(x)v\big)-\mu\big(P(x)\big)v-\mu(x)\alpha(v)&=&0,\label{eq:defi:RAPLP rep3}\\
\alpha\big((l_{\circ}-r_{\circ})(x)v\big)-(l_{\circ}-r_{\circ})\big(P(x)\big)v-(l_{\circ}-r_{\circ})(x)\alpha(v)&=&0,\label{eq:defi:RAPLP rep4}\\
\mu\big(Q(x)\big)v-\mu(x)\alpha(v)-\beta\big(\mu(x)v\big)&=&0,\label{eq:defi:RAPLP rep5}\\
\mu(x)\beta(v)-\mu\big(P(x)\big)v-\beta\big(\mu(x)v\big)&=&0,\label{eq:defi:RAPLP rep6}\\
r_{\circ}\big(Q(x)\big)v-r_{\circ}(x)\alpha(v)-\beta\big(r_{\circ}(x)v\big)&=&0,\label{eq:defi:RAPLP rep7}\\
l_{\circ}(x)\beta(v)-l_{\circ}\big(P(x)\big)v-\beta\big(l_{\circ}(x)v\big)&=&0,\label{eq:defi:RAPLP rep8}\\
\mu(y\circ x)v-l_{\circ}(y)\mu(x)v+\mu(x)r_{\circ}(y)v-\mu(x)l_{\circ}(y)v-\mu\big(x\cdot P(y)\big)v&=&0,\label{eq:defi:RAPLP rep9}\\
\mu(x)r_{\circ}(y)v-r_{\circ}(x\cdot y)v+\mu(y)l_{\circ}(x)v-\mu(y)r_{\circ}(x)v-\mu(x\cdot y)\alpha(v)&=&0,\label{eq:defi:RAPLP rep10}\\
\mu(x)l_{\circ}(y)v-l_{\circ}(y)\mu(x)v+\mu([x,y])v-\mu\big(x\cdot P(y)\big)v&=&0,\label{eq:defi:RAPLP rep11}\\
r_{\circ}(x)\mu(y)v-l_{\circ}(y)\mu(x)v-r_{\circ}(x\cdot y)v+\beta\big(\mu(x\cdot y)v\big)&=&0,\label{eq:defi:RAPLP rep12}\\
l_{\circ}(x\cdot
y)v-l_{\circ}(y)\mu(x)v-l_{\circ}(x)\mu(y)v+\beta\big(\mu(x\cdot
y)v\big)&=&0,\ \ \ \ \ \ \label{eq:defi:RAPLP rep13}
\end{eqnarray}
for all $x,y\in A, v\in V$, then we say
$(\mu,l_{\circ},r_{\circ},\alpha,\beta,V)$ is a {\bf
representation of $(A,\cdot,\circ,P,Q)$}.
\end{defi}

\begin{ex}
    Let $(A,\cdot,\circ,P,Q)$ be a relative PCA algebra.
    Then $(\mathcal{L}_{\cdot},\mathcal{L}_{\circ},\mathcal{R}_{\circ},P,Q)$ is a representation of $(A,\cdot,\circ,P,Q)$,
    which is called the {\bf adjoint representation of $(A,\cdot,\circ,P,Q)$}.
\end{ex}

\begin{pro}
    Let $(A,\cdot,\circ,P,Q)$ be a relative PCA algebra, $V$ be a vector space, and $\mu,l_{\circ},r_{\circ}:A\rightarrow\mathrm{End}_{\mathbb K} (V)$ and $\alpha,\beta:V\rightarrow V$ be linear maps.
    Then $(\mu,l_{\circ},r_{\circ},\alpha,\beta,V)$ is a  representation  of $(A,\cdot,\circ,P,Q)$ if and only if there is a relative PCA algebra $(A\oplus V,\cdot,\circ,P+\alpha,Q+\beta)$ in which $\cdot$ is defined by \eqref{eq:sd,asso}
    and $\circ$ is defined by \eqref{eq:sd,apl}.
    We call the relative PCA algebra structure on $A\oplus V$ the {\bf semi-direct product relative PCA algebra} and denote it by
    $(A\ltimes_{\mu,l_{\circ},r_{\circ}}V,P+\alpha,Q+\beta)$.
\end{pro}
\begin{proof}
    It is the special case of Proposition \ref{pro:2.20} when $A_{2}=V$ is equipped with the zero multiplication.
\end{proof}

\begin{pro}
   Let $(A,\cdot,\circ,P,Q)$ be a relative PCA algebra.
   If $(\mu,l_{\circ},r_{\circ},\alpha,\beta,V)$ is a   representation  of $(A,\cdot,\circ,P,Q)$,
   then $(-\mu^{*},r^{*}_{\circ}-l^{*}_{\circ},r^{*}_{\circ},\beta^{*},\alpha^{*},V^{*})$ is also a representation of $(A,\cdot,\circ,P,Q)$.
   In particular, $(-\mathcal{L}^{*}_{\cdot},-\mathrm{ad}^{*},\mathcal{R}^{*}_{\circ},Q^{*},P^{*},A^{*})$ is a representation of $(A,\cdot,\circ,P,Q)$.
\end{pro}
\begin{proof}
    By \cite{Bai2010}, $(-\mu^{*},V^{*})$ is a representation of $(A,\cdot)$.
    By \cite{TPA}, $( r^{*}_{\circ}-l^{*}_{\circ},r^{*}_{\circ}, V^{*} )$ is a representation of $(A,\circ)$.
    For all $x,y\in A, u^{*}\in V^{*},v\in V$, we have
    \begin{eqnarray*}
&&\langle -\mu^{*}(x)(r^{*}_{\circ}-l^{*}_{\circ}-r^{*}_{\circ})(y)u^{*}+\mu^{*}([x,y])u^{*}
+(r^{*}_{\circ}-l^{*}_{\circ}-r^{*}_{\circ})(y)\mu^{*}(x)u^{*}-\mu^{*}\big(x\cdot P(y)\big)u^{*},v\rangle\\
&&=\langle u^{*}, l_{\circ}(y)\mu(x)v-\mu([x,y])v-\mu(x)l_{\circ}(y)v+\mu\big(x\cdot P(y)\big)v\rangle\\
&&\overset{\eqref{eq:defi:RAPLP rep11}}{=}0.
\end{eqnarray*}
Hence for
$(-\mu^{*},r^{*}_{\circ}-l^{*}_{\circ},r^{*}_{\circ},\beta^{*},\alpha^{*},V^{*})$,
\eqref{eq:defi:RAPLP rep1} holds. Similarly, for
$(-\mu^{*},r^{*}_{\circ}-l^{*}_{\circ},r^{*}_{\circ},\beta^{*},\alpha^{*},V^{*})$,
\eqref{eq:defi:RAPLP rep2}-\eqref{eq:defi:RAPLP rep13} hold and
thus it is a representation of $(A,\cdot,\circ,P,Q)$.
\end{proof}

We recall the notions of matched pairs of commutative associative algebras, Lie algebras, relative Poisson algebras and anti-pre-Lie algebras.

\begin{defi}\label{defi:mp}
\begin{enumerate}
    \item \cite{Bai2010} Let $(A_{1},\cdot_{1})$ and $(A_{2},\cdot_{2})$ be two commutative
associative algebras, and
$\mu_{1}:A_{1}\rightarrow\mathrm{End}_{\mathbb K} (A_{2})$ and
$\mu_{2}:A_{2}\rightarrow\mathrm{End}_{\mathbb K} (A_{1})$ be
linear maps. We say $\big((A_{1},\cdot_{1})$, $(A_{2}$,
$\cdot_{2}),\mu_{1},\mu_{2}\big)$ is a \textbf{matched pair of
commutative associative  algebras} if there is a commutative
associative algebra $(A_{1}\oplus A_{2},\cdot)$ in which $\cdot$
is defined by
\begin{equation}\label{eq:mp,ca}
    (x+a)\cdot (y+b)=x\cdot_{1}
y+\mu_{2}(a)y+\mu_{2}(b)x+a\cdot_{2}
b+\mu_{1}(x)b+\mu_{1}(y)a,\;\forall x,y\in A_1, a,b\in A_2.
\end{equation}
\item\cite{Maj} Let $(A_{1},[-,-]_{1})$ and $(A_{2},[-,-]_{2})$ be
two Lie algebras, and $
\rho_{1}:A_{1}\rightarrow\mathrm{End}_{\mathbb K} (A_{2})$ and $
\rho_{2}:A_{2}\rightarrow\mathrm{End}_{\mathbb K} (A_{1})$ be
linear maps. We say $\big((A_{1},[-,-]_{1}),(A_{2},[-,-]_{2})$,
$\rho_{1},\rho_{2}\big)$  is   a \textbf{matched pair of Lie
algebras}  if there is a Lie algebra  $(A_{1}\oplus A_{2}, [-,-]
)$ in which $[-,-]$ is defined by
\begin{equation}\label{eq:mp,Lie}
    [x+a,y+b]=[x,y]_{1}+\rho_{2}(a)y-\rho_{2}(b)x+[a,b]_{2}+\rho_{1}(x)b-\rho_{1}(y)a,\;\;\forall
x,y\in A_1, a,b\in A_2.
\end{equation}
\item \cite{RPA} Let $(A_{1},\cdot_{1},[-,-]_{1},P_{1})$ and
$(A_{2},\cdot_{2},[-,-]_{2},P_{2})$ be two relative Poisson
algebras, and
$\mu_{1},\rho_{1}:A_{1}\rightarrow\mathrm{End}_{\mathbb K}
(A_{2})$ and
$\mu_{2},\rho_{2}:A_{2}\rightarrow\mathrm{End}_{\mathbb K}
(A_{1})$ be linear maps. We say $\big((A_{1}$, $\cdot_{1}$,
$[-,-]_{1},P_{1}),(A_{2},\cdot_{2},[-,-]_{2},P_{2}),\mu_{1},\rho_{1},\mu_{2},\rho_{2}\big)$
is a \textbf{matched pair of relative Poisson algebras}  if there
is a relative Poisson algebra structure $(A_{1}\oplus
A_{2},\cdot,[-,-],P_{1}+P_{2})$ in which $\cdot$ and $[-,-]$ are
defined by \eqref{eq:mp,ca} and \eqref{eq:mp,Lie} respectively.
\item \cite{TPA} Let $(A_{1},\circ_{1})$ and $(A_{2},\circ_{2})$
be two anti-pre-Lie algebras,  and $l_{1},r_{1}
:A_{1}\rightarrow\mathrm{End}_{\mathbb K} (A_{2})$ and $l_{2},
r_{2}:A_{2}\rightarrow\mathrm{End}_{\mathbb K} (A_{1})$ be linear
maps. We say $\big((A_{1},\circ_{1} ),(A_{2},\circ_{2} ), l_{1}$,
$r_{1},l_{2},r_{2}\big)$ is a {\bf matched pair of anti-pre-Lie
algebras} if there is an anti-pre-Lie algebra  $(A_{1}\oplus
A_{2},\circ)$ in which $\circ$ is defined by
\begin{equation}\label{eq:mp,APL}
    (x+a)\circ (y+b)=x\circ_{1}
y+l_{2}(a)y+r_{2}(b)x+a\circ_{2} b+l_{1}(x)b+r_{1}(y)a,\;\forall
x,y\in A_1, a,b\in A_2.
\end{equation}
\end{enumerate}
\end{defi}

\begin{pro}
\cite{RPA}  Suppose that $(A_{1},\cdot_{1},[-,-]_{1},P_{1})$ and
$(A_{2},\cdot_{2},[-,-]_{2},P_{2})$ are relative Poisson algebras,
and $\mu_{1},\rho_{1}:A_{1}\rightarrow\mathrm{End}_{\mathbb K}
(A_{2})$ and
$\mu_{2},\rho_{2}:A_{2}\rightarrow\mathrm{End}_{\mathbb K}
(A_{1})$ are linear maps.
 Then
$\big((A_{1},\cdot_{1},[-,-]_{1},P_{1}),(A_{2},\cdot_{2},[-,-]_{2},P_{2}),\mu_{1},\rho_{1},\mu_{2},\rho_{2}\big)$
is a matched pair of relative Poisson algebras if and only if the
following conditions are satisfied.
\begin{enumerate}
\item  $(\mu_{1},\rho_{1},P_{2},A_{2})$ is a representation of
$(A_{1},\cdot_{1},[-,-]_{1},P_{1})$. \item
$(\mu_{2},\rho_{2},P_{1},A_{1})$ is a representation of
$(A_{2},\cdot_{2},[-,-]_{2},P_{2})$. \item
$\big((A_{1},\cdot_{1}),(A_{2},\cdot_{2}),\mu_{1},\mu_{2}\big)$ is
a matched pair of commutative associative algebras. \item
$\big((A_{1},[-,-]_{1}),(A_{2},[-,-]_{2}),\rho_{1},\rho_{2}\big)$
is a matched pair of Lie algebras. \item  The following conditions
are satisfied: {\small
\begin{eqnarray}
\rho_{2}(a)(x\cdot_{1} y)+\mu_{2}\big(\rho_{1}(y)a\big)x-x\cdot_{1}\rho_{2}(a)y+\mu_{2}\big(\rho_{1}(x)a\big)y-y\cdot_{1}\rho_{2}(a)x-\mu_{2}\big(P_{2}(a)\big)(x\cdot_{1} y)&=&0,\ \ \ \ \ \ \label{eq:MP1}\\
\rho_{1}(x)(a\cdot_{2} b)+\mu_{1}\big(\rho_{2}(b)x\big)a-a\cdot_{2}\rho_{1}(x)b+\mu_{1}\big(\rho_{2}(a)x\big)b-b\cdot_{2}\rho_{1}(x)a-\mu_{1}\big(P_{1}(x)\big)(a\cdot_{2} b)&=&0,\ \ \ \ \ \ \label{eq:MP2}\\
\rho_{2}\big(\mu_{1}(x)a\big)y+[\mu_{2}(a)x,y]_{1}-x\cdot_{1}\rho_{2}(a)y+\mu_{2}\big(\rho_{1}(y)a\big)x-\mu_{2}(a)([x,y]_{1})+\mu_{2}(a)\big(x\cdot_{1} P_{1}(y)\big)&=&0,\ \ \ \ \ \ \label{eq:MP3}\\
\rho_{1}\big(\mu_{2}(a)x\big)b+[\mu_{1}(x)a,b]_{2}-a\cdot_{2}\rho_{1}(x)b+\mu_{1}\big(\rho_{2}(b)x\big)a-\mu_{1}(x)([a,b]_{2})+\mu_{1}(x)\big(a\cdot_{2}
P_{2}(b)\big)&=&0,\ \ \ \ \ \ \label{eq:MP4}
\end{eqnarray} }for all $x,y\in A_{1},a,b\in A_{2}$.
\end{enumerate}
\end{pro}

\begin{rmk}
In fact, there are also the similar equivalent characterizations
of matched pairs of commutative associative algebras, Lie algebras
and anti-pre-Lie algebras (see \cite{Bai2010}, \cite{Maj} and \cite{TPA} respectively). We will not list them since they are
not needed in the rest of this paper.
\end{rmk}

Now we introduce the notion of matched pairs of relative PCA
algebras.

\begin{defi}\label{defi:mp,rps}
Let $(A_{1},\cdot_{1},\circ_{1},P_{1},Q_{1})$ and
$(A_{2},\cdot_{2},\circ_{2},P_{2},Q_{2})$ be relative PCA
algebras, and
 $\mu_{1},l_{1},r_{1}:A_{1}\rightarrow\mathrm{End}_{\mathbb K} (A_{2})$ and $\mu_{2},l_{2},r_{2}:A_{2}\rightarrow\mathrm{End}_{\mathbb K} (A_{1})$ be linear maps.
We say $\big((A_{1}$, $\cdot_{1}$, $\circ_{1}$, $P_{1}$,
$Q_{1}),(A_{2},\cdot_{2},\circ_{2},P_{2},Q_{2}),\mu_{1},l_{1},
r_{1},\mu_{2},l_{2}, r_{2} \big)$ is a {\bf matched pair of
relative PCA algebras} if there is a relative PCA algebra
structure $(A_{1}\oplus
A_{2},\cdot,\circ,P_{1}+P_{2},Q_{1}+Q_{2})$ in which $\cdot$ and
$\circ$ are defined by \eqref{eq:mp,ca} and \eqref{eq:mp,APL}
respectively.
\end{defi}

\begin{pro}\label{pro:2.20}
Let $(A_{1},\cdot_{1},\circ_{1},P_{1},Q_{1})$ and
$(A_{2},\cdot_{2},\circ_{2},P_{2},Q_{2})$ be relative PCA
algebras, and
 $\mu_{1},l_{1},r_{1}:A_{1}\rightarrow\mathrm{End}_{\mathbb K} (A_{2})$ and $\mu_{2},l_{2},r_{2}:A_{2}\rightarrow\mathrm{End}_{\mathbb K} (A_{1})$ be linear maps.
Then $\big((A_{1}$, $\cdot_{1}$, $\circ_{1}$, $P_{1}$,
$Q_{1}),(A_{2},\cdot_{2},\circ_{2},P_{2},Q_{2}),\mu_{1},l_{1},
r_{1},\mu_{2},l_{2}, r_{2} \big)$ is a   matched pair of relative
PCA algebras if and only if the following conditions are
satisfied.
\begin{enumerate}
\item $( \mu_{1},l_{1}, r_{1}, P_{2}, Q_{2}, A_{2}  )$ is a
representation of  $(A_{1},\cdot_{1},\circ_{1},P_{1},Q_{1})$.

\item $( \mu_{2},l_{2}, r_{2}, P_{1}, Q_{1}, A_{1}  )$ is a
representation of $(A_{2},\cdot_{2},\circ_{2}$,
    $P_{2},Q_{2})$.
\item
$\big((A_{1},\cdot_{1},[-,-]_{1},P_{1}),(A_{2},\cdot_{2},[-,-]_{2},P_{2}),\mu_{1},l_{1}-r_{1},\mu_{2},l_{2}-r_{2}\big)$
is a matched pair of relative Poisson algebras. \item
$\big((A_{1},\cdot_{1} ),(A_{2},\cdot_{2} ),
l_{1},r_{1},l_{2},r_{2}\big)$ is a matched pair of anti-pre-Lie
algebras.

\item The following equations hold:
\begin{small}
 \begin{eqnarray*}
x\cdot_{1}r_{2}(a)y+\mu_{2}\big(l_{1}(y)a\big)x-y\circ_{1}\mu_{2}(a)x-r_{2}\big(\mu_{1}(x)a\big)y+\mu_{2}(a)[x,y]_{1}-\mu_{2}(a)\big(x\cdot_{1}P_{1}(y)\big)=0,&&\\
x\cdot_{1}l_{2}(a)y+\mu_{2}\big(r_{1}(y)a\big)x-l_{2}(a)(x\cdot_{1}y)+\mu_{2}\big((l_{1}-r_{1})(x)a\big)y+(r_{2}-l_{2})(a)x\cdot_{1}y-\mu_{2}\big(P_{2}(a)\big)(x\cdot_{1}y)=0,&&\\
\mu_{2}(a)(x\circ_{1}y)-x\circ_{1}\mu_{2}(a)y-r_{2}\big( \mu(y)a \big)x+(l_{2}-r_{2})(a)x\cdot_{1}y+\mu_{2}\big( (r_{1}-l_{1})(x)a \big)y-\mu_{2}(a)\big(P_{1}(x)\cdot_{1}y\big)=0,&&\\
r_{2}(a)(x\cdot_{1}y)-y\circ_{1}\mu_{2}(a)x-r_{2}\big( \mu_{1}(x)a\big)y-x\circ_{1}\mu_{2}(a)y-r_{2}\big( \mu_{1}(y)a \big)x+Q_{1}\big( \mu_{2}(a)(x\cdot_{1}y)\big)=0,&&\\
l_{2}\big( \mu_{1}(x)a\big) y+\mu_{2}(a)x\circ_{1}y-l_{2}(a)(x\cdot_{1}y)-x\circ_{1}\mu_{2}(a)y-r_{2}\big( \mu_{1}(y)a\big)x+Q_{1}\big( \mu_{2}(a)(x\cdot_{1}y)\big)=0,&&\\
a\cdot_{2}r_{1}(x)b+\mu_{1}\big(l_{2}(b)x\big)a-b\circ_{2}\mu_{1}(x)a-r_{1}\big(\mu_{2}(a)x\big)b+\mu_{1}(x)[a,b]_{2}-\mu_{1}(x)\big(a\cdot_{2}P_{2}(b)\big)=0,&&\\
a\cdot_{2}l_{1}(x)b+\mu_{1}\big(r_{2}(b)x\big)a-l_{1}(x)(a\cdot_{2}b)+\mu_{1}\big((l_{2}-r_{2})(a)x\big)b+(r_{1}-l_{1})(x)a\cdot_{2}b-\mu_{1}\big(P_{1}(x)\big)(a\cdot_{2}b)=0,&&\\
\mu_{1}(x)(a\circ_{2}b)-a\circ_{2}\mu_{1}(x)b-r_{1}\big( \mu(b)x \big)a+(l_{1}-r_{1})(x)a\cdot_{2}b+\mu_{1}\big( (r_{2}-l_{2})(a)x \big)b-\mu_{1}(x)\big(P_{2}(a)\cdot_{2}b\big)=0,&&\\
r_{1}(x)(a\cdot_{2}b)-b\circ_{2}\mu_{1}(x)a-r_{1}\big( \mu_{2}(a)x\big)b-a\circ_{2}\mu_{1}(x)b-r_{1}\big( \mu_{2}(b)x \big)a+Q_{2}\big( \mu_{1}(x)(a\cdot_{2}b)\big)=0,&&\\
l_{1}\big( \mu_{2}(a)x\big)
b+\mu_{1}(x)a\circ_{2}b-l_{1}(x)(a\cdot_{2}b)-a\circ_{2}\mu_{1}(x)b-r_{1}\big(
\mu_{2}(b)x\big)a+Q_{2}\big( \mu_{1}(x)(a\cdot_{2}b)\big)=0,&&
\end{eqnarray*}
\end{small}for all $x,y\in A_{1}, a,b\in A_{2}$.

\end{enumerate}

\end{pro}
\begin{proof}
    It follows from a straightforward checking.
\end{proof}

\section{Manin triples of relative Poisson algebras with respect to the \Witt forms and relative PCA bialgebras}\label{sec3}

In this section, we introduce the notions of Manin triples of
relative Poisson algebras with respect to the \Witt forms and
relative PCA bialgebras, and show that the two structures are equivalent
through specific matched pairs of relative Poisson
algebras.

\subsection{Manin triples of relative Poisson algebras with respect to the \Witt forms}

\begin{defi}
    Let $(A,\cdot_{A},[-,-]_{A},P)$ and $(A^{*},\cdot_{A^{*}},[-,-]_{A^{*}},Q^{*})$ be relative Poisson algebras.
    Suppose that there is a relative Poisson algebra structure $(A\oplus A^{*},\cdot,[-,-],P+Q^{*})$ on the direct sum $A\oplus A^{*}$ of vector
    spaces containing $(A,\cdot_{A},[-,-]_{A},P)$ and $(A^{*},\cdot_{A^{*}},[-,-]_{A^{*}},Q^{*})$ as relative Poisson subalgebras, and the bilinear form $\mathcal{B}_{d}$ on $A\oplus A^{*}$ given by \eqref{eq:Bd}
    is an \Witt form on $(A\oplus A^{*},\cdot,[-,-],P+Q^{*})$.
    Such a structure is called a {\bf Manin triple of relative Poisson algebras with respect to the \Witt form} and is denoted by $\big( (A\oplus A^{*},\cdot,[-,-],P+Q^{*}),(A,\cdot_{A},[-,-]_{A},P),(A^{*},\cdot_{A^{*}},[-,-]_{A^{*}},Q^{*}) \big)$.
\end{defi}

Recall the notion of double constructions of commutative
differential Frobenius algebras \cite{LLB2023}. Explicitly, let
$(A,\cdot_{A},P)$ and $(A^{*},\cdot_{A^{*}}, Q^*)$ be two
commutative differential algebras. Suppose that there is a
commutative associative algebra structure $(A\oplus A^{*},\cdot)$
on the direct sum $A\oplus A^{*}$ of vector spaces containing
$(A,\cdot_{A})$ and $(A^{*},\cdot_{A^{*}})$ as commutative
associative subalgebras, and the bilinear form $\mathcal{B}_{d}$
on $A\oplus A^{*}$ given by \eqref{eq:Bd}
    is invariant on $(A\oplus A^{*},\cdot)$.
    If $P+Q^{*}$ is a derivation of $(A\oplus A^{*},\cdot)$, then we say $\big( (A\oplus A^{*},\cdot,P+Q^{*}),(A,\cdot_{A},P),(A^{*},\cdot_{A^{*}},Q^{*}) \big)$ is
     a {\bf double construction  of a commutative differential Frobenius algebra}.

\begin{pro}\label{pro:3.2}
Let $\big( (A\oplus
A^{*},\cdot,P+Q^{*}),(A,\cdot_{A},P),(A^{*},\cdot_{A^{*}},Q^{*})
\big)$ be a  double construction of a commutative differential
Frobenius algebra. Let the Witt type Lie algebras of  $(A\oplus
A^{*},\cdot,P+Q^{*}),(A,\cdot_{A},P)$ and
$(A^{*},\cdot_{A^{*}},Q^{*})$  be $(A\oplus
A^{*},[-,-]),(A,[-,-]_{A})$ and $(A^{*},[-,-]_{A^{*}})$
respectively. Then  $\big( (A\oplus
A^{*},\cdot,[-,-],P+Q^{*}),(A,\cdot_{A},[-,-]_{A},P),(A^{*},\cdot_{A^{*}},[-,-]_{A^{*}}$,
$Q^{*}) \big)$ is a   Manin triple of relative Poisson algebras
with respect to the \Witt form.
\end{pro}
\begin{proof}
   By Proposition \ref{ex1.2}, $(A\oplus A^{*},\cdot,[-,-],P+Q^{*})$ is a relative Poisson algebra, and it contains $(A,\cdot_{A},[-,-]_{A},P)$ and $(A^{*},\cdot_{A^{*}},[-,-]_{A^{*}},Q^{*})$
   as relative Poisson subalgebras. Moreover, by Proposition \ref{pro2.9}, $\mathcal{B}_{d}$ is an \Witt form on $(A\oplus A^{*},\cdot,[-,-],P+Q^{*})$. Hence the conclusion follows.
\end{proof}

\begin{lem}
    Let $\big( (A\oplus A^{*},\cdot,[-,-],P+Q^{*}),(A,\cdot_{A},[-,-]_{A},P),(A^{*},\cdot_{A^{*}},[-,-]_{A^{*}},Q^{*}) \big)$ be a Manin triple of relative Poisson algebras with respect to the \Witt form. Then
     there is a compatible anti-pre-Lie algebra $(A\oplus A^{*},\circ)$ of $(A\oplus A^{*},[-,-])$ induced from $\mathcal{B}_{d}$ through \eqref{eq:APLc2c},
    which contains $(A,\circ_{A}=\circ|_{A\otimes A})$ and $(A^{*},\circ_{A^{*}}=\circ|_{A^{*}\otimes A^{*}})$ as anti-pre-Lie subalgebras.
   Moreover, $(A\oplus A^{*},\cdot,\circ,P+Q^{*},Q+P^{*})$ is a   relative PCA algebra structure  on $A\oplus A^{*}$, which contains
    $(A,\cdot_{A},\circ_{A},P,Q)$ and $(A^{*},\cdot_{A^{*}},\circ_{A^{*}},Q^{*},P^{*})$ as relative PCA subalgebras.
    \delete{
      The binary operations of $\cdot$ and $\circ$ on $A\oplus A^{*}$ satisfy
    \begin{eqnarray}
&& (x+a^{*})\cdot(y+b^{*})=x\cdot_{A}y-\mathcal{L}^{*}_{\cdot_{A^{*}}}(b^{*})x-\mathcal{L}^{*}_{\cdot_{A^{*}}}(a^{*})y
+a^{*}\cdot_{A^{*}}b^{*}-\mathcal{L}^{*}_{\cdot_{A }}(x^{*})b^{*}-\mathcal{L}^{*}_{\cdot_{A }}(y)a^{*},\ \ \ \ \label{eq:MT asso}\\
&& (x+a^{*})\circ(y+b^{*})=x\circ_{A}y-\mathcal{L}^{*}_{\circ_{A^{*}}}(b^{*})x-\mathcal{L}^{*}_{\circ_{A^{*}}}(a^{*})y
+a^{*}\circ_{A^{*}}b^{*}-\mathcal{L}^{*}_{\circ_{A }}(x^{*})b^{*}-\mathcal{L}^{*}_{\circ_{A }}(y)a^{*},\ \ \ \ \label{eq:MT APL}
    \end{eqnarray}
    for all $x,y\in A, a^{*},b^{*}\in A^{*}$. }
\end{lem}
\begin{proof}
The first half part follows from \cite[Lemma 2.9]{TPA}.
   By \cite[Lemma 3.4 (1)]{RPA}, the adjoint map $\widehat{P+Q^{*}}$ of $P+Q^{*}$ with respect to $\mathcal{B}_{d}$ is $Q+P^{*}$.
    Then by Proposition \ref{pro:from Witt to system}, $(A\oplus A^{*},\cdot,\circ,P+Q^{*},Q+P^{*})$ is a   relative PCA algebra, that is, the following equations hold:
    {\small \begin{eqnarray*}
&&(x+a^{*})\cdot(Q+P^{*})(y+b^{*})-(P+Q^{*})(x+a^{*})\cdot(y+b^{*})-(Q+P^{*})\big((x+a^{*})\cdot(y+b^{*})\big)=0,\\
&&(x+a^{*})\circ(Q+P^{*})(y+b^{*})-(P+Q^{*})(x+a^{*})\circ(y+b^{*})-(Q+P^{*})\big((x+a^{*})\circ (y+b^{*})\big)=0,\\
&&(x+a^{*})\cdot\big((y+b^{*})\circ(z+c^{*})\big)-(y+b^{*})\circ\big((x+a^{*})\cdot(z+c^{*})\big)+[x+a^{*},y+b^{*}]\cdot(z+c^{*})\\
&&\ \ -(x+a^{*})\cdot(P+Q^{*})(y+b^{*})\cdot(z+c^{*})=0,\\
&&\big((x+a^{*})\cdot(y+b^{*})\big)\circ(z+c^{*})-(y+b^{*})\circ\big((x+a^{*})\cdot(z+c^{*})\big)-(x+a^{*})\circ\big((y+b^{*})\cdot(z+c^{*})\big)\\
&&\ \ +(Q+P^{*})\big((x+a^{*})\cdot(y+b^{*})\cdot(z+c^{*})\big)=0,\;\forall x,y,z\in A, a^{*},b^{*},c^{*}\in A^{*}.
\end{eqnarray*}}Taking $a^{*}=b^{*}=c^{*}=0$ and $x=y=z=0$ respectively into the
above equations, we show that $(A,\cdot_{A},\circ_{A},P,Q)$ and
$(A^{*},\cdot_{A^{*}},\circ_{A^{*}},Q^{*},P^{*})$ are relative PCA
algebras.
\end{proof}

\begin{thm}\label{thm:equ 1}
    Let $(A,\cdot_{A},\circ_{A},P,Q)$ and $(A^{*},\cdot_{A^{*}},\circ_{A^{*}},Q^{*},P^{*})$ be relative PCA algebras, whose
  associated relative Poisson algebras are $(A,\cdot_{A},[-,-]_{A},P)$ and $(A^{*},\cdot_{A^{*}},[-,-]_{A^{*}},Q^{*})$ respectively.
    Then the following conditions are equivalent.
\begin{enumerate}
    \item\label{thm:equ 1,a} There is a Manin triple of relative Poisson algebras with respect to the \Witt form
    $\big( (A\oplus A^{*},\cdot,[-,-],P+Q^{*}),(A,\cdot_{A},[-,-]_{A},P),(A^{*},\cdot_{A^{*}},[-,-]_{A^{*}},Q^{*}) \big)$,
    such that the induced relative PCA algebra structure $(A\oplus A^{*},\cdot,\circ,P+Q^{*},Q+P^{*})$ on $A\oplus A^{*}$ from $\mathcal{B}_{d}$ contains $(A,\cdot_{A},\circ_{A},P,Q)$ and $(A^{*},\cdot_{A^{*}},\circ_{A^{*}},Q^{*},P^{*})$ as relative PCA subalgebras.
    \item\label{thm:equ 1,b} $\big( (A,\cdot_{A},\circ_{A},P,Q), (A^{*},\cdot_{A^{*}},\circ_{A^{*}},Q^{*},P^{*}),\; -\mathcal{L}^{*}_{\cdot_{A}},\;-\mathrm{ad}^{*}_{\circ_{A}},\;
    \mathcal{R}^{*}_{\circ_{A}},\;
    -\mathcal{L}^{*}_{\cdot_{A^{*}}},\;-\mathrm{ad}^{*}_{\circ_{A^{*}}},\; \mathcal{R}^{*}_{\circ_{A^{*}}}\big)$ is a matched pair of relative PCA algebras.
    \item\label{thm:equ 1,c}$\big((A ,\cdot_{A},[-,-]_{A},P ),(A^{*},\cdot_{A^{*}},[-,-]_{A^{*}},Q^{*}),-\mathcal{L}^{*}_{\cdot_{A}}$,
    $-\mathcal{L}^{*}_{\circ_{A}},
    -\mathcal{L}^{*}_{\cdot_{A^{*}}},-\mathcal{L}^{*}_{\circ_{A^{*}}}\big)$ is a matched pair of relative Poisson algebras.
\end{enumerate}
\end{thm}
\begin{proof}
Item (\ref{thm:equ 1,a}) $\Longrightarrow$ Item (\ref{thm:equ 1,b}). By Definition \ref{defi:mp,rps}, there is a matched pair of relative PCA algebras $\big((A,\cdot_{A},\circ_{A},P,Q), (A^{*},\cdot_{A^{*}},\circ_{A^{*}},Q^{*},P^{*}),\mu_{1},l_{1}, r_{1},\mu_{2},l_{2}, r_{2})$ giving rise to the relative PCA algebra structure $(A\oplus A^{*},\cdot,\circ,P+Q^{*},Q+P^{*})$ on $A\oplus A^{*}$.  Then $\big((A,\cdot_{A}), (A^{*},\cdot_{A^{*}}),\mu_{1},\mu_{2}\big)$ is a matched pair of commutative associative algebras, and
$\big((A, \circ_{A} ), (A^{*}$,
$\circ_{A^{*}} ), l_{1}, r_{1}, l_{2}, r_{2}\big)$ is a matched pair of anti-pre-Lie algebras.
    By \cite[Theorem 2.2.1]{Bai2010} and \cite[Theorem 2.10]{TPA} respectively,
    we have
    \begin{equation*}
    \mu_{1}=-\mathcal{L}^{*}_{\cdot_{A}},\;\mu_{2}=-\mathcal{L}^{*}_{\cdot_{A^{*}}},\;
        l_{1}=-\mathrm{ad}^{*}_{\circ_{A}},\; r_{1}=\mathcal{R}^{*}_{\circ_{A}},\;
     l_{2}=-\mathrm{ad}^{*}_{\circ_{A^{*}}},\; r_{2}=\mathcal{R}^{*}_{\circ_{A^{*}}},
    \end{equation*}
    that is, the binary operations $\cdot,\circ$ on $A\oplus A^{*}$ are given by
    \begin{eqnarray}
&& (x+a^{*})\cdot(y+b^{*})=x\cdot_{A}y-\mathcal{L}^{*}_{\cdot_{A^{*}}}(b^{*})x-\mathcal{L}^{*}_{\cdot_{A^{*}}}(a^{*})y
+a^{*}\cdot_{A^{*}}b^{*}-\mathcal{L}^{*}_{\cdot_{A }}(x^{*})b^{*}-\mathcal{L}^{*}_{\cdot_{A }}(y)a^{*},\ \ \ \ \label{eq:MT asso}\\
&& (x+a^{*})\circ(y+b^{*})=
x\circ_{A}y-\mathrm{ad}^{*}_{A^{*}}(a^{*})y+\mathcal{R}^{*}_{\circ_{A^{*}}}(b^{*})x
+a^{*}\circ_{A^{*}}b^{*}-\mathrm{ad}^{*}_{A }(x^{*})b^{*}+\mathcal{R}^{*}_{\circ_{A}}(y)a^{*},\ \ \ \ \label{eq:MT APL}
    \end{eqnarray}
  for all $x,y\in A, a^{*},b^{*}\in A^{*}$.  Therefore Item (\ref{thm:equ 1,b}) holds.

Item (\ref{thm:equ 1,b})$ \Longrightarrow$ Item (\ref{thm:equ 1,c}). It follows from by Proposition \ref{pro:2.20}.

Item (\ref{thm:equ 1,c}) $\Longrightarrow$ Item (\ref{thm:equ 1,a}).
Assume that  $\big((A ,\cdot_{A},[-,-]_{A},P
),(A^{*},\cdot_{A^{*}},[-,-]_{A^{*}},Q^{*}),-\mathcal{L}^{*}_{\cdot_{A}}$,
    $-\mathcal{L}^{*}_{\circ_{A}}$,
    $-\mathcal{L}^{*}_{\cdot_{A^{*}}}$,
    $-\mathcal{L}^{*}_{\circ_{A^{*}}}\big)$ is a matched pair of relative Poisson algebras.
    Then by Definition~\ref{defi:mp}, there is a relative Poisson algebra $(A\oplus A^{*},\cdot,[-,-],P+Q^{*})$ in which $\cdot$ is defined by \eqref{eq:MT asso} and $[-,-]$ is defined by
    \begin{eqnarray}
&& [x+a^{*},y+b^{*}]=[x,y]_{A}+\mathcal{L}^{*}_{\circ_{A^{*}}}(b^{*})x-\mathcal{L}^{*}_{\circ_{A^{*}}}(a^{*})y
+[a^{*},b^{*}]_{A^{*}}-\mathcal{L}^{*}_{\circ_{A }}(x^{*})b^{*}+\mathcal{L}^{*}_{\circ_{A }}(y)a^{*}.\ \ \ \ \label{eq:MT Lie}
    \end{eqnarray}
   Straightforwardly, $\mathcal{B}_{d}$ is an \Witt form on $(A\oplus A^{*},\cdot,[-,-],P+Q^{*})$  and the induced relative PCA algebra $(A\oplus A^{*},\cdot,\circ,P+Q^{*},Q+P^{*})$ from $\mathcal{B}_{d}$
    satisfies \eqref{eq:MT asso}-\eqref{eq:MT APL},
   and hence contains $(A,\cdot_{A},\circ_{A},P,Q)$ and $(A^{*},\cdot_{A^{*}},\circ_{A^{*}},Q^{*},P^{*})$ as relative PCA subalgebras.
    Therefore Item (\ref{thm:equ 1,a}) holds.
\end{proof}

\subsection{Relative PCA bialgebras}\

Let $A$ be a vector space.
If there is a linear map $\Delta:A\rightarrow A\otimes A$ such that the following equations hold:
\begin{equation*}
    \tau\Delta=\Delta,\;(\mathrm{id}\otimes\Delta) \Delta=(\Delta\otimes\mathrm{id}) \Delta,
\end{equation*}
where  $\tau(x\otimes y)=y\otimes x$ for all $x,y\in A$,
then we say $(A,\Delta)$ is a {\bf cocommutative coassociative coalgebra} \cite{Bai2010}.
On the other hand,
if there is a linear map $\theta:A\rightarrow A\otimes A$ such that the following equations hold:
\begin{eqnarray*}
&&(\mathrm{id}^{\otimes 3}-\tau\otimes\mathrm{id})(\mathrm{id}\otimes\theta)\theta=(\tau\otimes\mathrm{id}-\mathrm{id}^{\otimes 3})(\theta\otimes\mathrm{id})\theta,\\
&&( \mathrm{id}^{\otimes 3}+\xi+\xi^{2} )( \tau\otimes\mathrm{id}-  \mathrm{id}^{\otimes 3})(\theta\otimes\mathrm{id})\theta=0,
\end{eqnarray*}
where  $\xi(x\otimes y\otimes z)=y\otimes z\otimes x$ for all $x,y,z\in A$,
then we say $(A,\theta)$ is an {\bf anti-pre-Lie coalgebra} \cite{TPA}.
In this case, there is a Lie coalgebra $(A,\delta)$, where $\delta=\theta-\tau\theta$.

\begin{defi}
A {\bf relative PCA coalgebra} is a quintuple $(A,\Delta,\theta,Q,P)$, such that $(A,\Delta)$ is a cocommutative coassociative coalgebra, $(A,\theta)$ is an anti-pre-Lie coalgebra, $Q,P:A\rightarrow A$ are linear maps and the following equations hold:
\begin{eqnarray}
&&\Delta Q=(Q\otimes\mathrm{id}+\mathrm{id}\otimes Q)\Delta,\label{eq:cos1}\\
&& \delta Q=(Q\otimes \mathrm{id}+\mathrm{id}\otimes Q)\delta,\label{eq:cos2}\\
&&
 (\mathrm{id}\otimes\Delta)\delta(x)- (\delta\otimes
\mathrm{id} )\Delta(x)-(\tau\otimes\mathrm{id})(\mathrm{id}\otimes\delta)\Delta(x)-(Q\otimes
\mathrm{id}\otimes \mathrm{id})(\Delta\otimes
\mathrm{id})\Delta(x)=0,\label{eq:cos3}
 \\
&&\Delta P=(\mathrm{id}\otimes P-Q\otimes \mathrm{id})\Delta,\label{eq:cos4}\\
&&\theta P=(\mathrm{id}\otimes P-Q\otimes \mathrm{id})\theta,\label{eq:cos5}\\
&&(\mathrm{id}\otimes \theta)\Delta(x)-(\tau\otimes\mathrm{id})(\mathrm{id}\otimes\Delta)\theta(x)
+(\delta\otimes\mathrm{id})\Delta(x)-(\mathrm{id}\otimes Q\otimes\mathrm{id})(\Delta\otimes\mathrm{id})\Delta(x)    =0,\label{eq:cos6}\\
&&(\Delta\otimes\mathrm{id})\theta(x)=(\tau\otimes\mathrm{id})(\mathrm{id}\otimes\Delta)\theta(x)+(\mathrm{id}\otimes\Delta)\theta(x)-(\Delta\otimes\mathrm{id})\Delta\big(P(x)\big),\label{eq:cos7}
\end{eqnarray}
for all $x\in A$, where $\delta=\theta-\tau\theta$.
\end{defi}

\begin{pro}
Let $A$ be a vector space with linear maps
$\Delta,\theta:A\rightarrow A\otimes A$ and $Q,P:A\rightarrow A$.
Let $\cdot_{A^{*}},\circ_{A^{*}}:A^{*}\otimes A^{*}\rightarrow
A^{*}$ be the linear duals of $\Delta$ and $\theta$ respectively.
Then  $(A^{*},\cdot_{A^{*}},\circ_{A^{*}},Q^{*},P^{*})$ is a relative PCA algebra if and only if
$(A,\Delta,\theta,Q,P)$ is a relative PCA coalgebra.
\delete{, that is,
\begin{eqnarray*}
&&\langle\Delta(x), a^{*}\otimes b^{*}\rangle=\langle x,\Delta^{*}(a^{*}\otimes b^{*})\rangle=\langle x,a^{*}\cdot_{A^{*}}b^{*}\rangle,\\
&&\langle\theta(x), a^{*}\otimes b^{*}\rangle=\langle x,\theta^{*}(a^{*}\otimes b^{*})\rangle=\langle x,a^{*}\circ_{A^{*}}b^{*}\rangle,\;\;\forall x\in A, a^{*},b^{*}\in A^{*}.
\end{eqnarray*}}
\end{pro}
\begin{proof}
    By \cite{TPA},  $(A^{*},\circ_{A^{*}})$ is an anti-pre-Lie algebra if and only if $(A,\theta)$ is an anti-pre-Lie coalgebra.
    Set $[a^{*},b^{*}]_{A^{*}}=a^{*}\circ_{A^{*}}b^{*}-b^{*}\circ_{A^{*}} a^{*},\;\forall a^{*},b^{*}\in A^{*}$.
    Then $[-,-]_{A^{*}}$ is the linear dual of $\delta=\theta-\tau\theta$.
    By \cite{RPA}, $(A^{*},\cdot_{A^{*}},[-,-]_{A^{*}},Q^{*})$ is a relative Poisson algebra if and only if $(A,\Delta)$ is a cocommutative coassociative coalgebra, $(A,\delta)$ is a Lie coalgebra and \eqref{eq:cos1}-\eqref{eq:cos3} hold.
    Let $x\in A, a^{*},b^{*},c^{*}\in A^{*}$. Then we have
    \begin{eqnarray*}
    &&\langle a^{*}\cdot_{A^{*}}(b^{*}\circ_{A^{*}}c^{*}),x\rangle=\langle \Delta^{*}(\mathrm{id}\otimes\theta^{*})(a^{*}\otimes b^{*}\otimes c^{*}),x\rangle=\langle a^{*}\otimes b^{*}\otimes c^{*},(\mathrm{id}\otimes \theta)\Delta(x)\rangle,\\
    &&\langle b^{*}\circ_{A^{*}}(a^{*}\cdot_{A^{*}}c^{*}),x\rangle=\langle \theta^{*}(\mathrm{id}\otimes\Delta^{*})(\tau\otimes\mathrm{id})(a^{*}\otimes b^{*}\otimes c^{*}),x\rangle\\
    &&=\langle a^{*}\otimes b^{*}\otimes c^{*},(\tau\otimes\mathrm{id})(\mathrm{id}\otimes\Delta)\theta(x)\rangle,\\
    &&\langle [a^{*},b^{*}]_{A^{*}}\cdot_{A^{*}}c^{*},x\rangle=\langle \Delta^{*}(\delta^{*}\otimes\mathrm{id})(a^{*}\otimes b^{*}\otimes c^{*}),x\rangle=\langle a^{*}\otimes b^{*}\otimes c^{*},(\delta\otimes\mathrm{id})\Delta(x)\rangle,\\
    &&\langle a^{*}\cdot_{A^{*}}Q(b^{*})\cdot_{A^{*}}c^{*},x\rangle=\langle \Delta^{*}(\Delta^{*}\otimes\mathrm{id})(\mathrm{id}\otimes Q^{*}\otimes\mathrm{id})(a^{*}\otimes b^{*}\otimes c^{*}),x\rangle\\
    &&=\langle a^{*}\otimes b^{*}\otimes c^{*}, (\mathrm{id}\otimes Q\otimes\mathrm{id})(\Delta\otimes\mathrm{id})\Delta(x)\rangle.
    \end{eqnarray*}
    Hence \eqref{eq:cos6} holds if and only if the following equation holds:
    \begin{equation*}
        a^{*}\cdot_{A^{*}}(b^{*}\circ_{A^{*}}c^{*})-b^{*}\circ_{A^{*}}(a^{*}\cdot_{A^{*}}c^{*})+[a^{*},b^{*}]_{A^{*}}\cdot_{A^{*}}c^{*}-a^{*}\cdot_{A^{*}}Q(b^{*})\cdot_{A^{*}}c^{*}=0.
    \end{equation*}
    Similarly, \eqref{eq:cos4}, \eqref{eq:cos5} and \eqref{eq:cos7} hold if and only if the following equations hold respectively:
    \begin{eqnarray*}
        && a^{*}\cdot_{A^{*}}P^{*}(b^{*})-Q^{*}(a^{*})\cdot_{A^{*}}b^{*}-P^{*}(a^{*}\cdot_{A^{*}}b^{*})=0,\\
        &&a^{*}\circ_{A^{*}}P^{*}(b^{*})-Q^{*}(a^{*})\circ_{A^{*}}b^{*}-P^{*}(a^{*}\circ_{A^{*}}b^{*})=0,\\
        &&(a^{*}\cdot_{A^{*}}b^{*})\circ_{A^{*}}c^{*}-b^{*}\circ_{A^{*}}(a^{*}\cdot_{A^{*}}c^{*})-a^{*}\circ_{A^{*}}(b^{*}\cdot_{A^{*}}c^{*})+P^{*}(a^{*}\cdot_{A^{*}}b^{*}\cdot_{A^{*}}c^{*})=0.
    \end{eqnarray*}
    Hence $(A^{*},\cdot_{A^{*}},\circ_{A^{*}},Q^{*},P^{*})$ is a relative PCA algebra if and only if
$(A,\Delta,\theta,Q,P)$ is a relative PCA coalgebra.
\end{proof}
Recall that a {\bf commutative and cocommutative antisymmetric infinitesimal (ASI) bialgebra} \cite{Bai2010} is a triple $(A,\cdot,\Delta)$, such that
$(A,\cdot)$ is a commutative associative algebra,
$(A,\Delta)$ is a cocommutative coassociative coalgebra,
and the following equation  holds:
\begin{equation}\label{AssoBia}
\Delta(x\cdot y)=\big(\mathcal{L}_{\cdot}(x)\otimes \mathrm{id}\big)\Delta(y)+\big(\mathrm{id}\otimes\mathcal{L}_{\cdot}(y)\big)\Delta(x),\;\;\forall x,y\in A.
\end{equation}
An {\bf anti-pre-Lie bialgebra} \cite{TPA} is a triple $(A,\circ,\theta)$, such that $(A,\circ)$ is an anti-pre-Lie algebra,
$(A,\theta)$ is an anti-pre-Lie coalgebra,
and the following equations hold:
\begin{eqnarray}
     &&(\mathrm{id}^{\otimes 2}-\tau)\Big(\theta(x\circ y)-\big(\mathcal{L}_{\circ}(x)\otimes \mathrm{id}\big)\theta(y)-\big(\mathrm{id}\otimes\mathcal{L}_{\circ}(x)\big)\theta(y)+\big(\mathrm{id}\otimes\mathcal{R}_{\circ}(y)\big)\theta(x)\Big)=0,\label{eq:defi:anti-pre-Lie bialgebra1}\\
     &&\theta([x,y])=\big(\mathrm{id}\otimes\mathrm{ad}(x)-\mathcal{L}_{\circ}(x)\otimes \mathrm{id}\big)\theta(y)-\big(\mathrm{id}\otimes\mathrm{ad}(y)-\mathcal{L}_{\circ}(y)\otimes
        \mathrm{id}\big)\theta(x),\;\forall x,y\in A.\label{eq:defi:anti-pre-Lie bialgebra2}
\end{eqnarray}

\begin{defi}
    A {\bf relative PCA bialgebra} is a collection $(A,\cdot ,\circ ,\Delta,\theta,P,Q)$ such that
    the following conditions hold.
    \begin{enumerate}
        \item $(A,\cdot ,\circ ,$
    $P,Q)$ is a relative PCA algebra and  $(A,\Delta,\theta,Q,P)$ is a relative PCA coalgebra.
    \item $(A,\cdot,\Delta)$ is a commutative and cocommutative ASI bialgebra and $(A,\circ,\theta)$ is an anti-pre-Lie bialgebra.
    \item The following equations hold:
    {\small
     \begin{eqnarray}
    &&\begin{aligned}\label{b1}
&\theta(x\cdot y)+\big(\mathcal{L}_{\circ}(y)\otimes\mathrm{id}\big)\Delta(x)-\big(\mathrm{id}\otimes\mathcal{L}_{\cdot}(x)\big)\theta(y)\\
&+\big(\mathcal{L}_{\circ}(x)\otimes\mathrm{id}\big)\Delta(y)-\big(\mathrm{id}\otimes \mathcal{L}_{\cdot}(y)\big)\theta(x)-(Q\otimes\mathrm{id})\Delta(x\cdot y)=0,
\end{aligned}\\
&&\Delta(x\circ y)+\Big( (\tau+\mathrm{id}^{\otimes 2})\big(\mathrm{id}\otimes \mathcal{L}_{\cdot}(y)\big)\Big)\theta(x)-\big(\mathrm{id}\otimes\mathcal{L}_{\circ}(x)+\mathcal{L}_{\circ}(x)\otimes\mathrm{id}\big)\Delta(y)-\Delta\big(P(x)\cdot y\big)=0,\label{b2}\\
&& \big( \mathcal{L}_{\cdot}(x)\otimes\mathrm{id}-\mathrm{id}\otimes\mathcal{L}_{\cdot}(x)\big)\theta(y)+\big(\mathcal{L}_{\circ}(y)\otimes\mathrm{id}-\mathrm{id}\otimes\mathrm{ad}(y)\big)\Delta(x)-\Delta([x,y])+\Delta\big(x\cdot P(y)\big)=0,\label{b3}\\
&& \begin{aligned}\label{b4}
&\big( \mathrm{id}\otimes\mathcal{R}_{\circ}(y)\big)\Delta(x)
+\big(\mathcal{L}_{\cdot}(x)\otimes\mathrm{id}\big)\delta(y)
-\big(\mathrm{id}\otimes\mathcal{L}_{\circ}(x)\big)\Delta(y)\\
&+\tau\big(\mathrm{id}\otimes\mathcal{L}_{\cdot}(y)\big)\theta(x)
-\delta(x\cdot y)+(\mathrm{id}\otimes Q)\Delta(x\cdot y)=0,\;\forall x,y\in A.\end{aligned}
    \end{eqnarray}}
    \end{enumerate}
\end{defi}

\begin{thm}\label{thm:equ 2}
    Let $(A,\cdot_{A},\circ_{A},P,Q)$ and $(A^{*},\cdot_{A^{*}},\circ_{A^{*}},Q^{*},P^{*})$ be relative PCA algebras whose associated relative Poisson algebras are $(A,\cdot_{A},[-,-]_{A},P)$ and $(A^{*},\cdot_{A^{*}},[-,-]_{A^{*}},Q^{*})$ respectively.
    Let $\Delta,\theta:A\rightarrow A\otimes A$ be the linear duals of $\cdot_{A^{*}}$ and $\circ_{A^{*}}$ respectively.
    Then $(A,\cdot_{A},\circ_{A},\Delta,\theta,P,Q)$ is a relative PCA bialgebra if and only if
    $\big((A ,\cdot_{A},[-,-]_{A},P ),(A^{*},\cdot_{A^{*}},[-,-]_{A^{*}},Q^{*}),-\mathcal{L}^{*}_{\cdot_{A}}$,
    $-\mathcal{L}^{*}_{\circ_{A}},
    -\mathcal{L}^{*}_{\cdot_{A^{*}}}$, $-\mathcal{L}^{*}_{\circ_{A^{*}}}\big)$ is a matched pair of relative Poisson algebras.
\end{thm}
\begin{proof}
By \cite{Bai2010}, $(A,\cdot_{A},\Delta)$ is a commutative and cocommutative ASI  bialgebra if and only if
$\big( (A,\cdot_{A})$,
$(A^{*},\cdot_{A^{*}}),-\mathcal{L}^{*}_{\cdot_{A}}, -\mathcal{L}^{*}_{\cdot_{A^{*}}}  \big)$
is a matched pair of commutative associative algebras.
By \cite{TPA}, $(A,\circ_{A},\theta)$ is an anti-pre-Lie bialgebra if and only if $\big( (A,[-,-]_{A}),(A^{*},[-,-]_{A^{*}}),-\mathcal{L}^{*}_{\circ_{A}}, -\mathcal{L}^{*}_{\circ_{A^{*}}}  \big)$ is a matched pair of Lie algebras.
For all $x,y\in A, a^{*},b^{*}\in A^{*}$, we have
    \begin{eqnarray*}
-\langle\mathcal{L}^{*}_{\circ_{A^{*}}}(a^{*})(x\cdot_{A}y),b^{*}\rangle&=&\langle x\cdot_{A}y,a^{*}\circ_{A^{*}}b^{*}\rangle=\langle\theta(x\cdot_{A}y),a^{*}\otimes b^{*}\rangle,\\
\langle\mathcal{L}^{*}_{\cdot_{A^{*}}}\big(\mathcal{L}^{*}_{\circ_{A}}(y)a^{*}\big)x,b^{*}\rangle&=&-\langle x,\mathcal{L}^{*}_{\circ_{A}}(y)a^{*}\cdot_{A^{*}}b^{*}\rangle=\langle\big(\mathcal{L}_{\circ_{A}}(y)\otimes\mathrm{id}\big)\Delta(x),a^{*}\otimes b^{*}\rangle,\\
\langle x\cdot_{A}\mathcal{L}^{*}_{\circ_{A^{*}}}(a^{*})y,b^{*}\rangle&=&\langle y,a^{*}\circ_{A^{*}}\mathcal{L}^{*}_{\cdot_{A}}(x)b^{*}\rangle=-\langle\big(\mathrm{id}\otimes\mathcal{L}_{\cdot_{A}}(x)\big)\theta(y),a^{*}\otimes b^{*}\rangle,\\
\langle\mathcal{L}^{*}_{\cdot_{A^{*}}}\big(Q^{*}(a^{*})\big)(x\cdot_{A}y),b^{*}\rangle&=&-\langle x\cdot_{A}y,Q^{*}(a^{*})\cdot_{A^{*}}b^{*}\rangle=-\langle(Q\otimes\mathrm{id})\Delta(x\cdot_{A}y),a^{*}\otimes b^{*}\rangle.
\end{eqnarray*}
Hence \eqref{b1}  holds if and only  if \eqref{eq:MP1} holds for the case that
\begin{equation}\label{eq:assu}
    A_1=A,\;P_1=P,\;A_2=A^*,\;P_2=Q^*,\;\mu_1=-\mathcal{L}_{\cdot_{A}}^{*},\;\rho_1=-\mathcal{L}^{*}_{\circ_{A^{*}}},\;\mu_2=-\mathcal{L}_{\cdot_{A^*}}^{*},\;\rho_2=-\mathcal{L}_{\circ_{A^*}}^{*}.
\end{equation}
Similarly, \eqref{b2}-\eqref{b4}  hold  if and only  if
\eqref{eq:MP2}-\eqref{eq:MP4} hold  for the case given by
\eqref{eq:assu} respectively. Hence the conclusion follows.
\end{proof}

Combining Theorems \ref{thm:equ 1} and \ref{thm:equ 2} together, we have
\begin{cor}\label{cor:equ}
     Let $(A,\cdot_{A},\circ_{A},P,Q)$ and $(A^{*},\cdot_{A^{*}},\circ_{A^{*}},Q^{*},P^{*})$ be relative PCA algebras whose associated relative Poisson algebras are $(A,\cdot_{A},[-,-]_{A},P)$
     and $(A^{*},\cdot_{A^{*}},[-,-]_{A^{*}},Q^{*})$ respectively.
     Then the following statements are equivalent.
     \begin{enumerate}
         \item There is a Manin triple of relative Poisson algebras with respect to the \Witt form    $\big( (A\oplus A^{*},\cdot,[-,-],P+Q^{*}),(A,\cdot_{A},[-,-]_{A},P),(A^{*},\cdot_{A^{*}},[-,-]_{A^{*}},Q^{*}) \big)$,
    such that the induced relative PCA algebra structure $(A\oplus A^{*},\cdot,\circ,P+Q^{*},Q+P^{*})$ on $A\oplus A^{*}$ from $\mathcal{B}_{d}$ contains $(A,\cdot_{A},\circ_{A},P,Q)$ and $(A^{*},\cdot_{A^{*}},\circ_{A^{*}},Q^{*},P^{*})$ as relative PCA subalgebras.
    \item  $\big( (A,\cdot_{A},\circ_{A},P,Q), (A^{*},\cdot_{A^{*}},\circ_{A^{*}},Q^{*},P^{*}),\; -\mathcal{L}^{*}_{\cdot_{A}},\;-\mathrm{ad}^{*}_{\circ_{A}},\;
    \mathcal{R}^{*}_{\circ_{A}},\;
    -\mathcal{L}^{*}_{\cdot_{A^{*}}},\;-\mathrm{ad}^{*}_{\circ_{A^{*}}},\; \mathcal{R}^{*}_{\circ_{A^{*}}}\big)$ is a matched pair of relative PCA algebras.
    \item $\big((A ,\cdot_{A},[-,-]_{A},P ),(A^{*},\cdot_{A^{*}},[-,-]_{A^{*}},Q^{*}),-\mathcal{L}^{*}_{\cdot_{A}}$,
    $-\mathcal{L}^{*}_{\circ_{A}},
    -\mathcal{L}^{*}_{\cdot_{A^{*}}},-\mathcal{L}^{*}_{\circ_{A^{*}}}\big)$ is a matched pair of relative Poisson algebras.
    \item $(A,\cdot_{A},\circ_{A},\Delta,\theta,P,Q)$ is a relative PCA bialgebra, where
    $\Delta,\theta:A\rightarrow A\otimes A$ are the linear duals of $\cdot_{A^{*}}$ and $\circ_{A^{*}}$ respectively.
     \end{enumerate}
\end{cor}

Recall  that a {\bf cocommutative differential
coalgebra} \cite{HBG} is a triple $(A,\Delta,Q)$ such that $(A,\Delta)$ is  a
cocommutative coassociative coalgebra and $Q$ is a linear map
satisfying \eqref{eq:cos1}. An {\bf admissible cocommutative
differential coalgebra} is a tuple $(A,\Delta, Q, P)$ such that
$(A,\Delta, Q)$ is  a cocommutative differential coalgebra and $P$
is a linear map satisfying \eqref{eq:cos4}. Obviously, $(A,\Delta,
Q, P)$ is  an admissible cocommutative differential coalgebra if
and only if $(A^*,\Delta^*, Q^*,P^*)$ is an admissible commutative
differential algebra.

\begin{defi}{\rm \cite{HBG,LLB2023}}
A quintuple $(A,\cdot, \Delta,P,Q)$ is called a {\bf commutative
and cocommutative differential ASI bialgebra} if
\begin{enumerate}
    \item $(A,\cdot,\Delta)$ is a commutative
and cocommutative ASI bialgebra.
    \item  $(A,\cdot, P, Q)$ is an
    admissible commutative differential algebra.
    \item  $(A,\Delta, Q, P)$ is an admissible cocommutative differential coalgebra.
\end{enumerate}
\label{de:admdifbialg}
\end{defi}

\begin{lem} \cite{LLB2023}\label{lem:equ:diff}
Let $(A,\cdot_{A})$ and $(A^{*},\cdot_{A^{*}})$ be commutative
associative algebras, and $P,Q:A\rightarrow A$ be linear maps.
Suppose that $\Delta:A\rightarrow A\otimes A$ is the linear dual
of $\cdot_{A^{*}}$. Then there is a double construction of a
commutative differential  Frobenius algebra $\big( (A\oplus
A^{*},\cdot,P+Q^{*}),(A,\cdot_{A},P),(A^{*},\cdot_{A^{*}},Q^{*})
\big)$ if and only if $(A,\cdot_{A},\Delta,P,Q)$ is a commutative
and cocommutative differential ASI bialgebra.
\end{lem}

\begin{pro}\label{pro:equ:diff}
Let $(A,\cdot_{A},\Delta,P,Q)$ be a commutative and cocommutative
differential  ASI  bialgebra. Let $(A,\circ_{A})$ be the
anti-pre-Lie algebra given by \eqref{eq:ex}, and
$\theta:A\rightarrow A\otimes A$ be a linear map given by
\begin{equation}\label{eq:pro:equ:diff}
    \theta(x)=\Delta\big(P(x)\big)-(Q\otimes\mathrm{id})\Delta(x),\;\forall x\in A.
\end{equation}
Then $(A,\cdot_{A},\circ_{A},\Delta,\theta,P,Q)$ is a relative PCA bialgebra.
\end{pro}

\begin{proof}
It follows from a straightforward verification or can be verified as
follows. Let the linear duals of $\Delta$ and $\theta$ be
$\cdot_{A^{*}}$ and $\circ_{A^{*}}$ respectively. Then
$(A^{*},\cdot_{A^{*}}, Q^*)$ is a commutative differential
algebra, and by \eqref{eq:pro:equ:diff} we have
\begin{equation}
    a^{*}\circ_{A^{*}}b^{*}=P^{*}(a^{*}\cdot_{A^{*}}b^{*})-Q^{*}(a^{*})\cdot_{A^{*}}b^{*},\;\forall a^{*},b^{*}\in A^{*}.
\end{equation}
By Proposition \ref{ex:ex}, both $(A,\cdot_{A},\circ_{A},P,Q)$ and
$(A^{*},\cdot_{A^{*}},\circ_{A^{*}},Q^{*},P^{*})$ are relative PCA
algebras. The corresponding double construction of a commutative
differential  Frobenius algebra $\big( (A\oplus
A^{*},\cdot,P+Q^{*}),(A,\cdot_{A},P),(A^{*},\cdot_{A^{*}},Q^{*})
\big)$ gives rise to a Manin triple of relative Poisson algebras
with respect to the \Witt form $\big( (A\oplus
A^{*},\cdot,[-,-],P+Q^{*}),(A,\cdot_{A},[-,-]_{A},P),(A^{*},\cdot_{A^{*}},[-,-]_{A^{*}},Q^{*})
\big)$  by Proposition \ref{pro:3.2}. Hence by Corollary
\ref{cor:equ}, $(A,\cdot_{A},\circ_{A},\Delta,\theta,P,Q)$ is a
relative PCA bialgebra.
\end{proof}

Therefore Proposition \ref{pro:equ:diff} shows that commutative
and cocommutative differential  ASI  bialgebras naturally give
rise to relative PCA bialgebras, instead of relative Poisson
bialgebras introduced in \cite{RPA}.

\section{Coboundary relative PCA bialgebras, relative PCA Yang-Baxter equations and $\mathcal{O}$-operators}\label{sec4}

In this section, we study the coboundary cases of relative PCA
bialgebras. We introduce the notion of the relative PCA
Yang-Baxter equation  (RPCA-YBE) in a relative PCA algebra, and
show that
  antisymmetric solutions of the RPCA-YBE  give rise to coboundary relative PCA bialgebras.
  Furthermore, we introduce the notions of $\mathcal{O}$-operators of  relative PCA algebras and relative pre-PCA algebras, which give antisymmetric solutions of the RPCA-YBE in semi-direct product relative PCA algebras.

\subsection{Coboundary relative PCA bialgebras and relative PCA Yang-Baxter equations}\

Recall that  a commutative and cocommutative ASI
bialgebra $(A,\cdot,\Delta)$ is called \textbf{coboundary} \cite{Bai2010} if
there exists an $r\in A\otimes A$ such that
\begin{equation}\label{eq:AssoCob}
\Delta(x):=\Delta_{r}(x):=\big(\mathrm{id}\otimes\mathcal{L}_{\cdot}(x)-\mathcal{L}_{\cdot}(x)\otimes \mathrm{id}\big)r,\;\forall x\in A.
\end{equation}


Let $(A,\cdot)$ be a commutative associative algebra, and
$\Delta:A\rightarrow A\otimes A$ be a linear map defined by
 \eqref{eq:AssoCob}. By \cite{Bai2010},  $(A,\cdot,\Delta)$ is a commutative and cocommutative
ASI bialgebra if and only if  the following equations hold:
\begin{eqnarray}
&& \big(\mathrm{id}\otimes\mathcal{L}_{\cdot}(x)-\mathcal{L}_{\cdot}(x)\otimes \mathrm{id}\big)\big(r+\tau(r)\big)=0,\label{eq:ASI 1}\\
&&\big(\mathrm{id}\otimes \mathrm{id}\otimes\mathcal{L}_{\cdot}(x)-\mathcal{L}_{\cdot}(x)\otimes \mathrm{id}\otimes \mathrm{id}\big)\textbf{A}(r)=0, \;\forall x\in A,\label{eq:ASI 2}
\end{eqnarray}
where
\begin{eqnarray*}
&&\textbf{A}(r)=r_{12}\cdot r_{13}-r_{23}\cdot r_{12}+ r_{13}\cdot r_{23},\;\;{\rm and}\;\;{\rm for}\;\;r=\sum_i a_i\otimes b_i,\\
&&r_{12}\cdot r_{13}=\sum_{i,j} a_{i}\cdot a_{j}\otimes b_{i}\otimes b_{j}, r_{23}\cdot r_{12}=\sum_{i,j}a_{i}\otimes a_{j}\cdot b_{i}\otimes b_{j}, r_{13}\cdot r_{23}=\sum_{i,j}a_{i}\otimes a_{j}\otimes b_{i}\cdot b_{j}.
\end{eqnarray*}
The equation $\textbf{A}(r)=0$ is called the \textbf{associative
Yang-Baxter equation (AYBE) in $(A,\cdot)$}.

An anti-pre-Lie bialgebra $(A,\circ,\theta)$ is called \textbf{coboundary} \cite{TPA} if there exists an $r\in A\otimes A$ such that
\begin{equation}\label{eq:aplCob}
\theta(x):=\theta_{r}(x):=\big(-\mathcal{L}_{\circ}(x)\otimes \mathrm{id}+\mathrm{id}\otimes\mathrm{ad}(x)\big)r, \;\forall x\in A.
\end{equation}

Let $(A,\circ)$ be an anti-pre-Lie algebra, and
$\theta:A\rightarrow A\otimes A$ be a linear map defined by
 \eqref{eq:aplCob}. By \cite{TPA},  $(A,\circ,\theta)$ is an anti-pre-Lie  bialgebra if and only if the following equations
 hold:
    \begin{small}
\begin{equation}\label{eq:Apl 1}
        \begin{split}
            &\Big(\mathcal{L}_{\circ }(x)\otimes\mathrm{id}\otimes\mathrm{id}-(\tau\otimes\mathrm{id})\big(\mathcal{L}_{\circ }(x)\otimes\mathrm{id}\otimes\mathrm{id}\big)-\mathrm{id}\otimes\mathrm{id}\otimes\mathrm{ad}(x)\Big)\textbf{T}(r)\\
            &+\sum_{j}\big(\mathrm{id}\otimes\mathcal{L}_{\circ}(a_{j})\otimes\mathrm{ad}(x)-\mathrm{ad}(a_{j})\otimes\mathrm{id}\otimes\mathrm{ad}(x)\big)\Big(\big(r+\tau(r)\big)\otimes b_{j}\Big)\\
            &+(\mathrm{id}^{\otimes 3}-\tau\otimes\mathrm{id})\sum_{j}\big(\mathcal{L}_{\circ}(x\circ a_{j})\otimes\mathrm{id}\otimes\mathrm{id}-\mathcal{L}_{\circ}(x)\mathcal{R}_{\circ}(a_{j})\otimes\mathrm{id}\otimes\mathrm{id}\big)\Big(\big(r+\tau(r)\big)\otimes b_{j}\Big)=0,
        \end{split}
    \end{equation}
\begin{equation}\label{eq:Apl 2}
    \begin{split}
        &(\mathrm{id}^{\otimes 3}+\xi+\xi^{2})\bigg(-\big(\mathrm{id}\otimes\mathrm{id}\otimes\mathcal{L}_{\circ}(x)\big)\textbf{T}(r)+\sum_{j}\big(\mathrm{id}\otimes\mathrm{id}\otimes\mathcal{L}_{\circ}([x,b_{j}])\big)(\mathrm{id}^{\otimes 3}-\tau\otimes\mathrm{id})\Big(a_{j}\otimes \big(r+\tau(r)\big)\Big)\\
        &+\sum_{j}\big(\mathrm{id}\otimes\mathrm{ad}(b_{j})\otimes \mathcal{L}_{\circ}(x)\big)\Big(a_{j}\otimes \big(r+\tau(r)\big)\Big)+\sum_{j}\big(\mathrm{id}\otimes\mathrm{ad}(a_{j})\otimes \mathcal{L}_{\circ}(x)\big)\Big(b_{j}\otimes \big(r+\tau(r)\big)\Big)\\
        &+\sum_{j}\big(\mathcal{L}_{\circ}(a_{j})\otimes\mathrm{id}\otimes \mathcal{L}_{\circ}(x)\big)(\tau\otimes\mathrm{id})\Big(b_{j}\otimes \big(r+\tau(r)\big)\Big)+\sum_{j}\big(\mathrm{ad}(b_{j})\otimes \mathrm{id}\otimes\mathcal{L}_{\circ}(x)\big)
        \Big(\big(r+\tau(r)\big)\otimes a_{j}\Big)\bigg)=0,
    \end{split}
\end{equation}
    \begin{equation}\label{eq:Apl 3}
        \big(\mathrm{id}\otimes\mathcal{L}_{\circ}(x\circ y)-\mathrm{id}\otimes\mathcal{L}_{\circ}(x)\mathcal{L}_{\circ}(y)+\mathcal{L}_{\circ}(x)\mathcal{L}_{\circ}(y)\otimes \mathrm{id}-\mathcal{L}_{\circ}(x\circ y)\otimes \mathrm{id}+\mathcal{L}_{\circ}(y)\otimes\mathcal{L}_{\circ}(x)-\mathcal{L}_{\circ}(x)\otimes\mathcal{L}_{\circ}(y)\big)\big(r+\tau(r)\big)=0,
    \end{equation}
\end{small}for all $x,y\in A$, where
\begin{eqnarray*}\label{eq:TYBE}
&&\textbf{T}(r)=r_{12}\circ r_{13}+ r_{12}\circ r_{23}-
[r_{13},r_{23}],\;\;{\rm and}\;\;{\rm for}\;\;r=\sum_i a_i\otimes
b_i,\\ && r_{12}\circ r_{13}=\sum_{i,j} a_{i}\circ a_{j}\otimes
b_{i}\otimes b_{j}, r_{12}\circ r_{23}=\sum_{i,j}a_{i}\otimes
b_{i}\circ a_{j}\otimes b_{j},
[r_{13},r_{23}]=\sum_{i,j}a_{i}\otimes a_{j}\otimes[b_{i},b_{j}].
\end{eqnarray*}
The equation $\textbf{T}(r)=0$ is called the \textbf{anti-pre-Lie
Yang-Baxter equation (APL-YBE) in $(A,\circ)$}.


\begin{pro}\label{pro:discussion}
Let $(A,\cdot_{A},\Delta,P,Q)$ be a commutative and cocommutative
differential  ASI bialgebra and the corresponding relative PCA
bialgebra be given by Proposition \ref{pro:equ:diff}.
 If there is an $r\in A\otimes A$ such that \eqref{eq:AssoCob} and the following equation hold:
    \begin{equation}\label{eq:pro:d2}
        (Q\otimes \mathrm{id}-\mathrm{id}\otimes P)r=0,
    \end{equation}
    then we have
   \begin{equation}\label{eq:aplCob2}
\theta(x)=\big(\mathcal{L}_{\circ}(x)\otimes \mathrm{id}-\mathrm{id}\otimes\mathrm{ad}(x)\big)r, \;\forall x\in A.
\end{equation}
\delete{
    Let $(A,\cdot_{A},\circ_{A},\Delta,\theta,P,Q)$ be a relative PCA bialgebra and the corresponding relative PCA algebra structure on $A\oplus A^{*}$ be denoted by $(A\oplus A^{*},\cdot,\circ,P+Q^{*},Q+P^{*})$.
    Suppose that the anti-pre-Lie algebra $(A\oplus A^{*},\circ)$ is given by Example \ref{ex:ex}, that is,
    \begin{equation}\label{eq:pro:d1}
        (x+a^{*})\circ (y+b^{*})=(Q+P^{*})\big( (x+a^{*})\cdot(y+b^{*}) \big)-(P+Q^{*})(x+a^{*})\cdot (y+b^{*}),\;\forall x,y\in A, a^{*},b^{*}\in A^{*}.
    \end{equation} }
\end{pro}
\begin{proof} Set $r=\sum\limits_{i}a_i\otimes b_i$. Let $x\in A$. Then we have
        \begin{eqnarray*}
\theta(x)&=&\Delta\big(P(x)\big)-(Q\otimes \mathrm{id})\Delta(x)\\
&=&\sum_{i} a_{i}\otimes P(x)\cdot b_{i}-P(x)\cdot a_{i}\otimes b_{i}-Q(a_{i})\otimes x\cdot b_{i}+Q(x\cdot a_{i})\otimes b_{i}\\
&\overset{\eqref{eq:pro:d2}}{=}&\sum_{i}   a_{i}\otimes P(x)\cdot b_{i}-P(x)\cdot a_{i}\otimes b_{i}-a_{i}\otimes x\cdot P(b_{i})+Q(x\cdot a_{i})\otimes b_{i}\\
&\overset{\eqref{eq:ex}}{=}&\sum_{i} x\circ a_{i}\otimes b_{i}-a_{i}\otimes[x,b_{i}] \\
&=&\big(\mathcal{L}_{\circ}(x)\otimes\mathrm{id}-\mathrm{id}\otimes\mathrm{ad}(x)\big)r.
\end{eqnarray*}
Hence the conclusion follows.
\end{proof}

Therefore we are motivated to give the following notion.
\begin{defi}
A relative PCA bialgebra $(A,\cdot,\circ,\Delta,\theta,P,Q)$ is called \textbf{coboundary} if there exists an $r\in A\otimes A$ such that  \eqref{eq:AssoCob}  and  \eqref{eq:aplCob2}  hold.
\end{defi}

\begin{pro}\label{pro:cob cos}
Let $(A,\cdot,\circ,P,Q)$ be a relative PCA algebra. Let
$r=\sum\limits_{i}a_{i}\otimes b_{i}\in A\otimes A$ and
$\Delta,\theta:A\rightarrow A\otimes A$ be linear maps defined by
\eqref{eq:AssoCob}  and  \eqref{eq:aplCob2} respectively.
\begin{enumerate}
\item \eqref{eq:cos1} holds if and only if the following equation holds:
\begin{equation}\label{eq:pro:cob cos1}
\big(\mathrm{id}\otimes\mathcal{L}_{\cdot}(x)\big)(\mathrm{id}\otimes P-Q\otimes \mathrm{id})r+\big(\mathcal{L}_{\cdot}(x)\otimes \mathrm{id}\big)(\mathrm{id}\otimes Q-P\otimes \mathrm{id})r=0,\;\;\forall x\in A.
\end{equation}
    \item\label{pro:cob cos2}  \eqref{eq:cos2} holds if and only if the following equation holds:
\begin{equation}\label{eq:pro:cob cos2}
\begin{split}
&(\mathrm{id}^{\otimes 2}-\tau)\Big(\mathcal{L}_{\circ}\big((P+Q)x\big)\otimes\mathrm{id}-\mathcal{L}_{\circ}(x)Q\otimes\mathrm{id}
-\mathcal{L}_{\circ}(x)\otimes Q\Big)\big(r+\tau(r)\big)\\
&\ \ +(\mathrm{id}^{\otimes 2}-\tau)\big(\mathcal{R}_{\circ}(x)\otimes\mathrm{id}\big)\tau(Q\otimes\mathrm{id}-\mathrm{id}\otimes P)r=0,\;\;\forall x\in A.
\end{split}
\end{equation}
\item\label{pro:cob cos3} \eqref{eq:cos3} holds if and only if the following equation holds:
{\small
\begin{equation}\label{eq:pro:cob cos3}
\begin{split}
&\big(\mathcal{L}_{\circ}(x)\otimes\mathrm{id}\otimes\mathrm{id}+Q\otimes\mathrm{id}\otimes\mathcal{L}_{\cdot}(x)-(\mathrm{ad}(x)\otimes\mathrm{id}\otimes\mathrm{id})\xi^{2}\big){\bf A}(r)\\
&-\big((\mathrm{id}\otimes\mathcal{L}_{\cdot}(x)\otimes\mathrm{id})(\mathrm{id}^{\otimes 3}+\xi^{2})+\mathrm{id}\otimes\mathrm{id}\otimes\mathcal{L}_{\cdot}(x)\big){\bf T}(r)\\
&+\big(P\mathcal{L}_{\cdot}(x)\otimes\mathrm{id}\otimes\mathrm{id}\big)(r_{21}\cdot r_{31}-r_{12}\cdot r_{13})\\
&+\big(\mathrm{id}\otimes\mathcal{L}_{\cdot}(x)\otimes\mathrm{id}\big)(r_{21}\circ r_{13}+r_{12}\circ r_{13}+r_{23}\circ r_{21}+r_{23}\circ r_{12})\\
&+\sum_{j}\bigg( -\big(\mathrm{ad}(b_{j})\mathcal{L}_{\cdot}(x)\otimes\mathrm{id}\otimes\mathrm{id}\big)(\mathrm{id}\otimes\tau)\Big(\big(r+\tau(r)\big)\otimes a_{j}\Big)\\
&\ \ +\Big(\mathrm{id}\otimes\mathrm{id}\otimes\big(\mathcal{L}_{\circ}(x\cdot b_{j})-\mathcal{L}_{\cdot}(b_{j})\mathcal{L}_{\circ}(x)\big)\Big)(\tau\otimes\mathrm{id})\Big(a_{j}\otimes\big(r+\tau(r)\big)\Big)\bigg)\\
&\ \ -(\tau\otimes\mathrm{id})\big(\mathrm{id}\otimes\mathrm{id}\otimes\mathcal{L}_{\cdot}(x\cdot b_{j})\big)\big(a_{j}\otimes(Q\otimes\mathrm{id}-\mathrm{id}\otimes P)r\big)\\
&\ \ +\big(\mathrm{id}\otimes \mathcal{L}_{\cdot}(x\cdot a_{j}) \otimes \mathrm{id}\big)\big((Q\otimes\mathrm{id}-\mathrm{id}\otimes P)r\otimes b_{j}\big)\\
&\ \ +\big(\mathrm{id}\otimes\mathcal{L}_{\circ}(a_{j})\otimes\mathcal{L}_{\cdot}(x)-\mathrm{ad}(a_{j})\otimes\mathrm{id}\otimes\mathcal{L}_{\cdot}(x)- \mathcal{R}_{\circ}(a_{j})\mathcal{L}_{\cdot}(x)\otimes\mathrm{id}\otimes\mathrm{id} \big)
\Big(\big(r+\tau(r)\big)\otimes b_{j}\Big)\\
&\ \ +\Big(\mathrm{id}\otimes\big(\mathcal{L}_{\cdot}(a_{j})\mathcal{L}_{\circ}(x)-\mathcal{L}_{\circ}(x\cdot a_{j})\big)\otimes \mathrm{id}\Big)\Big(\big(r+\tau(r)\big)\otimes b_{j}\Big)\bigg)=0,\;\forall x\in A.\\
\end{split}
\end{equation}}where
\begin{eqnarray*}
&&r_{21}\cdot r_{31}=\sum_{i,j} b_{i}\cdot b_{j}\otimes a_{i}\otimes a_{j},\; r_{21}\circ r_{13}=\sum_{i,j} b_{i}\circ a_{j}\otimes a_{i}\otimes b_{j},\\
&& r_{23}\circ r_{21}=\sum_{i,j} b_{j}\otimes a_{i}\circ a_{j}\otimes b_{i}, r_{23}\circ r_{12}=\sum_{i,j} a_{j}\otimes a_{i}\circ b_{j}\otimes b_{i}.
\end{eqnarray*}
\item \eqref{eq:cos4} holds if and only if the following equation holds:
\begin{equation}\label{eq:pro:cob cos4}
\big(\mathrm{id}\otimes\mathcal{L}_{\cdot}(x)- \mathcal{L}_{\cdot}(x)\otimes \mathrm{id}\big)(\mathrm{id}\otimes Q-P\otimes \mathrm{id})r=0,\;\;\forall x\in A.
\end{equation}
\item \eqref{eq:cos5} holds if and only if the following equation holds:
\begin{equation}\label{eq:pro:cob cos5}
\big(\mathcal{L}_{\circ}(x)\otimes\mathrm{id}-\mathrm{id}\otimes\mathrm{ad}(x)\big)(Q\otimes\mathrm{id}-\mathrm{id}\otimes P)r=0,\;\;\forall x\in A.
\end{equation}
\item \eqref{eq:cos6} holds if and only if the following equation holds:
\begin{equation}\label{eq:pro:cob cos6}
\begin{split}
&\Big(\mathrm{id}\otimes Q\otimes\mathcal{L}_{\cdot}(x)-\big(\mathrm{id}\otimes\mathcal{L}_{\circ}(x)\otimes\mathrm{id}\big)(\tau\otimes\mathrm{id})\Big){\bf A}(r)\\
&+\big(\mathrm{id}\otimes\mathrm{id}\otimes\mathcal{L}_{\cdot}(x)-\mathcal{L}_{\cdot}(x)\otimes\mathrm{id}\otimes\mathrm{id}\big){\bf T}(r)\\
&+\sum_{j}\bigg( -\big( \mathcal{L}_{\cdot}(a_{j})\mathrm{ad}(x)+\mathrm{ad}(x\cdot a_{j})\big)\otimes\mathrm{id}\otimes\mathrm{id})\Big( \big(r+\tau(r)\big)\otimes b_{j}\Big)\\
&+\Big( \mathrm{id}\otimes \big( \mathcal{L}_{\circ}(x)\mathcal{L}_{\cdot}(a_{j})+\mathcal{L}_{\circ}(x\cdot a_{j}) \big)\otimes\mathrm{id} \Big)\Big( \big(r+\tau(r)\big)\otimes b_{j}\Big)\\
&+\big(\mathrm{ad}(a_{j})\otimes\mathrm{id}\otimes\mathcal{L}_{\cdot}(x)-\mathrm{id}\otimes\mathcal{L}_{\circ}(a_{j})\otimes\mathcal{L}_{\cdot}(x)\big)
\Big(\big(r+\tau(r)\big)\otimes b_{j}) \Big) \\
&-\big( \mathcal{L}_{\cdot}(x\cdot a_{j})\otimes\mathrm{id}\otimes\mathrm{id}\big)\big((\mathrm{id}\otimes Q-P\otimes\mathrm{id})r\otimes b_{j}\big)\\
&+\big(\mathrm{id}\otimes \mathrm{id}\otimes \mathcal{L}_{\cdot}(x\cdot b_{j})\big)\big( a_{j}\otimes (\mathrm{id}\otimes P-Q\otimes\mathrm{id})r\big) \bigg)=0,\;\forall x\in A.\\
\end{split}
\end{equation}
\item \eqref{eq:cos7} holds if and only if the following equation holds:
{\small
\begin{equation}\label{eq:pro:cob cos7}
\begin{split}
&\Big(\mathrm{id}\otimes\mathrm{id}\otimes\mathrm{ad}(x)-\mathrm{id}\otimes\mathrm{id}\otimes \mathcal{L}_{\cdot}\big(P(x)\big)-\mathcal{L}_{\circ}(x)\otimes\mathrm{id}\otimes\mathrm{id}-\big(\mathrm{id}\otimes\mathcal{L}_{\circ}(x)\otimes\mathrm{id}\big)(\tau\otimes\mathrm{id})\Big){\bf A}(r)\\
&+\sum_{j}\big( \mathrm{id}\otimes\mathcal{L}_{\circ}(x)\mathcal{L}_{\cdot}(a_{j})\otimes\mathrm{id}- \mathcal{L}_{\cdot}(a_{j})\mathrm{ad}(x)\otimes\mathrm{id}\otimes\mathrm{id}  \big)\Big( \big(r+\tau(r)\big)\otimes b_{j} \Big)=0,\;\;\forall x\in A.
\end{split}
\end{equation}}
\end{enumerate}
\end{pro}
\begin{proof}
We only give the proof of  Item (\ref{pro:cob cos3}) as an explicit example, and other items are proved similarly.
Let $x,y,z\in A$. By \eqref{eq:GLR} and \eqref{eq:rps4} respectively, we have
\begin{eqnarray}
    &&P(x\cdot y\cdot z)=[x,y\cdot z]+[y,z\cdot x]+[z,x\cdot y],\label{eq:pf1}\\
    &&(x\cdot y)\circ z+z\circ(x\cdot y)=y\circ(x\cdot z)+(x\cdot z)\circ y.\label{eq:pf2}
\end{eqnarray}
Hence we have
\begin{eqnarray*}
&&2Q(x\cdot y\cdot z)-P(x\cdot y\cdot z)-(x\cdot y)\circ z-z\circ(x\cdot y)\\
&&\overset{\eqref{eq:rps4},\eqref{eq:pf1}}{=}2x\circ(y\cdot z)+2y\circ(x\cdot z)-2(x\cdot y)\circ z\\
&&\ \ -[z,x\cdot y]-[x,y\cdot z]-[y,z\cdot x]-(x\cdot y)\circ z-z\circ(x\cdot y)\\
&&=x\circ(y\cdot z)+y\circ(z\cdot x)-2(x\cdot y)\circ z-2z\circ(x\cdot y)+(y\cdot z)\circ x+(z\cdot x)\circ y\\
&&\overset{\eqref{eq:pf2}}{=}0,
\end{eqnarray*}
that is,
\begin{equation}\label{eq:pf3}
    (2Q-P)(x\cdot y\cdot z)=(x\cdot y)\circ z+z\circ(x\cdot y).
\end{equation}
Moreover, by \eqref{eq:GLR} and \eqref{eq:rps3}, we have
\begin{equation}\label{eq:pf4}
    2x\cdot P(y)\cdot z=2[x,y]\cdot z+x\cdot (z\circ y)-(x\cdot z)\circ y.
\end{equation}
Now we have
{\small
\begin{eqnarray*}
&&(\mathrm{id}\otimes\Delta)\delta(x)-(\delta\otimes
\mathrm{id})\Delta(x) -(\tau\otimes
\mathrm{id})(\mathrm{id}\otimes\delta)\Delta(x)-(Q\otimes
\mathrm{id}\otimes \mathrm{id})(\Delta\otimes
\mathrm{id})\Delta(x)\\
&&=\sum_{i,j}\big( x\circ a_{i}\otimes a_{j}\otimes b_{i}\cdot b_{j}-x\circ a_{i}\otimes b_{i}\cdot a_{j}\otimes b_{j}-a_{i}\otimes a_{j}\otimes[x,b_{i}]\cdot b_{j}+a_{i}\otimes[x,b_{i}]\cdot a_{j}\otimes b_{j}\\
&&\ \ -b_{i}\otimes a_{j}\otimes(x\circ a_{i})\cdot b_{j}+b_{i}\otimes (x\circ a_{i})\cdot a_{j}\otimes b_{j}+[x,b_{i}]\otimes a_{j}\otimes a_{i}\cdot b_{j}-[x,b_{i}]\otimes a_{i}\cdot a_{j}\otimes b_{j}\\
&&\ \ -a_{i}\circ a_{j}\otimes b_{j}\otimes x\cdot b_{i}+a_{j}\otimes[a_{i},b_{j}]\otimes x\cdot b_{i}+(x\cdot a_{i})\circ a_{j}\otimes b_{j}\otimes b_{i}-a_{j}\otimes[x\cdot a_{i},b_{j}]\otimes b_{i}\\
&&\ \ +b_{j}\otimes a_{i}\circ a_{j}\otimes x\cdot b_{i}-[a_{i},b_{j}]\otimes a_{j}\otimes x\cdot b_{i}-b_{j}\otimes(x\cdot a_{i})\circ a_{j}\otimes b_{i}+[x\cdot a_{i},b_{j}]\otimes a_{j}\otimes b_{i}\\
&&\ \ -(x\cdot b_{i})\circ a_{j}\otimes a_{i}\otimes b_{j}+a_{j}\otimes a_{i}\otimes[x\cdot b_{i},b_{j}]+b_{i}\circ a_{j}\otimes x\cdot a_{i}\otimes b_{j}-a_{j}\otimes x\cdot a_{i}\otimes[b_{i},b_{j}]\\
&&\ \ +b_{j}\otimes a_{i}\otimes(x\cdot b_{i})\circ a_{j}-[x\cdot b_{i},b_{j}]\otimes a_{i}\otimes a_{j}-b_{j}\otimes x\cdot a_{i}\otimes b_{i}\circ a_{j}+[b_{i},b_{j}]\otimes x\cdot a_{i}\otimes a_{j}\\
&&\ \ -Q(a_{j})\otimes a_{i}\cdot b_{j}\otimes x\cdot b_{i}+Q(a_{i}\cdot a_{j})\otimes b_{j}\otimes x\cdot b_{i}
+Q(a_{j})\otimes x\cdot a_{i}\cdot b_{j}\otimes b_{i}-Q(x\cdot a_{i}\cdot a_{j})\otimes b_{j}\otimes b_{i}\big)\\
&&=A(1)+A(2)+A(3).
\end{eqnarray*}
Here
\begin{eqnarray*}
A(1)&=&\sum_{i,j}\big(x\circ a_{i}\otimes a_{j}\otimes b_{i}\cdot b_{j}-x\circ a_{i}\otimes b_{i}\cdot a_{j}\otimes b_{j}+[x,b_{i}]\otimes a_{j}\otimes a_{i}\cdot b_{j}\\
&&\ \ -[x,b_{i}]\otimes a_{i}\cdot a_{j}\otimes b_{j}+(x\cdot a_{i})\circ a_{j}\otimes b_{j}\otimes b_{i}+[x\cdot a_{i},b_{j}]\otimes a_{j}\otimes b_{i}\\
&&\ \ -(x\cdot b_{i})\circ a_{j}\otimes a_{i}\otimes b_{j}-[x\cdot b_{i},b_{j}]\otimes a_{i}\otimes a_{j}-Q(x\cdot a_{i}\cdot a_{j})\otimes b_{j}\otimes b_{i}\big)\\
&=&A(1,1)+A(1,2),\\
A(1,1)&=&\sum_{i,j}\big([x,b_{i}]\otimes a_{j}\otimes a_{i}\cdot b_{j}-[x,b_{i}]\otimes a_{i}\cdot a_{j}\otimes b_{j}+[x\cdot a_{i},b_{j}]\otimes a_{j}\otimes b_{i} -[x\cdot b_{i},b_{j}]\otimes a_{i}\otimes a_{j}\big)\\
&=&-\big(\mathrm{ad}(x)\otimes\mathrm{id}\otimes\mathrm{id}\big)\xi^{2}{\bf A}(r)-\sum_{j}\big(\mathrm{ad}(b_{j})\mathcal{L}_{\cdot}(x)\otimes\mathrm{id}\otimes\mathrm{id}\big)(\mathrm{id}\otimes\tau)\Big(\big(r+\tau(r)\big)\otimes a_{j}\Big)\\
&&+\sum_{i,j}([x,b_{i}\cdot b_{j}]\otimes a_{i}\otimes a_{j}+[b_{i},x\cdot b_{j}]\otimes a_{i}\otimes a_{j}+[b_{j},x\cdot b_{i}]\otimes a_{i}\otimes a_{j})\\
&=&-(\mathrm{ad}(x)\otimes\mathrm{id}\otimes\mathrm{id})\xi^{2}{\bf A}(r)-\sum_{j}\big(\mathrm{ad}(b_{j})\mathcal{L}_{\cdot}(x)\otimes\mathrm{id}\otimes\mathrm{id}\big)(\mathrm{id}\otimes\tau)\Big(\big(r+\tau(r)\big)\otimes a_{j}\Big)\\
&&+\sum_{i,j}P(x\cdot b_{i}\cdot b_{j})\otimes a_{i}\otimes a_{j}\\
&=&-(\mathrm{ad}(x)\otimes\mathrm{id}\otimes\mathrm{id})\xi^{2}{\bf A}(r)
+(P\mathcal{L}_{\cdot}(x)\otimes\mathrm{id}\otimes\mathrm{id})(r_{21}\cdot r_{31}-r_{12}\cdot r_{13})\\
&&-
\sum_{j}\big(\mathrm{ad}(b_{j})\mathcal{L}_{\cdot}(x)\otimes\mathrm{id}\otimes\mathrm{id}\big)(\mathrm{id}\otimes\tau)\Big(\big(r+\tau(r)\big)\otimes a_{j}\Big)\\
&&+\sum_{i,j}P(x\cdot a_{i}\cdot a_{j})\otimes b_{i}\otimes b_{j},\\
A(1,2)&=&\sum_{i,j}\big(x\circ a_{i}\otimes a_{j}\otimes b_{i}\cdot b_{j}-x\circ a_{i}\otimes b_{i}\cdot a_{j}\otimes b_{j}+(x\cdot a_{i})\circ a_{j}\otimes b_{j}\otimes b_{i}\\
&&-(x\cdot b_{i})\circ a_{j}\otimes  b_{i}\otimes b_{j}-Q(x\cdot a_{i}\cdot a_{j})\otimes b_{j}\otimes b_{i}\big)\\
&=&\big(\mathcal{L}_{\circ}(x)\otimes\mathrm{id}\otimes\mathrm{id}\big){\bf A}(r)-\sum_{j}\big(\mathcal{R}_{\circ}(a_{j})\mathcal{L}_{\cdot}(x)\otimes\mathrm{id}\otimes\mathrm{id}\big)\Big(\big(r+\tau(r)\big)\otimes b_{j}\Big)\\
&&+\sum_{i,j}\big(-x\circ(a_{i}\cdot a_{j})\otimes b_{i}\otimes b_{j}+(x\cdot a_{j})\circ a_{i}\otimes b_{i}\otimes b_{j}\\
&&+(x\cdot a_{i})\circ a_{j}\otimes b_{i}\otimes b_{j}-Q(x\cdot a_{i}\cdot a_{j})\otimes b_{i}\otimes b_{j}\big)\\
&\overset{\eqref{eq:rps4}}{=}&\big(\mathcal{L}_{\circ}(x)\otimes\mathrm{id}\otimes\mathrm{id}\big){\bf A}(r)-\sum_{j}\big(\mathcal{R}_{\circ}(a_{j})\mathcal{L}_{\cdot}(x)\otimes\mathrm{id}\otimes\mathrm{id}\big)\Big(\big(r+\tau(r)\big)\otimes b_{j}\Big)\\
&&+\sum_{i,j}\big(a_{j}\circ(x\cdot a_{i})\otimes b_{i}\otimes b_{j}+(x\cdot a_{i})\circ a_{j}\otimes b_{i}\otimes b_{j}-2Q(x\cdot a_{i}\cdot a_{j})\otimes b_{i}\otimes b_{j}\big)\\
&\overset{\eqref{eq:pf2}}{=}&\big(\mathcal{L}_{\circ}(x)\otimes\mathrm{id}\otimes\mathrm{id}\big){\bf A}(r)-\sum_{j}\big(\mathcal{R}_{\circ}(a_{j})\mathcal{L}_{\cdot}(x)\otimes\mathrm{id}\otimes\mathrm{id}\big)\Big(\big(r+\tau(r)\big)\otimes b_{j}\Big)\\
&&-\sum_{i,j}P(x\cdot a_{i}\cdot a_{j})\otimes b_{i}\otimes b_{j}.
\end{eqnarray*}
Thus
{\small
\begin{eqnarray*}
A(1)&=&A(1,1)+A(1,2)\\
&=&-\big(\mathrm{ad}(x)\otimes\mathrm{id}\otimes\mathrm{id}\big)\xi^{2}{\bf A}(r)+\big(\mathcal{L}_{\circ}(x)\otimes\mathrm{id}\otimes\mathrm{id}\big){\bf A}(r) +\big(P\mathcal{L}_{\cdot}(x)\otimes\mathrm{id}\otimes\mathrm{id}\big)(r_{21}\cdot r_{31}-r_{12}\cdot r_{13})\\
&&-\sum_{j} \bigg(\big(\mathrm{ad}(b_{j})\mathcal{L}_{\cdot}(x)\otimes\mathrm{id}\otimes\mathrm{id}\big)(\mathrm{id}\otimes\tau)\Big(\big(r+\tau(r)\big)\otimes a_{j}\Big)\\
&&+\big(\mathcal{R}_{\circ}(a_{j})\mathcal{L}_{\cdot}(x)\otimes\mathrm{id}\otimes\mathrm{id}\big)\Big(\big(r+\tau(r)\big)\otimes b_{j}\Big)\bigg).
\end{eqnarray*}}
\begin{eqnarray*}
    A(2)&=&\sum_{i,j}\big( a_{i}\otimes [x,b_{i}]\cdot a_{j}\otimes b_{j}+b_{i}\otimes(x\circ a_{i})\cdot a_{j}\otimes b_{j}-a_{j}\otimes[x\cdot a_{i},b_{j}]\otimes b_{i}\\
&&-b_{j}\otimes(x\cdot a_{i})\circ a_{j}\otimes b_{i}+b_{i}\circ a_{j}\otimes x\cdot a_{i}\otimes b_{j}-a_{j}\otimes x\cdot a_{i}\otimes[b_{i},b_{j}]\\
&&\ \ -b_{j}\otimes x\cdot a_{i}\otimes b_{i}\circ a_{j}+[b_{i},b_{j}]\otimes x\cdot a_{i}\otimes a_{j}+Q(a_{j})\otimes x\cdot a_{i}\cdot b_{j}\otimes b_{i}\big)\\
&=&A(2,1)+A(2,2)+A(2,3)+A(2,4),\\
A(2,1)&=&\sum_{i,j} (a_{i}\otimes[x,b_{i}]\cdot a_{j}\otimes b_{j}-a_{i}\otimes[x\cdot a_{j},b_{i}]\otimes b_{j})\\
&\overset{\eqref{eq:GLR}}{=}&\sum_{i,j}\big( a_{i}\otimes x\cdot[b_{i},a_{j}]\otimes b_{j}+a_{i}\otimes x\cdot a_{j}\cdot P(b_{i})\otimes b_{j} \big)\\
A(2,2)&=&\sum_{i,j}\big( b_{i}\otimes(x\circ a_{i})\cdot a_{j}\otimes b_{j}-b_{j}\otimes(x\cdot a_{i})\circ a_{j}\otimes b_{i}+Q(a_{j})\otimes x\cdot a_{i}\cdot b_{j}\otimes b_{i}\big)\\
&=&\sum_{j}\Big( \big(\mathrm{id}\otimes \mathcal{L}_{\cdot}(a_{j})\mathcal{L}_{\circ}(x)-\mathrm{id}\otimes\mathcal{L}_{\circ}(x\cdot a_{j}) \big)\big( r+\tau(r) \big)\otimes b_{j}\\
&&\ \ +\big(\mathrm{id}\otimes \mathcal{L}_{\cdot}(x\cdot a_{j}) \big)(Q\otimes\mathrm{id}-\mathrm{id}\otimes P)r\otimes b_{j}
\Big)\\
&&\ \ +\sum_{i,j}\big(a_{i}\otimes(x\cdot a_{j})\circ b_{i}\otimes b_{j}-a_{i}\otimes (x\circ b_{i})\cdot a_{j}\otimes b_{j}+a_{i}\otimes x\cdot a_{j}\cdot P(b_{i})\otimes b_{j}\big),\\
A(2,3)&=&\sum_{i,j}( b_{i}\circ a_{j}\otimes x\cdot a_{i}\otimes b_{j}-a_{j}\otimes x\cdot a_{i}\otimes [b_{i},b_{j}])\\
&=&-\big(\mathrm{id}\otimes\mathcal{L}_{\cdot}(x)\otimes\mathrm{id}\big){\bf T}(r)+\big(\mathrm{id}\otimes\mathcal{L}_{\cdot}(x)\otimes\mathrm{id}\big)(r_{21}\circ r_{13}+r_{12}\circ r_{13}) +\sum_{i,j} a_{i}\otimes x\cdot (b_{i} \circ a_{j} )\otimes b_{j},\\
A(2,4)&=&\sum_{i,j}( -b_{j}\otimes x\cdot a_{i}\otimes b_{i}\circ a_{j}+[b_{i},b_{j}]\otimes x\cdot a_{i}\otimes a_{j})\\
&=&-\big(\mathrm{id}\otimes\mathcal{L}_{\cdot}(x)\otimes\mathrm{id}\big)\xi^{2}{\bf T}(r)+\big(\mathrm{id}\otimes\mathcal{L}_{\cdot}(x)\otimes\mathrm{id}\big)(r_{23}\circ r_{21}+r_{23}\circ r_{12}) -\sum_{i,j} a_{i}\otimes x\cdot(a_{j}\circ b_{i})\otimes b_{j},
\end{eqnarray*}
Thus
{\small
\begin{eqnarray*}
A(2)&=&A(2,1)+A(2,2)+A(2,3)+A(2,4)\\
&=& -\big(\mathrm{id}\otimes\mathcal{L}_{\cdot}(x)\otimes\mathrm{id}\big)\big({\bf T}(r)+\xi^{2}{\bf T}(r)\big) +\big(\mathrm{id}\otimes\mathcal{L}_{\cdot}(x)\otimes\mathrm{id}\big)(r_{21}\circ r_{13}+r_{12}\circ r_{13}+r_{23}\circ r_{21}+r_{23}\circ r_{12})\\
&&+\sum_{j}\Big( \big(\mathrm{id}\otimes \mathcal{L}_{\cdot}(a_{j})\mathcal{L}_{\circ}(x)-\mathrm{id}\otimes\mathcal{L}_{\circ}(x\cdot a_{j}) \big) \big( r+\tau(r) \big)\otimes b_{j} \\
&&+\big(\mathrm{id}\otimes \mathcal{L}_{\cdot}(x\cdot a_{j}) \big)(Q\otimes\mathrm{id}-\mathrm{id}\otimes P)r\otimes b_{j}\Big)\\
&&+\sum_{i,j}a_{i}\otimes \big(2x\cdot [b_{i},a_{j}]+2x\cdot a_{j}\cdot P(b_{i})-(x\circ b_{i})\cdot a_{j}+(x\cdot a_{j})\circ b_{i}\big)\otimes b_{j}\\
&\overset{\eqref{eq:pf4}}{=}&-\big(\mathrm{id}\otimes\mathcal{L}_{\cdot}(x)\otimes\mathrm{id}\big)\big({\bf T}(r)+\xi^{2}{\bf T}(r)\big)+\big(\mathrm{id}\otimes\mathcal{L}_{\cdot}(x)\otimes\mathrm{id}\big)(r_{21}\circ r_{13}+r_{12}\circ r_{13}+r_{23}\circ r_{21}+r_{23}\circ r_{12})\\
&&+\sum_{j}\Big( \big(\mathrm{id}\otimes \mathcal{L}_{\cdot}(a_{j})\mathcal{L}_{\circ}(x)-\mathrm{id}\otimes\mathcal{L}_{\circ}(x\cdot a_{j}) \big)\big( r+\tau(r) \big)\otimes b_{j}\\
&&+\big(\mathrm{id}\otimes \mathcal{L}_{\cdot}(x\cdot a_{j}) \big)(Q\otimes\mathrm{id}-\mathrm{id}\otimes P)r\otimes b_{j}\Big).
\end{eqnarray*}}

\begin{eqnarray*}
A(3)&=&\sum_{i,j}\big( -a_{i}\otimes a_{j}\otimes[x,b_{i}]\cdot b_{j}-b_{i}\otimes a_{j}\otimes(x\circ a_{i})\cdot b_{j}-a_{i}\circ a_{j}\otimes b_{j}\otimes x\cdot b_{i}\\
&&+a_{j}\otimes[a_{i},b_{j}]\otimes x\cdot b_{i}+b_{j}\otimes a_{i}\circ a_{j}\otimes x\cdot b_{i}-[a_{i},b_{j}]\otimes a_{j}\otimes x\cdot b_{i}+a_{j}\otimes a_{i}\otimes [x\cdot b_{i},b_{j}]\\
&&+b_{j}\otimes a_{i}\otimes (x\cdot b_{i})\circ a_{j} -Q(a_{j})\otimes a_{i}\cdot b_{j}\otimes x\cdot b_{i}+Q(a_{i}\cdot a_{j})\otimes b_{j}\otimes x\cdot b_{i}\big)\\
&=&A(3,1)+A(3,2)+A(3,3),\\
A(3,1)&=&\sum_{i,j}( -a_{i}\circ a_{j}\otimes b_{j}\otimes x\cdot b_{i}-[a_{i},b_{j}]\otimes a_{j}\otimes x\cdot b_{i})\\
&=&-\sum_{j}\big(\mathrm{ad}(a_{j})\otimes\mathrm{id}\otimes\mathcal{L}_{\cdot}(x)\big)\Big(\big(r+\tau(r)\big)\otimes b_{j}\Big)-\sum_{i,j} a_{i}\circ a_{j}\otimes b_{i}\otimes x\cdot b_{j},\\
A(3,2)&=&\sum_{i,j} (a_{j}\otimes [a_{i},b_{j}]\otimes x\cdot b_{i}+b_{j}\otimes a_{i}\circ a_{j}\otimes x\cdot b_{i})\\
&=&\sum_{j} \big(\mathrm{id}\otimes\mathcal{L}_{\circ}(a_{j})\otimes\mathcal{L}_{\cdot}(x)\big)\Big(\big(r+\tau(r)\big)\otimes b_{j}\Big)-\sum_{i,j} a_{i}\otimes b_{i}\circ a_{j}\otimes x\cdot b_{j},\\
A(3,3)&=&\sum_{i,j} \big( -a_{i}\otimes a_{j}\otimes [x,b_{i}]\cdot b_{j}-b_{i}\otimes a_{j}\otimes (x\circ a_{i})\cdot b_{j}+a_{j}\otimes a_{i}\otimes [x\cdot b_{i},b_{j}]\\
&&+b_{j}\otimes a_{i}\otimes (x\cdot b_{i})\circ a_{j}-Q(a_{j})\otimes a_{i}\cdot b_{j}\otimes x\cdot b_{i}+Q(a_{i}\cdot a_{j})\otimes b_{j}\otimes x\cdot b_{i}\big)\\
&=&(Q\otimes\mathrm{id}\otimes\mathcal{L}_{\cdot}(x)){\bf A}(r) +\sum_{j} \bigg( (\tau\otimes\mathrm{id})\big(\mathrm{id}\otimes\mathrm{id}\otimes\mathcal{L}_{\cdot}(x\cdot b_{j})\big)   \big(a_{j}\otimes (\mathrm{id}\otimes P-Q\otimes\mathrm{id})r \big)\\
&&+ \Big(\mathrm{id}\otimes\mathrm{id}\otimes\big(\mathcal{L}_{\circ}(x\cdot b_{j})-\mathcal{L}_{\cdot}(b_{j})\mathcal{L}_{\circ}(x)\big)\Big)(\tau\otimes\mathrm{id})\Big(a_{j}\otimes \big(r+\tau(r)\big)\Big) \bigg)\\
&&+\sum_{i,j}\big( -a_{i}\otimes a_{j}\otimes x\cdot P(b_{i})\cdot b_{j}- a_{i}\otimes a_{j}\otimes [x,b_{i}]\cdot b_{j}+a_{i}\otimes a_{j}\otimes (x\circ b_{i})\cdot b_{j}\\
&&+a_{i}\otimes a_{j}\otimes [x\cdot b_{j},b_{i}]-a_{i}\otimes a_{j}\otimes (x\cdot b_{j})\circ b_{i} \big)\\
&\overset{\eqref{eq:rps4}}{=}&(Q\otimes\mathrm{id}\otimes\mathcal{L}_{\cdot}(x)){\bf A}(r) +\sum_{j} \bigg( (\tau\otimes\mathrm{id})\big(\mathrm{id}\otimes\mathrm{id}\otimes\mathcal{L}_{\cdot}(x\cdot b_{j})\big)   \big(a_{j}\otimes (\mathrm{id}\otimes P-Q\otimes\mathrm{id})r \big)\\
&&+ \Big(\mathrm{id}\otimes\mathrm{id}\otimes\big(\mathcal{L}_{\circ}(x\cdot b_{j})-\mathcal{L}_{\cdot}(b_{j})\mathcal{L}_{\circ}(x)\big)\Big)(\tau\otimes\mathrm{id})\Big(a_{j}\otimes \big(r+\tau(r)\big)\Big) \bigg)\\
&&+\sum_{i,j} a_{i}\otimes a_{j}\otimes x\cdot[b_{i},b_{j}],
\end{eqnarray*}
Thus
\begin{eqnarray*}
A(3)&=&A(3,1)+A(3,2)+A(3,3)\\
&=& \big(Q\otimes\mathrm{id}\otimes\mathcal{L}_{\cdot}(x)\big){\bf A}(r)-\big(\mathrm{id}\otimes\mathrm{id}\otimes\mathcal{L}_{\cdot}(x)\big){\bf T}(r)\\
&&+\sum_{j}\bigg(  \big(\mathrm{id}\otimes\mathcal{L}_{\circ}(a_{j})\otimes\mathcal{L}_{\cdot}(x)-\mathrm{ad}(a_{j})\otimes\mathrm{id}\otimes\mathcal{L}_{\cdot}(x)\big)
\Big(\big(r+\tau(r)\big)\otimes b_{j}\Big)\\
&&- (\tau\otimes\mathrm{id})\big(\mathrm{id}\otimes\mathrm{id}\otimes\mathcal{L}_{\cdot}(x\cdot b_{j})\big)\big( a_{j}\otimes (Q\otimes\mathrm{id}-\mathrm{id}\otimes P)r\big)\\
&&+\big(\mathrm{id}\otimes\mathrm{id}\otimes\mathcal{L}_{\circ}(x\cdot b_{j})-\mathrm{id}\otimes\mathrm{id}\otimes\mathcal{L}_{\cdot}(b_{j})\mathcal{L}_{\circ}(x)\big)(\tau\otimes\mathrm{id})\Big(a_{j}\otimes \big(r+\tau(r)\big)\Big)  \bigg).
\end{eqnarray*}}
Thus \eqref{eq:cos3} holds if and only if \eqref{eq:pro:cob cos3} holds.
\end{proof}

\begin{pro}\label{pro:cob mp}
Let $(A,\cdot,\circ,P,Q)$ be a relative PCA algebra. Let $r\in A\otimes A$ and $\Delta,\theta:A\rightarrow A\otimes A$ be linear
maps defined by \eqref{eq:AssoCob}  and  \eqref{eq:aplCob2}
respectively.
\begin{enumerate}
    \item \eqref{b1} holds if and only if the following equation holds:
    \begin{equation}\label{eq:pro:cob mp1}
\big(\mathrm{id}\otimes\mathcal{L}_{\cdot}(x\cdot y)\big)(\mathrm{id}\otimes P-Q\otimes\mathrm{id})r=0,\;\forall x,y\in A.
\end{equation}
\item \eqref{b2} holds if and only if the following equation holds:
    \begin{equation}\label{eq:pro:cob mp2}
\big(\mathcal{L}_{\cdot}(y)\otimes \mathcal{L}_{\circ}(x)-\mathcal{L}_{\cdot}(y)\mathrm{ad}(x)\otimes\mathrm{id}\big)\big(r+\tau(r)\big)=0,\;\forall x,y\in A.
\end{equation}
\item \eqref{b3} holds automatically.
\item\label{pro:cob mp4} \eqref{b4} holds if and only if the following equation holds:
   {\small
    \begin{equation}\label{eq:pro:cob mp4}
\begin{split}
&\Big(\big(\mathcal{L}_{\cdot}(x)\mathrm{ad}(y)-\mathcal{L}_{\cdot}(y)\mathrm{ad}(x)-\mathrm{ad}(x\cdot y)\big)\otimes\mathrm{id}-\mathcal{L}_{\cdot}(x)\otimes\mathcal{L}_{\circ}(y)+\mathcal{L}_{\cdot}(y)\otimes\mathcal{L}_{\circ}(x)+\mathrm{id}\otimes\mathcal{L}_{\circ}(x\cdot y)\Big)\big(r+\tau(r)\big)\\
&+\big(\mathcal{L}_{\cdot}(x\cdot y)\otimes\mathrm{id}\big)(P\otimes\mathrm{id}-\mathrm{id}\otimes Q)r=0,\;\forall x,y\in A.
\end{split}
\end{equation}}
\end{enumerate}
\end{pro}
\begin{proof}
 We only give the proof of Item  (\ref{pro:cob mp4}), and other items are proved similarly. Set $r =\sum\limits_{i}a_{i}\otimes b_{i}$. Let $x,y\in A$. Then we have
 \begin{eqnarray*}
&&\big( \mathrm{id}\otimes\mathcal{R}_{\circ}(y)\big)\Delta(x)
+\big(\mathcal{L}_{\cdot}(x)\otimes\mathrm{id}\big)\delta(y)
-\big(\mathrm{id}\otimes\mathcal{L}_{\circ}(x)\big)\Delta(y)\\
&&\ \ +\tau\big(\mathrm{id}\otimes\mathcal{L}_{\cdot}(y)\big)\theta(x)
-\delta(x\cdot y)+(\mathrm{id}\otimes Q)\Delta(x\cdot y)\\
&&=\sum_{i}\big( a_{i}\otimes(x\cdot b_{i})\circ y-x\cdot a_{i}\otimes b_{i}\circ y+x\cdot(y\circ a_{i})\otimes b_{i}-x\cdot a_{i}\otimes [y,b_{i}]\\
&&\ \ -x\cdot b_{i}\otimes y\circ a_{i}+x\cdot[y,b_{i}]\otimes a_{i}-a_{i}\otimes x\circ(y\cdot b_{i})+y\cdot a_{i}\otimes x\circ b_{i}\\
&&\ \ +y\cdot b_{i}\otimes x\circ a_{i}-y\cdot[x,b_{i}]\otimes a_{i}+b_{i}\otimes(x\cdot y)\circ a_{i}-[x\cdot y,b_{i}]\otimes a_{i}\\
&&\ \ -(x\cdot y)\circ a_{i}\otimes b_{i}+a_{i}\otimes[x\cdot y,b_{i}]+a_{i}\otimes Q(x\cdot y\cdot b_{i})-x\cdot y\cdot a_{i}\otimes Q(b_{i}) \big)\\
&&=B(1)+B(2)+B(3)+B(4),
\end{eqnarray*}
where
{\small
\begin{eqnarray*}
B(1)&=&\sum_{i}\big( x\cdot(y\circ a_{i})\otimes b_{i}+x\cdot [y,b_{i}]\otimes a_{i}-y\cdot[x,b_{i}]\otimes a_{i}\\
&&-[x\cdot y,b_{i}]\otimes a_{i}-(x\cdot y)\circ a_{i}\otimes b_{i}-x\cdot y\cdot a_{i}\otimes Q(b_{i})\big)\\
&=&\Big(\big(\mathcal{L}_{\cdot}(x)\mathrm{ad}(y)-\mathcal{L}_{\cdot}(y)\mathrm{ad}(x)-\mathrm{ad}(x\cdot y)\big)\otimes\mathrm{id}\Big)\big(r+\tau(r)\big) +\big(\mathcal{L}_{\cdot}(x\cdot y)\otimes\mathrm{id}\big)(P\otimes\mathrm{id}-\mathrm{id}\otimes Q)r\\
&&+\sum_{i}\big( x\cdot(y\circ a_{i})\otimes b_{i}-(x\cdot y)\circ a_{i}\otimes b_{i}-x\cdot[y,a_{i}]\otimes b_{i} +y\cdot[x,a_{i}]\otimes b_{i}+[x\cdot y,a_{i}]\otimes b_{i}\big)\\
&\overset{\eqref{eq:rps3}}{=}&\Big(\big(\mathcal{L}_{\cdot}(x)\mathrm{ad}(y)-\mathcal{L}_{\cdot}(y)\mathrm{ad}(x)-\mathrm{ad}(x\cdot y)\big)\otimes\mathrm{id}\Big)\big(r+\tau(r)\big)
+\big(\mathcal{L}_{\cdot}(x\cdot y)\otimes\mathrm{id}\big)(P\otimes\mathrm{id}-\mathrm{id}\otimes Q)r,\\
B(2)&=&\sum_{i}(-x\cdot a_{i}\otimes b_{i}\circ y-x\cdot a_{i}\otimes [y,b_{i}]-x\cdot b_{i}\otimes y\circ a_{i})\\
&=&-\big(\mathcal{L}_{\cdot}(x)\otimes\mathcal{L}_{\circ}(y)\big)\big(r+\tau(r)\big),\\
B(3)&=&\sum_{i}( y\cdot a_{i}\otimes x\circ b_{i}+y\cdot b_{i}\otimes x\circ a_{i})\\
&=&\big(\mathcal{L}_{\cdot}(y)\otimes\mathcal{L}_{\circ}(x)\big)\big(r+\tau(r)\big),\\
B(4)&=&\sum_{i}\big( a_{i}\otimes (x\cdot b_{i})\circ y-a_{i}\otimes x\circ(y\cdot b_{i})+b_{i}\otimes(x\cdot y)\circ a_{i} +a_{i}\otimes[x\cdot y,b_{i}]+a_{i}\otimes Q(x\cdot y\cdot b_{i})\big)\\
&=&\big(\mathrm{id}\otimes\mathcal{L}_{\circ}(x\cdot y)\big)\big(r+\tau(r)\big)\\
&&+\sum_{i}\big( a_{i}\otimes(x\cdot b_{i})\circ y-a_{i}\otimes x\circ(y\cdot b_{i})-a_{i}\otimes (x\cdot y)\circ b_{i}
+a_{i}\otimes [x\cdot y,b_{i}]+a_{i}\otimes Q(x\cdot y\cdot b_{i})\big)\\
&\overset{\eqref{eq:rps4}}{=}&\big(\mathrm{id}\otimes\mathcal{L}_{\circ}(x\cdot y)\big)\big(r+\tau(r)\big).
\end{eqnarray*}}Hence \eqref{b4} holds if and only if \eqref{eq:pro:cob mp4} holds.
\end{proof}

Combining Propositions \ref{pro:cob cos} and \ref{pro:cob mp}, we have the following result.

\begin{thm}\label{thm:cob}
    Let $(A,\cdot,\circ,P,Q)$ be a relative PCA algebra.
Let $r=\sum\limits_{i}a_{i}\otimes b_{i}\in A\otimes A$ and
$\Delta,\theta:A\rightarrow A\otimes A$ be linear maps defined by
\eqref{eq:AssoCob}  and  \eqref{eq:aplCob2} respectively. Then
$(A,\cdot,\circ,\Delta,\theta,P,Q)$ is a relative PCA bialgebra if
and only if \eqref{eq:ASI 1}-\eqref{eq:ASI 2}, \eqref{eq:Apl
1}-\eqref{eq:Apl 3}, \eqref{eq:pro:cob cos1}-\eqref{eq:pro:cob
cos7} and \eqref{eq:pro:cob mp1}-\eqref{eq:pro:cob mp4} hold.
\end{thm}

\begin{defi}
 Let $(A,\cdot,\circ,P,Q)$ be a relative PCA algebra and $r\in A\otimes A$.
 If $\textbf{A}(r)=\textbf{T}(r)=0$, \eqref{eq:pro:d2} and the following equation hold:
\begin{equation}\label{eq:pro:d2,2}
    (P\otimes\mathrm{id}-\mathrm{id}\otimes Q)r=0,
\end{equation}
 then we say $r$ is a solution of the {\bf relative PCA  Yang-Baxter equation (RPCA-YBE) in $(A,\cdot,\circ,P,Q)$}.
\end{defi}

Note that when $r$ is antisymmetric, then \eqref{eq:pro:d2} holds if and only if \eqref{eq:pro:d2,2} holds. By Theorem \ref{thm:cob}, we have the following result.

\begin{cor}\label{cor:cob}
     Let $(A,\cdot,\circ,P,Q)$ be a relative PCA algebra. If $r\in A\otimes A$ is an antisymmetric solution of the RPCA-YBE in $(A,\cdot,\circ,P,Q)$,
     then $(A,\cdot,\circ,\Delta,\theta,P,Q)$ is a relative PCA bialgebra, where $\Delta,\theta:A\rightarrow A\otimes A$ are linear maps defined by
\eqref{eq:AssoCob}  and  \eqref{eq:aplCob2} respectively.
\end{cor}

Let $(A,\cdot, P)$ be a commutative differential algebra and
$Q:A\rightarrow A$ be a linear map. Then $r\in A\otimes A$ is
called a solution of the {\bf $Q$-admissible AYBE in $(A,\cdot,
P)$} \cite{LLB2023} if $\textbf{A}(r)=0$ and \eqref{eq:pro:d2} and
\eqref{eq:pro:d2,2} hold.

\begin{pro}
 Let $(A,\cdot, P, Q)$ be an admissible commutative differential algebra and
    $(A,\cdot,\circ$, $P$, $Q)$ be the relative PCA algebra given by Proposition \ref{ex:ex}.
    Then for any $r=\sum\limits_{i} a_{i}\otimes b_{i}\in A\otimes A$, we have
  {\small  \begin{eqnarray*}
    \textbf{T}(r)&=&(Q\otimes\mathrm{id}\otimes \mathrm{id}-\mathrm{id}\otimes Q\otimes\mathrm{id})\textbf{A}(r) +\sum_{j}\Big(\big(\mathrm{id}\otimes\mathcal{L}_{\cdot}(a_{j})\otimes\mathrm{id}\big)\big((Q\otimes\mathrm{id}-\mathrm{id}\otimes P)r\otimes b_{j}\big)\\
    &&-\big(\mathrm{id}\otimes\mathrm{id}\otimes\mathcal{L}_{\cdot}(b_{j})\big)(\tau\otimes\mathrm{id})\big(a_{j}\otimes(Q\otimes\mathrm{id}-\mathrm{id}\otimes P)r\big)\\
    &&+\big(\mathcal{L}_{\cdot}(a_{j})\otimes\mathrm{id}\otimes\mathrm{id}\big)\big((\mathrm{id}\otimes Q-P\otimes\mathrm{id})r\otimes b_{j}\big)+\big(\mathrm{id}\otimes\mathrm{id}\otimes\mathcal{L}_{\cdot}(b_{j})\big)\big(a_{j}\otimes(Q\otimes\mathrm{id}-\mathrm{id}\otimes P)r\big)\Big).
    \end{eqnarray*}}In particular, if $r$ is  a solution of the $Q$-admissible AYBE in $(A,\cdot,P)$, then $r$ is also a solution of the RPCA-YBE in $(A,\cdot,\circ,P,Q)$.
\end{pro}
\begin{proof}
By the assumption, we have {\small \begin{eqnarray*}
 \textbf{T}(r)&=&r_{12}\circ r_{13}+r_{12}\circ r_{23}-[r_{13},r_{23}]\\
 &=&\sum_{i,j} a_{i}\circ a_{j}\otimes b_{i}\otimes b_{j}+a_{i}\otimes b_{i}\circ a_{j}\otimes b_{j}-a_{i}\otimes a_{j}\otimes [b_{i},b_{j}]\\
 &=&\sum_{i,j} Q(a_{i}\cdot a_{j})\otimes b_{i}\otimes b_{j}-P(a_{i})\cdot a_{j}\otimes b_{i}\otimes b_{j}+a_{i}\otimes Q(b_{i}\cdot a_{j})\otimes b_{j}\\
 &&-a_{i}\otimes P(b_{i})\cdot a_{j}\otimes b_{j}+a_{i}\otimes a_{j}\otimes P(b_{i})\cdot b_{j}-a_{i}\otimes a_{j}\otimes b_{i}\cdot P(b_{j})\\
 &=&\sum_{i,j} Q(a_{i}\cdot a_{j})\otimes b_{i}\otimes b_{j}-Q(a_{i})\otimes b_{i}\cdot a_{j}\otimes b_{j}+Q(a_{i})\otimes a_{j}\otimes b_{i}\cdot b_{j}\\
 &&+Q(a_{i})\otimes b_{i}\cdot a_{j}\otimes b_{j}-a_{i}\otimes P(b_{i})\cdot a_{j}\otimes b_{j}\\
 &&+a_{i}\otimes a_{j}\otimes P(b_{i})\cdot b_{j}-Q(a_{i})\otimes a_{j}\otimes b_{i}\cdot b_{j}\\
 &&-a_{i}\cdot a_{j}\otimes Q(b_{i})\otimes b_{j}+a_{i}\otimes Q(b_{i}\cdot a_{j})\otimes b_{j}-a_{i}\otimes Q(a_{j})\otimes b_{i}\cdot b_{j}\\
 &&+a_{i}\cdot a_{j}\otimes Q(b_{i})\otimes b_{j}-P(a_{i})\cdot a_{j}\otimes b_{i}\otimes b_{j}\\
 &&+a_{i}\otimes Q(a_{j})\otimes b_{i}\cdot b_{j}-a_{i}\otimes a_{j}\otimes b_{i}\cdot P(b_{j})\\
 &=&(Q\otimes\mathrm{id}\otimes \mathrm{id}-\mathrm{id}\otimes Q\otimes\mathrm{id})\textbf{A}(r) +\sum_{j}\Big(\big(\mathrm{id}\otimes\mathcal{L}_{\cdot}(a_{j})\otimes\mathrm{id}\big)\big((Q\otimes\mathrm{id}-\mathrm{id}\otimes P)r\otimes b_{j}\big)\\
 &&-\big(\mathrm{id}\otimes\mathrm{id}\otimes\mathcal{L}_{\cdot}(b_{j})\big)(\tau\otimes\mathrm{id})\big(a_{j}\otimes(Q\otimes\mathrm{id}-\mathrm{id}\otimes P)r\big)\\
 &&+\big(\mathcal{L}_{\cdot}(a_{j})\otimes\mathrm{id}\otimes\mathrm{id}\big)\big((\mathrm{id}\otimes Q-P\otimes\mathrm{id})r\otimes b_{j}\big)+\big(\mathrm{id}\otimes\mathrm{id}\otimes\mathcal{L}_{\cdot}(b_{j})\big)\big(a_{j}\otimes(Q\otimes\mathrm{id}-\mathrm{id}\otimes P)r\big)\Big).
\end{eqnarray*}}Hence the conclusion follows.
\end{proof}

Let $V$ and $A$ be vector spaces.
Then any $r\in V\otimes A$ is identified with a linear map (still denoted by $r$) from $V^{*}$ to $A$ by
\begin{equation}\label{eq:identify}
    \langle r(u^{*}),a^{*}\rangle =\langle r, u^{*}\otimes a^{*}\rangle,\;\forall u^{*}\in V^{*},a^{*} \in A^{*}.
\end{equation}
Then we have the following result.
\begin{pro}\label{pro:4.8}
 Let $(A,\cdot,\circ,P,Q)$ be a relative PCA algebra and $r\in A\otimes A$ be antisymmetric.
 Then $r$ is a solution of the RPCA-YBE in $(A,\cdot,\circ,P,Q)$ if and only if the following equations hold:
 \begin{eqnarray}
    r(a^{*})\cdot r(b^{*})&=&r\Big( -\mathcal{L}^{*}_{\cdot }\big(r(a^{*})\big)b^{*}-\mathcal{L}^{*}_{\cdot }\big(r(b^{*})\big)a^{*}\Big),\label{eq:of1}\\
    r(a^{*})\circ r(b^{*})&=&r\Big( -\mathrm{ad}^{*} \big(r (a^{*})\big)b^{*}+ \mathcal{R}^{*}_{\circ }\big(r(b^{*})\big)a^{*}\Big),\label{eq:of2}\\
    Pr&=&rQ^{*},\label{eq:of3}
 \end{eqnarray}
 for all $a^{*},b^{*}\in A^{*}$.
\end{pro}
\begin{proof}
By \cite{Bai2010}, $\textbf{A}(r)=0$ if and only if \eqref{eq:of1} holds.
By \cite{TPA}, $\textbf{T}(r)=0$ if and only if \eqref{eq:of2} holds.
By \cite{RPA}, \eqref{eq:pro:d2,2} holds if and only if \eqref{eq:of3} holds. Hence the conclusion follows.
\end{proof}

\subsection{$\mathcal{O}$-operators of relative PCA algebras}\

\begin{defi}
Let $(\mu,l_{\circ},r_{\circ},\alpha,\beta,V)$ be a representation
of  a relative PCA algebra  $(A,\cdot,\circ,P,Q)$.  If there is a
linear map $T:V\rightarrow A$ satisfying the following conditions:
\begin{enumerate}
    \item\cite{Bai2010} $T$ is an $\mathcal{O}$-operator of $(A,\cdot)$ associated to $(\mu,V)$, that is,
    \begin{equation}
        T(u)\cdot T(v)=T\Big( \mu\big(T(u)\big)v+\mu\big(T(v)\big)u     \Big),\;\forall u,v\in V;
        \end{equation}
      \item\cite{TPA} $T$ is an $\mathcal{O}$-operator of $(A,\circ)$ associated to $(l_{\circ},r_{\circ},V)$, that is,
    \begin{equation}
        T(u)\circ T(v)=T\Big( l_{\circ}\big(T(u)\big)v+r_{\circ}\big(T(v)\big)u     \Big),\;\forall u,v\in V;
    \end{equation}
    \item the following equations hold:
    \begin{equation}\label{eq:oop3}
        PT=T\alpha,
    \end{equation}
    \begin{equation}\label{eq:oop4}
       QT=T\beta,
    \end{equation}
\end{enumerate}
then we say $T$ is an {\bf $\mathcal{O}$-operator of $(A,\cdot,\circ,P,Q)$ associated to $(\mu,l_{\circ},r_{\circ},\alpha,\beta,V)$}.
\end{defi}

Now we can rewrite Proposition \ref{pro:4.8} as follows.

\begin{cor}\label{cor:as rps}
    Let $(A,\cdot,\circ,P,Q)$ be a relative PCA algebra and $r\in A\otimes A$ be antisymmetric.
    Then $r$ is a solution of the RPCA-YBE in $(A,\cdot,\circ,P,Q)$ if and only if $r:A^{*}\rightarrow A$ is an $\mathcal{O}$-operator of $(A,\cdot,\circ,P,Q)$ associated to $(-\mathcal{L}^{*}_{\cdot},-\mathrm{ad}^{*},\mathcal{R}^{*}_{\circ},Q^{*},P^{*},A^{*})$.
\end{cor}

\delete{
\begin{rmk}
If $r\in A\otimes A$ is antisymmetric, then its operator form satisfies $r^{*}=-r$.
Hence we have
\begin{equation*}
    (Pr-rQ^{*})^{*}=r^{*}P^{*}-Qr^{*}=Qr-rP^{*},
\end{equation*}
and thus $Pr=rQ^{*}$ if and only if $Qr=rP^{*}$.
This indicates that the condition \eqref{eq:oop4} is superfluous for Corollary \ref{cor:as rps}.
In fact, in the study of relative Poisson bialgebras \cite{RPA},
the authors introduced the notion of $\mathcal{O}$-operators of relative Poisson algebras where \eqref{eq:oop4} is 
not involved.
Then it is shown that if
the tensor form of an $\mathcal{O}$-operator $r:A^{*}\rightarrow A$ of a relative Poisson  algebra $(A,\cdot,[-,-],P)$ associated to  a representation $(-\mathcal{L}^{*}_{\cdot},\mathrm{ad}^{*},Q^{*},A^{*})$
is antisymmetric, then it gives rise to a relative Poisson bialgebra through specific co-multiplications.
In our present study, \eqref{eq:oop4} is  needed due to the existence of $Q$ in a relative PCA algebra $(A,\cdot,\circ,P,Q)$.
In the following, we shall see the convenience of  the existence of the linear map $Q$ and condition \eqref{eq:oop4} in advance:
they fit naturally in the framework of \cite{Bai2007} concerning the semi-direct products.
\end{rmk}}

\begin{thm}\label{thm:oop}
  Let $(\mu,l_{\circ},r_{\circ},\alpha,\beta,V)$ be a representation
of  a relative PCA algebra  $(A,\cdot,\circ,P,Q)$.
   Let $T:V\rightarrow A$ be a linear map which is identified as an element in $$V^{*}\otimes A\subset A\ltimes_{-\mu^{*},r^{*}_{\circ}-l^{*}_{\circ}, r^{*}_{\circ}}V^{*}\otimes A\ltimes_{-\mu^{*},r^{*}_{\circ}-l^{*}_{\circ}, r^{*}_{\circ}}V^{*}$$ by \eqref{eq:identify}. Then $r=T-\tau(T)$ is a solution of the RPCA-YBE in the relative PCA algebra $( A\ltimes_{-\mu^{*},r^{*}_{\circ}-l^{*}_{\circ}, r^{*}_{\circ}}V^{*},P+\beta^{*},Q+\alpha^{*})$ if and only if
   $T$ is an $\mathcal{O}$-operator of $(A,\cdot,\circ,P,Q)$ associated to $(\mu,l_{\circ},r_{\circ},\alpha,\beta,V)$.
\end{thm}
\begin{proof}
    By \cite{Bai2010}, $r=T-\tau(T)$ is a solution of the AYBE in  $ A\ltimes_{-\mu^{*}} V^{*} $ if and only if
   $T$ is an $\mathcal{O}$-operator of $(A,\cdot )$ associated to $(\mu, V)$.
   By \cite{TPA}, $r=T-\tau(T)$ is a solution of the APL-YBE in  $ A\ltimes_{r^{*}_{\circ}-l^{*}_{\circ}, r^{*}_{\circ}} V^{*} $ if and only if
   $T$ is an $\mathcal{O}$-operator of $(A,\circ )$ associated to $(l_{\circ},r_{\circ},V)$.
   By \cite{RPA}, $r=T-\tau(T)$ satisfies
   \begin{eqnarray*}
       \big( (P+\beta^{*})\otimes\mathrm{id} \big)r=\big( \mathrm{id}\otimes ( Q+\alpha^{*} )  \big)r
   \end{eqnarray*}
   if and only if \eqref{eq:oop3}-\eqref{eq:oop4} hold.
   Hence the conclusion follows.
\end{proof}

We recall the notions of Zinbiel algebras and pre-APL algebras.

A \textbf{Zinbiel algebra} \cite{Lod} is a vector space $A$
equipped with a binary operation  $\star:A\otimes A\rightarrow A$
such that
\begin{equation}\label{eq:Zinbiel}
x\star(y\star z)=(y\star x)\star z+(x\star y)\star z, \;\;\forall x,y,z\in A.
\end{equation}
Let $(A,\star)$ be a Zinbiel algebra.
Then there is a commutative associative algebra $(A,\cdot)$ in which $\cdot$ is defined by
\begin{equation}\label{eq:ZintoAss}
x\cdot y=x\star y+y\star x,\;\;\forall x,y\in A,
\end{equation}
which is called the {\bf descendent commutative associative
algebra of $(A,\star)$}. Moreover, $(\mathcal{L}_{\star},A)$ is a
representation of  $(A,\cdot)$.

A {\bf differential Zinbiel algebra} \cite{LLB2023} is a Zinbiel algebra $(A,\star)$ with a derivation $P$, that is,
\begin{equation} \label{eq:pre-sys3}
    P(x\star y)=P(x)\star y+x\star P(y), \;\forall x,y\in A.
\end{equation}
Moreover, if there is a linear map $Q:A\rightarrow A$ such that the following equations hold:
\begin{eqnarray}
    && Q(x)\star y-x\star P(y)-Q(x\star y)=0,  \label{eq:pre-sys5}\\
    &&  x\star Q(y)-P(x)\star y-Q(x\star y)=0, \label{eq:pre-sys6}
    \end{eqnarray}
then we say $(A,\star,P,Q)$ is an {\bf admissible differential
Zinbiel algebra} \cite{HBG}.

\begin{ex}\label{ex:ADZA}
    Let $A=\mathrm{span}\{e_{1}, e_{2}, e_{3}\}$ be a vector space, $\star:A\otimes A\rightarrow A$ be a binary operation with nonzero products given by
    \begin{equation*}
        e_{1}\star e_{1}=e_{1}\star e_{2}=e_{3}
    \end{equation*}
    and $P:A\rightarrow A$ be a linear map given by
    \begin{equation*}
        P(e_{1})=e_{1}+e_{2},\; P(e_{2})=2e_{2},\; P(e_{3})=3e_{3}.
    \end{equation*}
    Then $(A,\star,P,-P)$ is an admissible differential Zinbiel
    algebra.
\end{ex}

A {\bf pre-APL algebra} \cite{TPA}  is a triple $(A,\succ$,
    $\prec)$, where $A$ is a vector space and $\succ,\prec:A\otimes A\rightarrow A$ are binary operations, such that the following equations hold:
    \begin{small}
\begin{eqnarray}
x\succ(y\succ z)-y\succ(x\succ z)
    &=&(y\succ x+y\prec x)\succ z-(x\succ y+x\prec y)\succ z, \ \ \ \ \  \label{eq:defi:quasi anti-pre-Lie algebras1}\\
    z\prec(x\succ y+ x\prec y)&=&x\succ(z\prec y)+(x\succ z)\prec y-(z\prec x)\prec y, \ \ \ \ \  \label{eq:defi:quasi anti-pre-Lie algebras2}\\
   x\succ(y\succ z)-y\succ(x\succ z)
       &=&(y\succ z)\prec x-(x\succ z)\prec y-(z\prec y)\prec x+(z\prec x)\prec y,\; \ \ \ \ \  \label{eq:defi:quasi anti-pre-Lie algebras3}
\end{eqnarray}
\end{small}for all $x,y,z\in A$. Let $(A,\succ,\prec)$ be a pre-APL algebra. Then there is an anti-pre-Lie algebra $(A,\circ)$ in which $\circ$ is given by
    \begin{equation}\label{eq:pre-APL to APL}
        x\circ y=x\succ y+ x\prec y,\;\;\forall x,y\in A,
    \end{equation}
    which is called the {\bf sub-adjacent anti-pre-Lie algebra
     of $(A,\succ,\prec)$}.
 Moreover, $(\mathcal{L}_{\succ},\mathcal{R}_{\prec},A)$ is a representation of $(A,\circ)$, and the identity map $\mathrm{id}$ is an $\mathcal{O}$-operator of $(A,\circ)$ associated to $(\mathcal{L}_{\succ},\mathcal{R}_{\prec}, A)$.

   Recall a {\bf pre-Lie algebra} \cite{Bai2021.2,Bur} is a vector space $A$ together with a binary operation $\diamond:A\otimes A\rightarrow A$ such that
      \begin{equation*}
        (x\diamond y)\diamond z-x\diamond (y\diamond z)=
        (y\diamond x)\diamond z-y\diamond (x\diamond z),\;\forall x,y,z\in A.
      \end{equation*}
  A pre-Lie algebra $(A,\diamond)$ is a Lie-admissible algebra and $(\mathcal{L}_{\diamond},A)$ is a representation of $(A,[-,-])$ which is the sub-adjacent Lie algebra of $(A,\diamond)$.

Let $(\rho,V)$ be a representation of a Lie algebra $(A,[-,-])$
and $T:V\rightarrow A$ be a linear map.
  If $T:V\rightarrow A$ is an $\mathcal{O}$-operator of $(A,[-,-])$ associated to $(\rho,V)$ in the sense that
  \begin{equation*}
    [T(u),T(v)]=T\Big(\rho\big(T(u)\big)v-\rho\big(T(v)\big)u\Big),\;\forall u,v\in V,
  \end{equation*}
then by \cite{Bai2007}, there is a pre-Lie algebra structure on $V$ defined by
\begin{equation*}
    u\diamond v=\rho\big(T(u)\big)v.
\end{equation*} 
\begin{pro}
    Let $(A,\succ,\prec)$ be a pre-APL algebra and  $(A,\circ)$ be the sub-adjacent anti-pre-Lie algebra of $(A,\succ,\prec)$. Then there is a pre-Lie algebra $(A,\diamond)$ in which $\diamond$ is defined by
    \begin{equation*}
        x\diamond y=x\succ y-y\prec x,\;\forall x,y\in A,
    \end{equation*}
which is called the {\bf sub-adjacent pre-Lie algebra of $(A,\succ,\prec)$}.
Moreover, both $(A,\circ)$ and $(A,\diamond)$ have the same sub-adjacent Lie algebra $(A,[-,-])$ in which $[-,-]$ is defined by
\begin{equation*}
    [x,y]=x\succ y+x\prec y-y\succ x-y\prec x,\;\;\forall x,y\in
    A.
\end{equation*}
\end{pro}
\begin{proof}
It can be obtained from a straightforward verification or as follows.
Since $\mathrm{id}$ is an $\mathcal{O}$-operator of $(A,\circ)$
associated to $(\mathcal{L}_{\succ},\mathcal{R}_{\prec}, A)$, it
is also an $\mathcal{O}$-operator of the Lie algebra $(A,[-,-])$ associated to
$(\mathcal{L}_{\succ}-\mathcal{R}_{\prec}, A)$. Therefore there is
a pre-Lie algebra $(A,\diamond)$ in which $\diamond$ is defined by
\begin{equation*}
    x\diamond y=(\mathcal{L}_{\succ}-\mathcal{R}_{\prec})\big(\mathrm{id}(x)\big)y=x\succ y-y\prec x,\;\forall x,y\in A.
\end{equation*}The other result is obvious.
\end{proof}


\begin{defi}\label{de:GPPA}\cite{RPA}
        A \textbf{relative pre-Poisson algebra} is a quadruple
        $(A,\star,\diamond,P)$, where $(A,\star)$ is a Zinbiel algebra,
        $(A,\diamond)$ is a pre-Lie algebra,  $P:A\rightarrow A$ is a
        derivation of both $(A,\star)$ and $(A,\diamond)$, that is,
        \begin{eqnarray}
        P(x\star y)&=&P(x)\star y+x\star P(y),\label{eq:GPPA1}\\
        P(x\diamond y)&=&P(x)\diamond y+x\diamond P(y),\;\;\forall x,y\in A, \label{eq:GPPA2}
        \end{eqnarray}
        and the following compatible conditions are satisfied:
        \begin{eqnarray}
        y\diamond(x\star z)-x\star(y\diamond z)+(x\diamond y-y\diamond x)\star z-(x\star P(y)+P(y)\star x)\star z&=&0,\label{eq:GPPA4}\\
            (x\star y+y\star x) \diamond z-x\star(y\diamond z)-y\star(x\diamond z)+(x\star y+y\star x)\star P(z)&=&0,\ \ \ \ \label{eq:GPPA3}
        \end{eqnarray}
for all $x,y,z\in A$.
\end{defi}

\begin{defi}
Let $(A,\star)$ be a Zinbiel algebra and  $(A,\cdot)$ be the
descendent commutative associative algebra. Let $(A,\succ,\prec)$
be a pre-APL algebra, $(A,\circ)$ be the sub-adjacent anti-pre-Lie
algebra, $(A,\diamond)$ be the sub-adjacent pre-Lie algebra and
$(A,[-,-])$ be their sub-adjacent Lie algebra. If there are linear
maps $P,Q:A\rightarrow A$ satisfying the following conditions:
\begin{enumerate}
    \item $(A,\star,\diamond,P)$ is a relative pre-Poisson algebra;
    \item \eqref{eq:pre-sys5}-\eqref{eq:pre-sys6} hold such that $(A,\star,P,Q)$ is an admissible differential Zinbiel algebra;
    \item the following equations hold:
    \begin{eqnarray}
        y\prec Q(x)-P(y)\prec x-Q(y\prec x)&=&0, \label{eq:pre-sys7}\\
        x\succ Q(y)-P(x)\succ y-Q(x\succ y)&=&0,  \label{eq:pre-sys8}\\
        (y\circ x)\star z-y\succ(x\star z)-x\star(y\diamond z)-\big(x\cdot P(y)\big)\star z&=&0,  \label{eq:pre-sys9}\\
        x\star(z\prec y)-z\prec(x\cdot y) +y\star(x\diamond z)-(x\cdot y)\star P(z)&=&0,  \label{eq:pre-sys10}\\
        x\star(y\succ z)-y\succ(x\star z)+[x,y]\star z-\big( x\cdot P(y) \big)\star z&=&0, \label{eq:pre-sys11}\\
        (y\star z)\prec x-y\succ(x\star z)-z\prec(x\cdot y)+Q\big( (x\cdot y)\star z\big)&=&0,  \label{eq:pre-sys12}\\
        (x\cdot y)\succ z-y\succ(x\star z)-x\succ(y\star z)+Q\big((x\cdot y)\star z\big)&=&0,\;\;\forall x,y,z\in A,  \label{eq:pre-sys13}
    \end{eqnarray}
\end{enumerate}
then we say $(A,\star,\succ,\prec,P,Q)$ is a {\bf relative pre-PCA algebra}.
\end{defi}

\delete{
\begin{defi}
    Let $(A,\star,P,Q)$ be an admissible differential Zinbiel algebra \cm{admissible is necessary here?} and $(A,\succ,\prec)$ be a pre-APL algebra.
    If the following equations hold:
    \begin{eqnarray}
    &&  P(x\succ y-y\prec x)=P(x)\succ y-y\prec P(x)+x\succ P(y)-P(y)\prec x, \label{eq:pre-sys4}\\
    &&  y\prec Q(x)-P(y)\prec x-Q(y\prec x)=0, \label{eq:pre-sys7}\\
    && x\succ Q(y)-P(x)\succ y-Q(x\succ y)=0,  \label{eq:pre-sys8}\\
    &&x\star(y\succ z-z\prec y)-[x,y]\star z-y\succ (x\star z)+(x\star z)\prec y+\big( x\cdot P(y)\big)\star z=0,\label{eq:pre-sys1}\\
    &&(x\cdot y)\succ z-z\prec(x\cdot y)-x\star(y\succ z-z\prec y)-y\star(x\succ z-z\prec x)+(x\cdot y)\star P(z)=0,\ \ \ \ \label{eq:pre-sys2}\\
    && (y\circ x)\star z-y\succ(x\star z)+x\star(z\prec y)-x\star(y\succ z)-\big(x\cdot P(y)\big)\star z=0,  \label{eq:pre-sys9}\\
    &&x\star(z\prec y)-z\prec(x\cdot y) +y\star(x\succ z)-y\star(z\prec x)-(x\cdot y)\star P(z)=0,  \label{eq:pre-sys10}\\
    &&  x\star(y\succ z)-y\succ(x\star z)+[x,y]\star z-\big( x\cdot P(y) \big)\star z=0, \label{eq:pre-sys11}\\
    && (y\star z)\prec x-y\succ(x\star z)-z\prec(x\cdot y)+Q\big( (x\cdot y)\star z\big)=0,  \label{eq:pre-sys12}\\
    && (x\cdot y)\succ z-y\succ(x\star z)-x\succ(y\star z)+Q\big((x\cdot y)\star z\big)=0,\;\forall x,y,z\in A,  \label{eq:pre-sys13}
    \end{eqnarray}
    then we say $(A,\star,\succ,\prec,P,Q)$ is a {\bf relative pre-PCA algebra}.
\end{defi}}

\delete{
\begin{defi}
    Let $(A,\star)$ be a Zinbiel algebra, $(A,\succ,\prec)$ be a pre-APL algebra and $P,Q:A\rightarrow A$ be linear maps.
    If the following equations hold:
    \begin{eqnarray}
    &&x\star(y\succ z-z\prec y)-[x,y]\star z-y\succ (x\star z)+(x\star z)\prec y+\big( x\cdot P(y)\big)\star z=0,\label{eq:pre-sys1}\\
    &&(x\cdot y)\succ z-z\prec(x\cdot y)-x\star(y\succ z-z\prec y)-y\star(x\succ z-z\prec x)+(x\cdot y)\star P(z)=0,\ \ \ \ \label{eq:pre-sys2}\\
    && P(x\star y)=P(x)\star y+x\star P(y),  \label{eq:pre-sys3}\\
    &&  P(x\succ y-y\prec x)=P(x)\succ y-y\prec P(x)+x\succ P(y)-P(y)\prec x, \label{eq:pre-sys4}\\
    && Q(x)\star y-x\star P(y)-Q(x\star y)=0,  \label{eq:pre-sys5}\\
    &&  x\star Q(y)-P(x)\star y-Q(x\star y)=0, \label{eq:pre-sys6}\\
    &&  y\prec Q(x)-P(y)\prec x-Q(y\prec x)=0, \label{eq:pre-sys7}\\
    && x\succ Q(y)-P(x)\succ y-Q(x\succ y)=0,  \label{eq:pre-sys8}\\
    && (y\circ x)\star z-y\succ(x\star z)+x\star(z\prec y)-x\star(y\succ z)-\big(x\cdot P(y)\big)\star z=0,  \label{eq:pre-sys9}\\
    &&x\star(z\prec y)-z\prec(x\cdot y) +y\star(x\succ z)-y\star(z\prec x)-(x\cdot y)\star P(z)=0,  \label{eq:pre-sys10}\\
    &&  x\star(y\succ z)-y\succ(x\star z)+[x,y]\star z-\big( x\cdot P(y) \big)\star z=0, \label{eq:pre-sys11}\\
    && (y\star z)\prec x-y\succ(x\star z)-z\prec(x\cdot y)+Q\big( (x\cdot y)\star z\big)=0,  \label{eq:pre-sys12}\\
    && (x\cdot y)\succ z-y\succ(x\star z)-x\succ(y\star z)+Q\big((x\cdot y)\star z\big)=0,\;\forall x,y,z\in A,  \label{eq:pre-sys13}
    \end{eqnarray}
    then we say $(A,\star,\succ,\prec,P,Q)$ is a {\bf relative pre-PCA algebra}.
\end{defi}}

\begin{pro}\label{pro:ADZA}
    Let $(A,\star,P,Q)$ be an admissible differential Zinbiel algebra.
    Define binary operations $\succ,\prec:A\otimes A\rightarrow A$ respectively by
    \begin{equation}\label{eq:pro:ADZA}
        x\succ y=Q(x\star y)-P(x)\star y,\; x\prec y=Q(y\star x)-y\star P(x),\;\forall x,y\in A.
    \end{equation}
    Then $(A,\star,\succ,\prec,P,Q)$ is a  relative pre-PCA algebra.
\end{pro}
\begin{proof}
Let $x,y,z\in A$. Then we have
\begin{eqnarray*}
&&x\star(y\diamond z)-[x,y]\star z-y\diamond (x\star z) +\big( x\cdot P(y)\big)\star z\\
&&=x\star \big(y\star P(z)-P(y)\star z\big)-\big( x\cdot P(y)-P(x)\cdot y \big)\star z\\
&&\ \ -Q\big( y\star(x\star z)  \big)+P(y)\star (x\star z)+Q\big( y\star(x\star z)  \big)-y\star P(x\star z)+\big(x\cdot P(y)\big)\star z\\
&&\overset{\eqref{eq:Zinbiel}}{=}x\star \big(y\star P(z)\big)+\big( P(x)\cdot y \big)\star z-y\star P(x\star z)\\
&&\overset{\eqref{eq:Zinbiel},\eqref{eq:pre-sys3}}{=}0.
\end{eqnarray*}
Hence \eqref{eq:GPPA4} holds. Similarly \eqref{eq:GPPA2} and \eqref{eq:GPPA3} hold, and thus $(A,\star,\diamond,P)$ is a relative pre-Poisson algebra.
Moreover, we have
\begin{eqnarray*}
&&(y\circ x)\star z-y\succ (x\star z)-x\star(y\diamond z)-\big(x\cdot P(y)\big)\star z\\
&&=Q(x\cdot y)\star z-\big(P(y)\cdot x\big)\star z-Q\big( y\star (x\star z)\big)+P(y)\star(x\star z)\\
&&\ \ -x\star\big(y\star P(z)\big)+x\star\big( P(y)\star z\big)-\big(x\cdot P(y)\big)\star z\\
&&\overset{\eqref{eq:pre-sys5},\eqref{eq:Zinbiel}}{=}(x\cdot y)\star P(z)-x\star\big(y\star P(z)\big)\\
&&\overset{\eqref{eq:Zinbiel}}{=}0.
\end{eqnarray*}
Hence \eqref{eq:pre-sys9} hold. Similarly \eqref{eq:pre-sys7},
\eqref{eq:pre-sys8} and \eqref{eq:pre-sys10}-\eqref{eq:pre-sys13}
hold. Therefore $(A,\star,\succ,\prec,P,Q)$ is a  relative pre-PCA
algebra.
\end{proof}
\begin{pro}\label{pro:oop}
    Let $(A,\star,\succ,\prec,P,Q)$ be a relative pre-PCA algebra.
    Then $(A,\cdot,\circ,P,Q)$ with the binary operations $\cdot$ and $\circ$ defined by \eqref{eq:ZintoAss} and \eqref{eq:pre-APL to APL} respectively is a relative PCA algebra, which is called the {\bf sub-adjacent relative PCA algebra of $(A,\star,\succ,\prec,P,Q)$}.
    Moreover, $(\mathcal{L}_{\star},\mathcal{L}_{\succ},\mathcal{R}_{\prec},P,Q,A)$ is a representation of $(A,\cdot,\circ,P,Q)$,
    and $\mathrm{id}$ is an $\mathcal{O}$-operator of $(A,\cdot,\circ,P,Q)$ associated to $(\mathcal{L}_{\star},\mathcal{L}_{\succ},\mathcal{R}_{\prec},P,Q,A)$.
\end{pro}
\begin{proof}
By \cite{RPA}, $(A,\cdot,[-,-],P)$ is a relative Poisson algebra.
For all $x,y,z\in A$, we have
\begin{eqnarray*}
&&x\cdot(y\circ z)-y\circ(x\cdot z)+[x,y]\cdot z-x\cdot P(y)\cdot z\\
&&=x\star(y\succ y)+x\star(y\prec z)+(y\succ z)\star x+(y\prec z)\star x\\
&&\ \ -y\succ(x\star z)-y\succ(z\star x)-y\prec(x\star z)-y\prec(z\star x)\\
&&\ \ +[x,y]\star z+z\star(x\succ y+x\prec y-y\succ x-y\prec x)\\
&&\ \ -\big(x\cdot P(y)\big)\star z-(x\cdot z)\star P(y)-\big(z\cdot P(y)\big)\star x\\
&&\overset{\eqref{eq:pre-sys9}-\eqref{eq:pre-sys11}}{=}0.
\end{eqnarray*}
Hence \eqref{eq:rps3} holds. Similarly \eqref{eq:rps1}, \eqref{eq:rps2} and \eqref{eq:rps4} hold, and hence $(A,\star,\succ,\prec,P,Q)$ is a relative PCA algebra.
The rest results are straightforward.
\end{proof}

\begin{pro}\label{pro:r,anti}
Let $(A,\star,\succ,\prec,P,Q)$ be a relative pre-PCA algebra and $(A,\cdot,\circ,P,Q)$ be the sub-adjacent relative PCA algebra.
Let $\{ e_{1},\cdots,e_{n} \}$ be a basis of $A$ and $\{ e^{*}_{1},\cdots,e^{*}_{n} \}$ be the dual basis.
Then
\begin{equation*}
    r=\sum_{i=1}^{n}( e^{*}_{i}\otimes e_{i}-e_{i}\otimes e^{*}_{i} )
\end{equation*}
is an antisymmetric solution of the RPCA-YBE in the relative PCA algebra
$(A\ltimes_{-\mathcal{L}_{\star}^{*},\mathcal{R}^{*}_{\prec}-\mathcal{L}^{*}_{\succ},\mathcal{R}^{*}_{\prec}}A^{*},P+Q^{*},Q+P^{*})$.
\end{pro}
\begin{proof}
By Proposition \ref{pro:oop}, $\mathrm{id}$ is an
$\mathcal{O}$-operator of $(A,\cdot,\circ,P,Q)$ associated to
$(\mathcal{L}_{\star},\mathcal{L}_{\succ},\mathcal{R}_{\prec}$,
$P$, $Q$, $A)$. Note that $\mathrm{id}$ is identified as
$\sum\limits_{i=1}^{n}e^{*}_{i}\otimes e_{i}$ in $A^{*}\otimes
A\subset
A\ltimes_{-\mathcal{L}_{\star}^{*},\mathcal{R}^{*}_{\prec}-\mathcal{L}^{*}_{\succ},\mathcal{R}^{*}_{\prec}}A^{*}\otimes
A\ltimes_{-\mathcal{L}_{\star}^{*},\mathcal{R}^{*}_{\prec}-\mathcal{L}^{*}_{\succ},\mathcal{R}^{*}_{\prec}}A^{*}$.
Hence the conclusion follows from Theorem \ref{thm:oop}.
\end{proof}

\begin{ex}
Let $(A,\star,P,-P)$ be an admissible differential Zinbiel algebra
given by Example \ref{ex:ADZA}. Then by Proposition
\ref{pro:ADZA}, there is a relative pre-PCA algebra
$(A,\star,\succ,\prec, P,-P)$ in which $\succ$ and $\prec$ are
defined by \eqref{eq:pro:ADZA}, whose nonzero products are given
by
\begin{equation*}
e_{1}\succ e_{1}=e_{1}\succ e_{2}=-4e_{3},\; e_{1}\prec e_{1}=e_{2}\prec e_{1}=-5e_{3}.
\end{equation*}
Let $\{e^{*}_{1},e^{*}_{2}, e^*_{3}\}$ be the dual basis. By
Proposition \ref{pro:oop}, there is a relative PCA algebra
$$(A\ltimes_{-\mathcal{L}_{\star}^{*},\mathcal{R}^{*}_{\prec}-\mathcal{L}^{*}_{\succ},\mathcal{R}^{*}_{\prec}}A^{*},P-P^{*},-P+P^{*})$$
whose nonzero products are given by
\begin{eqnarray*}
&&e_{1}\cdot e_{1}=2e_{3},\; e_{1}\cdot e_{2}=e_{3},\; e_{1}\circ e_{1}=-9e_{3},\; e_{1}\circ e_{2}=-4 e_{3},\; e_{2}\circ e_{1}=-5e_{3},\\
&&e_{1}\cdot e^{*}_{3}=e^{*}_{1}+e^{*}_{2},\; e_{1}\circ e^{*}_{3}=e^{*}_{1}+e^{*}_{2},\; e^{*}_{3}\circ e_{1}=5e^{*}_{1}+5e^{*}_{2}.
\end{eqnarray*}
By Proposition \ref{pro:r,anti},
\begin{equation*}
    r=\sum_{i=1}^{3}( e^{*}_{i}\otimes e_{i}-e_{i}\otimes e^{*}_{i} )
\end{equation*}
is an antisymmetric solution of the relative PCA-YBE in $(A\ltimes_{-\mathcal{L}_{\star}^{*},\mathcal{R}^{*}_{\prec}-\mathcal{L}^{*}_{\succ},\mathcal{R}^{*}_{\prec}}A^{*},P-P^{*},-P+P^{*})$.
Then by Corollary \ref{cor:cob}, there is a relative PCA bialgebra $(A\ltimes_{-\mathcal{L}_{\star}^{*},\mathcal{R}^{*}_{\prec}-\mathcal{L}^{*}_{\succ},\mathcal{R}^{*}_{\prec}}A^{*},\Delta,\theta,P-P^{*},-P+P^{*})$
with $\Delta$ and $\theta$ given by the following nonzero coproducts:
\begin{eqnarray*}
&& \Delta(e_{1})=\Delta(e_{2})=e_{1}\otimes e^{*}_{1}+e^{*}_{1}\otimes e_{1},\; \Delta(e^{*}_{3})=2e^{*}_{1}\otimes e^{*}_{1}+e^{*}_{1}\otimes e^{*}_{2}+e^{*}_{2}\otimes e^{*}_{1},\\
&& \theta(e_{1})=\theta(e_{2})=5e_{3}\otimes e^{*}_{1}+e^{*}_{1}\otimes e_{3},\;
\theta(e^{*}_{3})=-9e^{*}_{1}\otimes e^{*}_{1}-4e^{*}_{1}\otimes e^{*}_{2}-5e^{*}_{2}\otimes e^{*}_{1}.
\end{eqnarray*}
\end{ex}

\bigskip

 \noindent
 {\bf Acknowledgements.}  This work is supported by
NSFC (11931009, 12271265, 12261131498, 12326319, 12401031, W2412041), 
the Postdoctoral Fellowship Program of CPSF (GZC20240755, 2024T005TJ, 2024M761507),
the Fundamental
Research Funds for the Central Universities and Nankai Zhide
Foundation.
The authors thank the referees for  valuable suggestions to improve the paper.

\noindent
{\bf Declaration of interests. } The authors have no conflicts of interest to disclose.

\noindent
{\bf Data availability. } Data sharing is not applicable as no new data were created or analyzed.

\end{document}